\documentclass[12pt,letter]{amsart}
\usepackage{amsmath,amssymb,amsthm,amscd,amsxtra,amsfonts,mathrsfs,graphicx,enumerate}
\input xy
\xyoption{all}
\newtheorem{Theorem}{Theorem}[section]
\newtheorem{Lemma}[Theorem]{Lemma}
\newtheorem{Proposition}[Theorem]{Proposition}
\newtheorem{Standard Theorem}[Theorem]{Standard Theorem}
\newtheorem{Corollary}[Theorem]{Corollary}
\newtheorem{Example}[Theorem]{Example}
\newtheorem{Remark}[Theorem]{Remark}

\newtheorem{Definition}[Theorem]{Definition}

\newtheorem{Definition/Lemma}[Theorem]{Definition/Lemma}
\newtheorem{Conjecture}[Theorem]{Conjecture}
\newtheorem{Lemma/Definition}[Theorem]{Lemma/Definition}
\advance\evensidemargin-.5in
\advance\oddsidemargin-.5in
\advance\textwidth1in

\begin{document}
 \author{Charlie Beil}
 \thanks{The author was supported in part by the Simons Foundation, a DOE grant, IPMU, and the PFGW grant.}
 \address{Simons Center for Geometry and Physics, State University of New York, Stony Brook, NY 11794-3636, USA}
 \email{cbeil@scgp.stonybrook.edu}
 \title{The Geometry of Noncommutative Singularity Resolutions}
 \date{}

\begin{abstract}
We introduce a geometric realization of noncommutative singularity resolutions.  To do this, we first present a new conjectural method of obtaining conventional resolutions using coordinate rings of matrix-valued functions.  We verify this conjecture for all cyclic quotient surface singularities, the Kleinian $D_n$ and $E_6$ surface singularities, the conifold singularity, and a non-isolated singularity, using appropriate quiver algebras.  This conjecture provides a possible new generalization of the classical McKay correspondence.  Then, using symplectic reduction within these rings, we obtain new, non-conventional resolutions that are hidden if only commutative functions are considered.  Geometrically, these non-conventional resolutions result from shrinking exceptional loci to ramified (non-Azumaya) point-like spheres.
\end{abstract}
\maketitle
\tableofcontents

\section{Motivation: a geometric perspective}

We aim to make progress towards answering the following question. 
\begin{quote}
Given a variety X with mild singularities, find a coordinate ring of matrix-valued functions on X that ``sees'' appropriate conventional resolutions in a new way. Using these matrix-valued functions, obtain new, non-conventional resolutions that are hidden if only commutative functions on X are considered.
\end{quote}

The rings of matrix-valued functions that we will consider are quiver algebras.  Stated concisely, a quiver algebra is a quotient of an algebra whose basis consists of all paths in a quiver (that is, directed graph), and the product of two paths is their concatenation if defined and zero otherwise.  A representation of (or module over) a quiver algebra is obtained by associating a vector space to each vertex of the quiver, representing each arrow by a linear map from the vector space at its tail to the vector space at its head, and requiring these linear maps satisfy the relations of the algebra.
  
To motivate our approach to geometry, let $R$ be a commutative noetherian $\mathbb{C}$-algebra.  The points $m$ of the affine variety $X = \operatorname{Max}R$ may always be identified with the simple modules $R/m \cong \mathbb{C}$ over the ring of polynomial functions $R$ on $X$, and a point $m$ in $X$ is smooth (singular) if and only if the projective dimension of the corresponding simple module $R/m$ equals the complex topological dimension of $X$ at $m$, 
$$\operatorname{pd}_R(R/m) = \operatorname{dim}\left(R_m\right)$$
(resp.\ is infinite).  It is therefore natural to extend this idea to noncommutative coordinate rings: if a f.g.\ noncommutative $\mathbb{C}$-algebra $A$ is a finitely generated module over its center $Z$ (or ``module-finite over its center''), then we deem a point $p \in \operatorname{Max}A$ (equivalently, simple $A$-module $V$ whose annihilator is $p$ \cite[Corollary 4.2.3]{Smith}) smooth if its projective dimension equals the topological dimension at $p \cap Z \in \operatorname{Max}Z$,
\begin{equation} \label{proj dim = krull dim}
\operatorname{pd}_A(V) = \operatorname{dim}\left( Z_{p\cap Z} \right).
\end{equation}
Moreover, in the commutative case the evaluation of a function $f = f(x) \in R$ at the point $m = (x-a)$ is the corresponding representation of $f$, namely $f(a)=[f] \in R/m$, so we say the evaluation of a function $f \in A$ at the point $p$ is the representation of $f$ corresponding to $V$, and thus in general $f$ will be a matrix-valued function.

The algebras $A$ and $Z$ are both noetherian by the Artin-Tate lemma \cite[Theorem 4.2.1]{Smith}, and $\operatorname{Max}A$ admits the Zariski topology with closed sets 
$$V(I) := \left\{ p \in \operatorname{Max}A \ | \ I \subseteq p \right\}$$
with $I$ any ideal (since maximal ideals are prime
).  If in addition $A$ is prime then the map $\phi: \operatorname{Max}A \rightarrow \operatorname{Max}Z$ given by $p \mapsto p \cap Z$ is bijective and continuous over an open dense subset of $\operatorname{Max}Z$ called the Azumaya locus of $A$ \cite[Theorem 4.2.7]{Smith}, so $\operatorname{Max}A$ and $\operatorname{Max}Z$ may be regarded in some sense as birationally equivalent.  We therefore call the map $\phi$ a noncommutative resolution of $Z$ if $A$ is smooth in the sense that (\ref{proj dim = krull dim}) holds for each $p \in \operatorname{Max}A$.  Such resolutions were first proposed by the physicists Berenstein, Douglas, and Leigh \cite{Beren, BD, BL} in the context of string theory (see also \cite{DGM}), and formalized independently and more abstractly by Van den Bergh in his definition of a noncommutative crepant resolution \cite[Definition 4.1]{VdB}.  In Van den Bergh's approach, birationality is extended to the noncommutative setting by replacing isomorphic function fields \cite[Corollary 4.5]{H} with Morita equivalent ``noncommutative function fields'' (see for example \cite[section 5.2]{B}). 

We will propose a program to unify, in a geometric sense, the commutative resolutions of a singularity with its noncommutative resolutions.  In so doing we will present a new conjectural method of obtaining commutative resolutions from a noncommutative coordinate ring in section \ref{Almost large modules}.  Using this ring we will then introduce, in section \ref{Shrinking families of almost large modules}, a way of shrinking the irreducible components of the exceptional locus to smooth point-like spheres, where many such spheres may occupy the same point in space.  From this we obtain new resolutions, unseen by the commutative functions, that are (possibly proper) subsets of the maximal ideal spectra of the noncommutative coordinate ring.  The conjecture will be verified for a number of examples in section \ref{Res sing}, including at least one where the singularity is not two-dimensional; not a quotient by a finite group; non-Gorenstein; non-toric; non-isolated.

It would be interesting to understand how our construction is related to Van den Bergh's construction, where a commutative resolution of $\operatorname{Spec}R$ is obtained from a noncommutative $R$-algebra $A$ as an open subset of the fine moduli space $\mathcal{M}^{\theta}_d(A)$ of stable $A$-modules with a fixed dimension vector $d \in \mathbb{Z}_{\geq 0}^{|Q_0|}$ and generic stability parameter $\theta \in \mathbb{Z}^{|Q_0|}$ \cite[Theorem 6.3.1]{VdB}, which is based on the methods of \cite{BKR}.\\
\\
\textbf{Conventions.}  $A$ denotes a finitely generated ( = f.g.) algebra (usually over $\mathbb{C}$).  All modules are left modules, and all representations are complex unless specified otherwise.  The $A$-module $V$ corresponding to a representation $\rho: A \rightarrow \operatorname{End}_{\mathbb{C}}(V)$ is the module defined by $av := \rho(a)v$ for $a \in A$, $v \in V$.  A module isoclass will often be referred to as just a module.  Multiplication of paths in a quiver algebra is read right to left, following the composition of maps.  $Q_{\ell}$ denotes the set of paths of length $\ell$ in a quiver $Q$, and $Q_{\geq 0}$ denotes the set of all paths in $Q$.  Given a quiver algebra $A = \mathbb{C}Q/I$ and vertex $i \in Q_0$, we denote by $S_i$ the ``vertex simple module'' corresponding to the representation of $A$ with a single 1-dimensional vector space at vertex $i$, and with all arrows represented by zero.\\
\\
\textbf{Acknowledgements.} I would like to give special thanks to David Morrison and David Berenstein for invaluable discussions and support.  
I would also like to thank Alastair King, Bal\'azs Sendr\"oi, Leonard Wesley, and Tea Rose for their encouragement.  I would like to thank IPMU for their hospitality and financial support while a portion of this work was completed, as well as the Simons Workshop in Mathematics and Physics 2010.  Work supported in part by the Simons Foundation, the Department of Energy DE-FG02-91ER40618, and the PFGW grant.

\section{Almost large modules}

\subsection{Definition and conjecture} \label{Almost large modules}

We call a simple module (and its corresponding representation) \textit{large} if it is of maximal $\mathbb{C}$-dimension\footnote{When $A$ is module-finite over its center, such modules are also tiny \cite[Theorem 4.2.2]{Smith}!} 
\begin{equation} \label{d}
d = \operatorname{max}\left\{ \operatorname{dim}_{\mathbb{C}} V \ | \ V \text{ a simple $A$-module} \right\}.
\end{equation}
If $A$ is a f.g.\ $\mathbb{C}$-algebra, module-finite over its center $Z$, then $d < \infty$ \cite[Theorem 4.2.2]{Smith}.  If $A$ is also prime then the Azumaya locus of $A$ is the open dense set of points $m \in \operatorname{Max}Z$ such that $A/Am \cong \operatorname{Mat}_d(\mathbb{C})$ (characterizing the ``noncommutative residue fields'' of $A$).  Furthermore, there is a bijection between $Am \in \operatorname{Max}A$ and the large modules $V$, given by $Am = \operatorname{ann}_AV$ \cite[Theorem 4.2.7]{Smith}, 
and so the large modules are parameterized by the Azumaya locus.  Under suitable conditions the Azumaya locus coincides with the smooth locus of $Z$, a fact first discovered by Le Bruyn when the algebra is graded \cite[Theorem 1]{Le Bruyn}, and by Brown and Goodearl when the algebra is not graded \cite[section 3]{BGood}.

\begin{Theorem} \label{LBG}
(Le Bruyn, Brown-Goodearl \cite[Theorem 3.8]{BGood}.) If an algebra is prime, noetherian, Auslander-regular, Cohen-Macaulay, and module-finite over its center $Z$, and if the compliment of the Azumaya locus has codimension at least 2 in $\operatorname{Max}Z$, then the Azumaya and smooth loci coincide.\footnote{A ring $S$ is \textit{Auslander-regular} if $S$ has finite global dimension and satisfies the Auslander condition, namely, that if $p<q$ are non-negative integers and $M$ is a finitely generated $R$-module, then $\operatorname{Ext}^p_S(N,S) = 0$ for every submodule $N$ of $\operatorname{Ext}_S^q(M,S)$.  $S$ is \textit{Cohen-Macaulay} if it has finite Gelfand-Kirillov dimension $\operatorname{GKdim}(S) < \infty$ and 
$$\operatorname{min}\left\{ r \ | \ \operatorname{Ext}^r_S(M,S) \not = 0 \right\} +\operatorname{GKdim}(M) = \operatorname{GKdim}(S)$$ for every finitely generated $S$-module $M$.}
\end{Theorem}

We introduce the following definitions in hopes of extending this theorem to smooth resolutions of the center of $A$ when $A$ is an infinite dimensional basic algebra, module-finite over its center.  

Recall that two idempotents $e_i$ and $e_j$ are orthogonal if $e_ie_j = e_je_i = \delta_{ij}e_i$; an idempotent is primitive if it cannot be expressed as the sum of two nontrivial orthogonal idempotents; and a set of idempotents is complete if their sum is $1 \in A$.  If $\{e_1,\ldots, e_n\}$ is a complete set of primitive orthogonal idempotents then $A$ decomposes into a direct sum of indecomposable $A$-modules $A = Ae_1 \oplus \cdots \oplus Ae_n$, which is unique up to isomorphism and permutation of the factors since each $Ae_i$ is projective \cite[Corollary 20.23]{Lam}.  A subset $\{e_{i_1}, \ldots, e_{i_m}\}$ of $\{e_1, \ldots, e_n\}$ is basic if $Ae_{i_1}, \ldots, Ae_{i_m}$ is a complete, non-redundant set of representatives of $A$-modules of the form $Ae$ for some primitive idempotent $e$, and $A$ is basic if $\{e_{i_1},\ldots,e_{i_m}\} = \{e_1, \ldots,e_n\}$.  Finally, if $A$ is a basic $k$-algebra and $d \in (\mathbb{Z}_{\geq 0})^n$, then we denote by $\operatorname{Rep}_dA$ the set of $A$-modules $V$ with dimension vector $d = \left( \operatorname{dim}_k(e_iV) \right)$.

We introduce the following definition in order to capture the notion of a path in a quiver algebra without having to refer to one specific basis.

\begin{Definition} \rm{
We say a subset $\mathcal{P}$ of a basic $k$-algebra $A$ is a \textit{path-like set} if $\mathcal{P} \setminus \{0\}$ is a $k$-basis for $A$, $\mathcal{P}$ contains a basic set of idempotents, and $a,b \in \mathcal{P}$ implies $ab \in \mathcal{P}$.
} \end{Definition}

\begin{Remark} \rm{
If $A = \mathbb{C}Q/I$ is a quiver algebra with vertex set $Q_0 = \{1,2,\ldots, n\}$ and $a \in e_2Q_1e_1$, then the set $\{e_1+a,e_2-a,e_3, \ldots, e_n\}$ is a complete set of primitive orthogonal idempotents in $A$ different from the vertex idempotents.  Note that $e_1+a$ and $e_2-a$ are primitive since there are $A$-module isomorphisms $A(e_1+a) \cong Ae_1$ and $A(e_2-a) \cong Ae_2$, and $Ae_1$ and $Ae_2$ are indecomposable.\footnote{There are $A$-module monomorphisms $A(e_1+a) \stackrel{\operatorname{id}}{\longrightarrow} Ae_1$; $Ae_1 \stackrel{\operatorname{id}}{\longrightarrow} A(e_1+a)$; $A(e_2-a) \stackrel{ \cdot e_2}{\longrightarrow}Ae_2$; and $Ae_2 \stackrel{\cdot (e_2-a)}{\longrightarrow}A(e_2-a)$.}
} \end{Remark}  

Recall that in a noetherian integral domain $R$, the codimension of a prime ideal $p$ is the length $\ell$ of a maximal chain $p_0 \subsetneq p_1 \subsetneq \cdots \subsetneq p_{\ell} = p$ of distinct prime ideals, and $\ell$ equals the codimension of the subvariety defined by $p$ in $\operatorname{Max}R$.

\begin{Definition} \label{almost large} \rm{
Let $A$ be a f.g.\ basic algebra, module-finite over its prime center $Z$.  Suppose $d$ is the dimension vector of a large $A$-module.  For $1 \leq \ell \leq \operatorname{dim}Z$, we say a subset $P$ of $A$ has \textit{codimension $\ell$} if there is a path-like set $\mathcal{P}$ of $A$ and a maximal chain of subsets
\begin{equation} \label{max chain}
0 \subsetneq P_1 \subsetneq \cdots \subsetneq P_{\ell} = P
\end{equation}
such that each $P_j$ is the $\mathcal{P}$-annihilator of a module in $\operatorname{Rep}_dA$.  If $V \in \operatorname{Rep}_dA$ is non-simple and satisfies $\operatorname{ann}_{\mathcal{P}}V = P$ then we say $V$ is an \textit{almost large $A$-module}.
}\end{Definition}

Note that $P$ is a multiplicatively closed subset of $A$.  Also, if $d \not = (1, \ldots, 1)$ then the ideal generated by $P$ will in general not be prime.  We will call $V$ an $\ell_{\mathcal{P}} = \ell$ almost large module.

Recall that the top $\operatorname{Top}V$ of a module $V$ is the largest semisimple quotient of $V$, while the socle $\operatorname{Soc}V$ (``bottom'') is the largest semisimple submodule of $V$.  
If $A$ is module-finite over its noetherian center $Z$, then we say $A$ is \textit{homologically smooth} if (\ref{proj dim = krull dim}) holds for each $p \in \operatorname{Max}A$.

\begin{Conjecture} \label{conjecture}
Let $A$ be as in Definition \ref{almost large}, and in addition homologically smooth with a singular center $Z$.  Suppose a primitive idempotent $e \in A$ satisfies
\begin{equation} \label{e in L}
\operatorname{max}\left\{ \operatorname{dim}_{\mathbb{C}}(eW) \ | \ W \text{ a large $A$-module} \right\} = 1.
\end{equation}
If the large $A$-module isoclasses are parameterized by the smooth locus of $\operatorname{Max}Z$ then the following hold:
\begin{enumerate}
 \item The isoclasses of almost large $A$-modules $V$, with $\operatorname{Soc}V = e\operatorname{Soc}V$, are parameterized by the exceptional locus $E$ of a smooth resolution $Y \rightarrow \operatorname{Max}Z$.
 \item For any fixed path-like set $\mathcal{P}$ of $A$, there is a natural bijection between the irreducible components $E_i$ of $E$ and the distinct subsets $P$ with the properties that $P$ is the $\mathcal{P}$-annihilator of an almost large module $V$ with $\operatorname{Soc}V = e \operatorname{Soc}V$, and if $P = P_{\ell}$ occurs in a maximal chain (\ref{max chain}) then the proceeding term $P_{\ell -1}$ is the $\mathcal{P}$-annihilator of a large module.
 \item If there exists a sequence of $\mathcal{P}$-annihilators
$$0 \subsetneq P_1 \subsetneq \cdots \subsetneq P_j \subsetneq \cdots \subsetneq P_{\ell},$$
where $P_j$ corresponds to the irreducible component $E_i$ by the natural bijection, then the isoclasses of almost large modules $V$, with $\operatorname{Soc}V = e \operatorname{Soc}V$ and $\mathcal{P}$-annihilator $P_{\ell}$, are parameterized by a codimension $\ell$ (in $Y$) quasi-projective subvariety of $E_i$.
\end{enumerate}
\end{Conjecture}

We will verify this conjecture for a number of examples in section \ref{Res sing}.  The underlying idea is then
\begin{center} \framebox{ \begin{tabular}{rcl}
\begin{tabular}{r} smooth locus of an affine variety \end{tabular} & $\longleftrightarrow$ & \begin{tabular}{l} large module isoclasses \end{tabular}\\ \\
\begin{tabular}{r} exceptional locus of a \\ smooth resolution \end{tabular} & $\longleftrightarrow$ & \begin{tabular}{l} almost large module isoclasses \\ with isomorphic 1-dim'l socles \end{tabular}\\ \\
\begin{tabular}{r} exceptional locus shrunk \\ to zero size \end{tabular} & $\longleftrightarrow$ & \begin{tabular}{l} tops of these \\ almost large module isoclasses \end{tabular}
\end{tabular}}
\end{center}
where the correspondence is given by parameterization.  The last item will be introduced in the next section.  The guiding principle is that if $V$ and $W$ are two non-isomorphic large modules and the points $\operatorname{ann}_ZV$ and $\operatorname{ann}_Z W$ lie on the same line that passes through a singular point of $\operatorname{Max}Z$, then $V$ and $W$ become isomorphic, and hence $\operatorname{ann}_Z V$ and $\operatorname{ann}_Z W$ become identified, when a minimal number of elements in $A$ are set equal to zero.  

\begin{Remark} \rm{
We will only verify (2) in Conjecture \ref{conjecture} for the path-like set $\mathcal{P} = Q_{\geq 0} \cup \{0\}$, though it will easily follow for any path-like set containing the vertex idempotents, since such a set is multiplicatively generated by the vertex idempotents and a basis for $\mathbb{C}Q_1$ consisting of elements of the form $\sum_{a \in e_jQ_1e_i}\gamma_a a$, with $\gamma_a \in \mathbb{C}$, $i,j \in Q_0$.
} \end{Remark} 

\begin{Remark} \rm{
In physics terms, a path-like set $\mathcal{P}$ may be viewed as the set of dibaryon operators in a quiver gauge theory, and the $\mathcal{P}$-annihilator of a point in the vacuum moduli space would then be the set of all dibaryons with zero vev at that point (in some sense, since a non-cyclic path will not be gauge invariant, and vev's are gauge invariant).
}\end{Remark}

\begin{Remark} \rm{
Let $A = \mathbb{C}Q/I$ be a quiver algebra satisfying the hypothesis of Conjecture \ref{conjecture}, and let $i \in Q_0$ be such that $e_i$ satisfies (\ref{e in L}).  We ask the question: does the set of almost large $A$-modules whose socles are isomorphic to the vertex simple $S_i$ always equal the entire set of non-simple modules whose socles are isomorphic to $S_i$ and whose dimension vector $d$ equals that of a large module?  Similarly, if a resolution of the center of $A$ is an open subset of the $\theta$-stable moduli space $\mathcal{M}_d^{\theta}(A)$ with generic stability parameter $\theta = \left(-1 + \sum_{j \in Q_0}d_j, -1 , \ldots, -1 \right) \in \mathbb{Z}^{|Q_0|}$, where the first component is $\theta_i$, then is the resolution necessarily the entire moduli space?
}\end{Remark}

\subsection{Shrinking families of almost large modules} \label{Shrinking families of almost large modules}
In most cases we consider, isoclasses of almost large modules are parameterized by collections of $\mathbb{P}^n$'s.  To make precise the notion of a $\mathbb{P}^n$-family of module isoclasses, we introduce the following definition; note the similarity with Definition \ref{impression} given below.

\begin{Definition} \label{P^n family} \rm{
Let $A$ be a $\mathbb{C}$-algebra, set $\mathbb{C}[t] := \mathbb{C}[t_1, \ldots, t_{n+1}]$, and suppose that there exists an algebra monomorphism
\begin{equation}
\sigma: A \longrightarrow \operatorname{End}_{\mathbb{C}[t]}(\mathbb{C}[t]^{\oplus d}).
\label{sigma''}
\end{equation}
Then for each $z \in \mathbb{C}^{n+1}$ the composition of $\sigma$ with the evaluation map at $z$,
$$A \stackrel{\sigma}{\longrightarrow} \operatorname{End}_{\mathbb{C}[t]}(\mathbb{C}[t]^{\oplus d}) \stackrel{\epsilon_{z}}{\longrightarrow} \operatorname{End}_{\mathbb{C}[t]}((\mathbb{C}[t]/(t-z))^{\oplus d}) \cong \operatorname{End}_{\mathbb{C}}(\mathbb{C}^{d}),$$
is a representation of $A$, and $V_z:= \mathbb{C}^d$ is an $A$-module with $av:= \epsilon_z \sigma(a)v$.  We say that the set of module isoclasses 
$$\left\{ \left[ V_{z} \right] \ | \ z \in \mathbb{C}^{n+1} \setminus 0 \right\}$$ 
is a $\mathbb{P}^n$\textit{-family} if it has the property that $V_z \cong V_{z'}$ if and only if there exists a $\lambda \in \mathbb{C}^*$ such that $(z'_1, \ldots, z'_{n+1}) = (\lambda z_1, \ldots, \lambda z_{n+1})$.
} \end{Definition}  

In section \ref{A first example} we will recall how $|\lambda|$ may be realized as the radius of $\mathbb{P}^n$ when viewed as an $n$-dimensional sphere using symplectic geometry.  Let $A = kQ/I$ be a quiver algebra admitting a $\mathbb{P}^n$-family $\left\{ \left[ V_z \right] \right\}$ of $A$-modules.  For $i \in Q_0$ set $d_i := \operatorname{dim}_{\mathbb{C}}(e_iV_x)$ and $d := \sum_i d_i$.  Denote by $\lambda$ an indeterminate and $\lambda_*$ an arbitrary element of $\mathbb{C}^*$.  
Let $V_t := \mathbb{C}[t]^{\oplus d}$ be the $A$-module defined by $av := \sigma(a)v$.  Suppose there exists an isomorphism 
\begin{equation} \label{phi}
\phi_{\lambda}: V_t \stackrel{\cong}{\longrightarrow} V_{\lambda t}
\end{equation}
where
$$\phi_{\lambda} \in \bigoplus_{i \in Q_0} \operatorname{GL}_{d_i}\left( \mathbb{C}(\lambda) \right).$$
Then for each $z \in \mathbb{C}^{n+1}\setminus 0$ and $\lambda_* \in \mathbb{C}^*$ there is an isomorphism 
$$\phi_{\lambda_*}: V_z \stackrel{\cong}{\longrightarrow} V_{\lambda_* z}.$$
For each $i \in Q_0$ we will denote by $\phi_{\lambda,i}$ the restriction of $\phi_{\lambda}$ to the factor $\operatorname{GL}_{d_i}(\mathbb{C})$.

Suppose the least power of $\lambda$ that appears in all the matrix entries of $\phi_{\lambda}$ is $m \in \mathbb{Z}$.  Since there is a trivial diagonal $\mathbb{C}^*$-action on the isomorphism parameters, there is also an isomorphism $\lambda_*^{-m}\phi_{\lambda_*}: V_z \stackrel{\cong}{\longrightarrow} V_{\lambda_* z}$.  With this choice of rescaling, the limit
$$\phi_0:= \operatorname{lim}_{\lambda \rightarrow 0} \lambda^{-m} \phi_{\lambda} \in \bigoplus_{i \in Q_0} \operatorname{Mat}_{d_i}\left( \mathbb{C} \right)$$
is nonzero and finite.  We will write $\phi_{\lambda}$ as $\phi^z_{\lambda}$ when we need to specify the module $V_z$ on which $\phi_{\lambda}$ is acting.

\begin{Lemma} \label{V_0}
$V_z/\operatorname{ker} \phi^z_0 \cong V_{z'}/\operatorname{ker}\phi^{z'}_0$ for each $z, z' \in \mathbb{C}^{n+1}\setminus 0$.
\end{Lemma}

\begin{proof}
Let $\sigma_{\lambda t}: A \rightarrow \operatorname{End}_{\mathbb{C}[t]}(\mathbb{C}[t]^{\oplus d})$ be the $\mathbb{C}[t]$-representation corresponding to $V_{\lambda t}$, so in particular $\sigma_t := \sigma$, and without loss of generality suppose the least power of $\lambda$ that appears in the matrix entries of $\phi_{\lambda}$ is zero.  For each arrow $a \in Q_1$, each $t_i$ that appears in the matrix entries of $\sigma(a) = \sigma_{t}(a)$ is mapped to $\lambda t_i$ in the matrix $\sigma_{\lambda t}(a)$ under the transformation given by
$$\phi_{\lambda, \operatorname{h}(a)} \sigma_t(a) = \sigma_{\lambda t}(a) \phi_{\lambda,\operatorname{t}(a)}.$$
In particular $t_i$ is mapped to $0$ in the matrix $\sigma_{0t}(a)$ under the transformation given by
$$\phi_{0,\operatorname{h}(a)} \sigma_t(a) = \sigma_{0t}(a) \phi_{0, \operatorname{t}(a)},$$
so $\sigma_0(a) = \sigma_{0t}(a)$ does not depend on the $t_i$, and thus the matrix $\epsilon_z \sigma_0(a)$ does not depend on the choice of $z$.  Now $a$ acts on $V_{0z}$ by $\epsilon_z \sigma_0(a)$, so $V_{0z} = V_{0z'}$ for each $z,z' \in \mathbb{C}^{n+1}\setminus 0$, and under this identification, $\operatorname{im}\phi_0^z = \operatorname{im}\phi_0^{z'}$.

The module epimorphisms
$$\phi^z_0: V_z \rightarrow \operatorname{im}\phi_0^z \ \ \ \text{ and } \ \ \ \phi^{z'}_0: V_{z'} \rightarrow \operatorname{im}\phi_0^{z'}$$
then imply $V_z/\operatorname{ker}\phi^z_0 \cong \operatorname{im}\phi_0^z = \operatorname{im}\phi_0^{z'} \cong V_{z'}/\operatorname{ker}\phi^{z'}_0$.
\end{proof}

Set $V_0 := V_z/\operatorname{ker}\phi^z_0$.  By Lemma \ref{V_0}, $V_0$ does not depend on the choice of $z \in \mathbb{C}^{n+1}\setminus 0$ up to isomorphism.

\begin{Lemma} \label{1-dim socle}
If there is a $z \in \mathbb{C}^{n+1}\setminus 0$ such that the socle of $V_z$ is 1-dimensional, then $V_0$ does not depend on the choice of $\phi_{\lambda}$ satisfying (\ref{phi}).
\end{Lemma}

\begin{proof}
Let $z \in \mathbb{C}^{n+1}\setminus 0$ be such that $\operatorname{Soc}V_z$ is 1-dimensional, say at $0 \in Q_0$.  Since $z$ is fixed we will write $\operatorname{ker}\phi^z_0$ as just $\operatorname{ker}\phi_0$.

Let $\phi_{\lambda}$ and $\phi'_{\lambda}$ be two isomorphisms $V_t \stackrel{\cong}{\longrightarrow} V_{\lambda t}$, so they are also isomorphisms $V_z \stackrel{\cong}{\longrightarrow} V_{\lambda z}$.  We claim that $\operatorname{ker} \phi_0 = \operatorname{ker} \phi'_0 \subset V_z$.  Denote by $\rho$ and $\rho_{\lambda_*}$ the representations $A \rightarrow \operatorname{Mat}_{d}(\mathbb{C})$ corresponding to $V_z$ and $V_{\lambda_* z}$ respectively.  

Fix $i \in Q_0$.  Then for each path $p \in e_0Q_{\geq 0}e_i$,
\begin{equation} \label{cc'}
c \rho(p) \phi_{\lambda, i}^{-1} = \phi_{\lambda, 0} \rho(p) \phi_{\lambda,i}^{-1} = \rho_{\lambda}(p) = \phi'_{\lambda,0} \rho(p) \phi_{\lambda, i}^{'-1} = c' \rho(p) \phi_{\lambda, i}^{'-1},
\end{equation}
where $\phi_{\lambda,0} = c \in \mathbb{C}^*$ and $\phi'_{\lambda,0} = c' \in \mathbb{C}^*$.  Choose $d_i = \operatorname{dim}_{\mathbb{C}} e_iV_z$ paths $\left\{ p_1, \ldots, p_{d_i} \right\}$ from $i$ to $0 \in Q_0$ inductively as follows.  Choose $v_1 \in e_i V_z \cong \mathbb{C}^{d_i}$.  Since $\operatorname{Soc}V_z \cong \mathbb{C}$ is at $0 \in Q_0$, there exists a path $p_1 \in e_0Q_{\geq 0}e_i$ such that $\rho(p_1)v_1 \not = 0$.  Now suppose the paths $\left\{ p_1, \ldots, p_{j-1} \right\}$ have been chosen.  Choose $v_j \in \operatorname{ker}\rho(p_1) \cap \ldots \cap \operatorname{ker} \rho(p_{j-1}) \cap e_iV_z$.  Again since $\operatorname{Soc}V_z \cong \mathbb{C}$ is at $0 \in Q_0$ there exists a path $p_j \in e_0Q_{\geq 0}e_i$ such that $\rho(p_j)v_j \not = 0$.  View each $\rho(p_k)$ as an element of $\operatorname{Mat}_{1 \times d_i}(\mathbb{C})$ and recall $\phi_{\lambda,i} \in \operatorname{Mat}_{d_i \times d_i}(\mathbb{C})$.  Then 
$$\operatorname{dim \ ker} \left[ \begin{array}{c} \rho(p_1) \\ \vdots \\ \rho(p_{j-1}) \\ \rho(p_j) \end{array} \right] < \operatorname{dim \ ker} \left[ \begin{array}{c} \rho(p_1) \\ \vdots \\ \rho(p_{j-1}) \end{array} \right] < 
\operatorname{dim \ ker} \left[ \begin{array}{c} \rho(p_1) \end{array} \right] = d_i -1.$$
Thus setting
$$B := \left[ \begin{array}{c} \rho(p_1) \\ \rho(p_2) \\ \vdots \\ \rho(p_{d_i}) \end{array} \right] \in \operatorname{Mat}_{d_i \times d_i}(\mathbb{C})$$
we have $\operatorname{dim \ ker}B = 0$ so $B$ is injective.  But from (\ref{cc'}),
$$B \phi_{\lambda,i}^{-1} \phi'_{\lambda,i} = c^{-1}c'B,$$
and since $B$ is injective $\phi_{\lambda,i}^{-1} \phi'_{\lambda,i} = c^{-1}c'\textbf{1}_{d_i}$, so $\phi_{\lambda,i} = cc'^{-1} \phi'_{\lambda,i}$, so $w \in \operatorname{ker} \phi_{0,i} \cap e_iV_z$ if and only if $w \in \operatorname{ker}\phi'_{0,i} \cap e_iV_z$, and thus $\operatorname{ker}\phi_0 = \operatorname{ker}\phi'_0$, proving our claim.

It follows that $V_z/\operatorname{ker}\phi_0 = V_z/\operatorname{ker}\phi'_0$, and so by Lemma \ref{V_0},
$$V_0(\phi_{\lambda}) \cong V_z/\operatorname{ker}\phi_0 = V_z/\operatorname{ker}\phi'_0 \cong V_0(\phi'_{\lambda}).$$
\end{proof}

\begin{Definition} \label{shrinktozerosize} \rm{
Suppose $A$ is module-finite over its noetherian center $Z$, and let $\left\{ [V_z] \right\}$ be a $\mathbb{P}^n$-family where each member has a 1-dimensional socle.  If $V_0 = \bigoplus W_i$ is semisimple with simple summands $W_i$ then we say that the $\mathbb{P}^n$ parameterizing this family shrinks to the points $\operatorname{ann}_A W_i \in \operatorname{Max}A$, and sits over the points $\operatorname{ann}_Z W_i \in \operatorname{Max}Z$.
} \end{Definition}

\begin{Remark} \rm{
In all the examples we will encounter, $V_0$ is the top of each member of its corresponding $\mathbb{P}^n$-family, though in general $V_0$ need not be semisimple.
} \end{Remark}

\begin{figure}
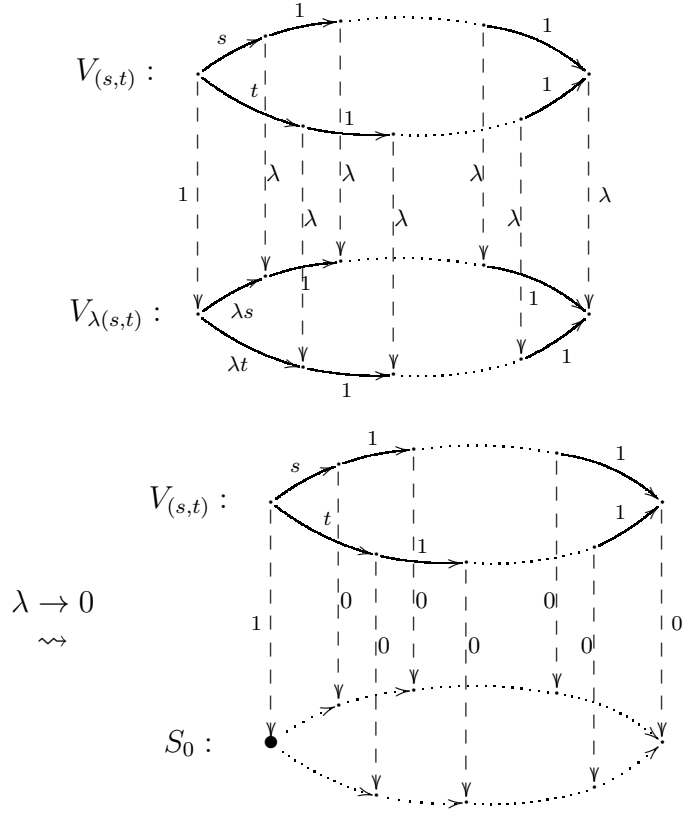

$$\xy
(-26,16)*{\cdot}="1";(-17,21)*{\cdot}="2";(-7,23)*{\cdot}="3";(12,22.5)*{\cdot}="4";(26,16)*{\cdot}="5";(17,10)*{\cdot}="6";(0,8)*{\cdot}="7";(-12,9)*{\cdot}="8";
(-26,-16)*{\cdot}="1'";(-17,-11)*{\cdot}="2'";(-7,-9)*{\cdot}="3'";(12,-9.5)*{\cdot}="4'";(26,-16)*{\cdot}="5'";(17,-22)*{\cdot}="6'";(0,-24)*{\cdot}="7'";(-12,-23)*{\cdot}="8'";
{\ar@/^.1pc/^s"1";"2"};{\ar@/^.1pc/^1"2";"3"};{\ar@{..}@/^.2pc/"3";"4"};{\ar@/^.3pc/^1"4";"5"};
{\ar@/_.1pc/^1"6";"5"};{\ar@{..}@/_.2pc/"7";"6"};{\ar@/_.1pc/^1"8";"7"};{\ar@/_.2pc/^{t}"1";"8"};
{\ar@/^.1pc/_{\lambda s}"1'";"2'"};{\ar@/^.1pc/_{1}"2'";"3'"};{\ar@{..}@/^.2pc/"3'";"4'"};{\ar@/^.3pc/_{1}"4'";"5'"};
{\ar@/_.1pc/_{1}"6'";"5'"};{\ar@{..}@/_.2pc/"7'";"6'"};{\ar@/_.1pc/_{1}"8'";"7'"};{\ar@/_.2pc/_{\lambda t}"1'";"8'"};
{\ar@{-->}_{1}"1";"1'"};{\ar@{-->}"2";"2'"};{\ar@{-->}"3";"3'"};{\ar@{-->}"4";"4'"};{\ar@{-->}^{\lambda}"5";"5'"};{\ar@{-->}"6";"6'"};{\ar@{-->}"7";"7'"};{\ar@{-->}"8";"8'"};
(-16,3)*{\text{\scriptsize{$\lambda$}}}="";(-11,-3)*{\text{\scriptsize{$\lambda$}}}="";(-6,3)*{\text{\scriptsize{$\lambda$}}}="";(1,-3)*{\text{\scriptsize{$\lambda$}}}="";(11,3)*{\text{\scriptsize{$\lambda$}}}="";(16,-3)*{\text{\scriptsize{$\lambda$}}}="";
(-37,16)*{V_{(s,t)}:}="";
(-37,-16)*{V_{\lambda (s,t)}:}="";
\endxy$$
$$\begin{array}{c}
\lambda \rightarrow 0 \\
\leadsto 
\end{array}
\ \ \ \
\xy 
(-26,16)*{\cdot}="1";(-17,21)*{\cdot}="2";(-7,23)*{\cdot}="3";(12,22.5)*{\cdot}="4";(26,16)*{\cdot}="5";(17,10)*{\cdot}="6";(0,8)*{\cdot}="7";(-12,9)*{\cdot}="8";
(-26,-16)*{\bullet}="1'";(-17,-11)*{\cdot}="2'";(-7,-9)*{\cdot}="3'";(12,-9.5)*{\cdot}="4'";(26,-16)*{\cdot}="5'";(17,-22)*{\cdot}="6'";(0,-24)*{\cdot}="7'";(-12,-23)*{\cdot}="8'";
{\ar@/^.1pc/^s"1";"2"};{\ar@/^.1pc/^1"2";"3"};{\ar@{..}@/^.2pc/"3";"4"};{\ar@/^.3pc/^1"4";"5"};
{\ar@/_.1pc/^1"6";"5"};{\ar@{..}@/_.2pc/"7";"6"};{\ar@/_.1pc/^1"8";"7"};{\ar@/_.2pc/^t"1";"8"};
{\ar@{..>}@/^.1pc/"1'";"2'"};{\ar@{..>}@/^.1pc/"2'";"3'"};{\ar@{..}@/^.2pc/"3'";"4'"};{\ar@{..>}@/^.3pc/"4'";"5'"};
{\ar@{..>}@/_.1pc/"6'";"5'"};{\ar@{..}@/_.2pc/"7'";"6'"};{\ar@{..>}@/_.1pc/"8'";"7'"};{\ar@{..>}@/_.2pc/"1'";"8'"};
{\ar@{-->}_{1}"1";"1'"};{\ar@{-->}"2";"2'"};{\ar@{-->}"3";"3'"};{\ar@{-->}"4";"4'"};{\ar@{-->}^{0}"5";"5'"};{\ar@{-->}"6";"6'"};{\ar@{-->}"7";"7'"};{\ar@{-->}"8";"8'"};
(-16,3)*{\text{\scriptsize{$0$}}}="";(-11,-3)*{\text{\scriptsize{$0$}}}="";(-6,3)*{\text{\scriptsize{$0$}}}="";(1,-3)*{\text{\scriptsize{$0$}}}="";(11,3)*{\text{\scriptsize{$0$}}}="";(16,-3)*{\text{\scriptsize{$0$}}}="";
(-37,16)*{V_{(s,t)}:}="";
(-37,-16)*{S_0:}="";
\endxy$$
\caption{ 
The $\mathbb{P}^1$-family $\{[V_{(s,t)}]\}$ shrunk to the vertex simple $[S_0]$ at the bold vertex.  A dotted arrow denotes an arrow represented by zero and a dotted edge denotes some number of arrows.} 
\label{P^1 family}
\end{figure}

\subsection{A first example: the blowup of $\mathbb{C}^n$} \label{A first example}

We now introduce a new noncommutative perspective on the tautological line bundle 
$$\pi: L := \left\{ (x,v) \in \mathbb{P}^{n-1} \times \mathbb{C}^n \ | \ v \in x \right\} \rightarrow \mathbb{C}^n, \ \ \ (x,v) \mapsto v,$$ 
whose total space is $\mathbb{C}^n$ blownup at the origin.  Consider the quiver algebra
\begin{equation} \label{taut}
A:=\mathbb{C}Q/\left\langle \left[c,c'\right] \ | \ c,c' \text{ cycles} \right\rangle
\end{equation}
with quiver given in figure \ref{tlb}.
\begin{figure}
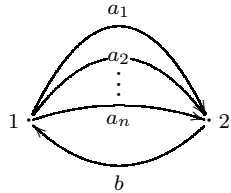

$$\xy
(-12,0)*{\cdot}="1";(12,0)*{\cdot}="2";
{\ar@/^.5pc/_{a_n}"1";"2"};{\ar@/^2pc/|-{a_2}"1";"2"};{\ar@/^1.4pc/^b"2";"1"};
{\ar@/^3pc/^{a_1}"1";"2"};
(-14,0)*{}="1'";(14,0)*{}="2'";
{\ar@{}|-1"1'";"1'"};{\ar@{}|-2"2'";"2'"};
(0,6)*{}="3";{\ar@{}|-{\vdots}"3";"3"};
\endxy$$
\caption{The tautological line bundle quiver.} \label{tlb}
\end{figure}
Recall that $S_i$ denotes the vertex simple at $i \in Q_0$.

\begin{Proposition} \label{first proposition}
Let $A$ be the quiver algebra (\ref{taut}).  The isoclasses of large modules, and almost large modules with socle $S_2$ (resp.\ $S_1$), are parameterized by $\mathbb{C}^n$ blownup at the origin (resp.\ $\mathbb{C}^n$).  Specifically,
\begin{itemize}
 \item the large modules are parameterized by $\mathbb{C}^n \setminus \{0\}$, while
 \item the almost large modules with socle $S_2$ (resp.\ $S_1$) are parameterized by the exceptional divisor $\pi^{-1}(0) = \mathbb{P}^{n-1}$ (resp.\ the single point $0$).
\end{itemize}
\end{Proposition}

\begin{proof} 
Denote by $Z$ the center of $A$.  The ideal of relations of $A$ is defined so that the corner rings $e_1Ae_1 = Ze_1$, $e_2Ae_2 = Ze_2$ are commutative, and so the algebra homomorphism 
$$\tau: A \rightarrow \operatorname{End}_A\left(\mathbb{C}[z_1,\ldots,z_n]\right)$$
defined by
$$\tau(a_i) = \left[ \begin{array}{cc} 0 & 0 \\ z_i & 0 \end{array} \right], \ \ \ \tau(b) = \left[ \begin{array}{cc} 0 & 1 \\ 0 & 0 \end{array} \right], \ \ \ \tau(e_1) = \left[ \begin{array}{cc} 1 & 0 \\ 0 & 0 \end{array} \right], \ \ \ \tau(e_2) = \left[ \begin{array}{cc} 0 & 0 \\ 0 & 1 \end{array} \right],$$
is a monomorphism.  It then follows from \cite[Proposition 2.9]{B} that the large modules have dimension vector $(1,1)$.  A module $V$ with this dimension vector is simple if and only if there is some $i$ such that $a_i$ and $b$ are represented by nonzero scalars, say $z_i$ and $y$.  However, if $y \not = 0$ then we may assume $y =1$, as shown by the isomorphism (i) in figure \ref{impression iso} (the dashed lines denote the isomorphism parameters between $A$-modules $W$ and $V$, where the resulting ``squares commute'').  Moreover, if two modules $V$ and $V'$ satisfy $y = y'=1$, then $V \cong V'$ if and only if $z_i = z_i'$ for each $i$, and so the large module isoclasses are parameterized by $\mathbb{C}^n \setminus 0$.

Now consider the module isomorphisms (ii) and (iii) in figure \ref{impression iso}, where the dotted arrows denote arrows represented by zero.  Denote by $\mathcal{P}$ the path-like set $Q_{\geq 0} \cup \{ 0 \}$.  For $w_1, \ldots, w_j \in \left\{ y, z_1, \ldots, z_n \right\}$ let $P(w_1 = \cdots = w_j = 0)$ denote the $\mathcal{P}$-annihilator of a module in $\operatorname{Rep}_{(1,1)}A$ with $w_1 = \cdots = w_j = 0$ and all other arrows represented by nonzero scalars.  Note that $\operatorname{dim}Z = n$ since $Z \cong \mathbb{C}[z_1, \ldots, z_n]$.  Then for $1 \leq \ell \leq n$ there is a maximal chain of subsets as in Definition \ref{almost large}, 
$$0 \subsetneq P_1(y=0) \subsetneq P_2(y=z_{i_1}=0) \subsetneq P_3(y=z_{i_1}=z_{i_2}=0) \subsetneq$$
$$\cdots \subsetneq P_{\ell} := P_{\ell}(y=z_{i_1}=z_{i_2}= \cdots = z_{i_{\ell -1}} = 0),$$
so any module whose $\mathcal{P}$-annihilator is $P_{\ell}$ is almost large.  Similarly
$$0 \subsetneq P_1(z_1=0) \subsetneq P_2(z_1 = z_2 = 0) \subsetneq \cdots \subsetneq P' := P_{n}(z_1 = z_2 = \cdots = z_n = 0),$$
so any module whose $\mathcal{P}$-annihilator is $P'$ is also almost large.  Any module whose $\mathcal{P}$-annihilator is $P_{\ell}$ has socle $S_2$ (since $\ell \not = n+1$), and the isoclasses of all such modules forms a $\mathbb{P}^{n-1}$-family since $\lambda \in \operatorname{GL}_1(\mathbb{C}) = \mathbb{C}^*$, which is shown by the module isomorphism (ii) in figure \ref{impression iso}.  Any module whose $\mathcal{P}$-annihilator is $P'$ has socle $S_1$, and there is only one such module up to isomorphism, shown by the module isomorphism (iii) in figure \ref{impression iso}.  In this case $y \in \mathbb{C}^*$, and the $Z$-annihilator of this single isoclass is the maximal ideal $m$ at the origin of $\mathbb{C}^n$.  Note that any module whose $\mathcal{P}$-annihilator is $P(z_1 = \cdots = z_{\ell}=0)$, where $1 \leq \ell \leq n-1$, is large and thus not almost large.

The path-like set $\mathcal{P} = Q_{\geq 0} \cup \{0\}$ is sufficient for determining all almost large modules since the almost large modules with socle $S_1$ or $S_2$ obtained from $Q_{\geq 0} \cup \{0\}$ exhaust the set of all modules in $\operatorname{Rep}_{(1,1)}A$ with socle $S_1$ or $S_2$.
\end{proof}
\begin{figure}
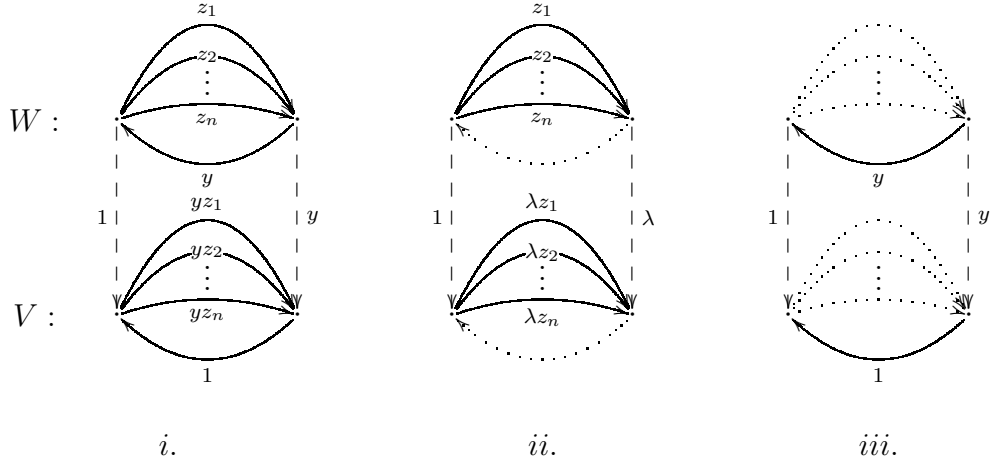

$$\begin{array}{ccccc}
\xy
(-12,13)*{\cdot}="1";(12,13)*{\cdot}="2";
{\ar@/^.5pc/_{z_n}"1";"2"};{\ar@/^2pc/|-{z_2}"1";"2"};{\ar@{->}@/^1.4pc/^y"2";"1"};
{\ar@/^3pc/^{z_1}"1";"2"};
(0,19)*{}="3";{\ar@{}|-{\vdots}"3";"3"};
(-12,-13)*{\cdot}="1b";(12,-13)*{\cdot}="2b";
{\ar@/^.5pc/_{yz_n}"1b";"2b"};{\ar@/^2pc/|-{yz_2}"1b";"2b"};{\ar@{->}@/^1.4pc/^1"2b";"1b"};
{\ar@/^3pc/^{yz_1}"1b";"2b"};
(0,-7)*{}="3b";{\ar@{}|-{\vdots}"3b";"3b"};
{\ar@{-->}_1"1";"1b"};{\ar@{-->}^{y}"2";"2b"};
(-23,13)*{W:}="";(-23,-13)*{V:}="";
\endxy
& \ \ \ \ & 
\xy
(-12,13)*{\cdot}="1";(12,13)*{\cdot}="2";
{\ar@/^.5pc/_{z_n}"1";"2"};{\ar@/^2pc/|-{z_2}"1";"2"};{\ar@{..>}@/^1.4pc/"2";"1"};
{\ar@/^3pc/^{z_1}"1";"2"};
(0,19)*{}="3";{\ar@{}|-{\vdots}"3";"3"};
(-12,-13)*{\cdot}="1b";(12,-13)*{\cdot}="2b";
{\ar@/^.5pc/_{\lambda z_n}"1b";"2b"};{\ar@/^2pc/|-{\lambda z_2}"1b";"2b"};{\ar@{..>}@/^1.4pc/"2b";"1b"};
{\ar@/^3pc/^{\lambda z_1}"1b";"2b"};
(0,-7)*{}="3b";{\ar@{}|-{\vdots}"3b";"3b"};
{\ar@{-->}_1"1";"1b"};{\ar@{-->}^{\lambda}"2";"2b"};
\endxy
& \ \ \ \ & 
\xy
(-12,13)*{\cdot}="1";(12,13)*{\cdot}="2";
{\ar@{..>}@/^.5pc/"1";"2"};{\ar@{..>}@/^2pc/"1";"2"};{\ar@/^1.4pc/^y"2";"1"};
{\ar@{..>}@/^3pc/"1";"2"};
(0,19)*{}="3";{\ar@{}|-{\vdots}"3";"3"};
(-12,-13)*{\cdot}="1b";(12,-13)*{\cdot}="2b";
{\ar@{..>}@/^.5pc/"1b";"2b"};{\ar@{..>}@/^2pc/"1b";"2b"};{\ar@/^1.4pc/^1"2b";"1b"};
{\ar@{..>}@/^3pc/"1b";"2b"};
(0,-7)*{}="3b";{\ar@{}|-{\vdots}"3b";"3b"};
{\ar@{-->}_1"1";"1b"};{\ar@{-->}^{y}"2";"2b"};
\endxy
\\ \\
i. & & ii. & & iii.
\end{array}$$
\caption{Some isomorphic $A$-modules.  Dotted arrows denote arrows represented by zero, and dashed arrows denote isomorphism parameters between $A$-modules.} \label{impression iso}
\end{figure}

We now describe how to shrink the $\mathbb{P}^{n-1}$ to zero size using the noncommutative algebra $A$.  Let $M = \mathbb{C}^n\setminus \{0 \}$, $\mathbb{T} = U(1) \subset \mathbb{C}^*$, and consider the moment map
$$\mu: M \rightarrow g^*= \mathbb{R}$$
defined by
$$\mu(z_1, \ldots, z_n) = \frac 12 \left( |z_1|^2 + \cdots + |z_n|^2 \right).$$
Then
$$\begin{array}{rcl}
\mu^{-1}(1/2)/\mathbb{T} & = & \left\{ (z_1,\ldots, z_n) \in M \ | \ |z_1|^2+ \cdots + |z_n|^2=1 \right\}/\mathbb{T}\\ \\
& = & \left\{ \mathbb{P}^{n-1} \text{ with radius } 1 \right\},\end{array}$$
and more generally
$$\begin{array}{rcl}
\mu^{-1}(|\lambda|^2/2 )/\mathbb{T} & = & \left\{ (\lambda z_1, \ldots, \lambda z_n) \in M \ | \ |z_1|^2 + \cdots + |z_n|^2 = 1 \right\}/\mathbb{T} \\ \\
& = & \left\{ \mathbb{P}^{n-1} \text{ with radius } |\lambda| \right\}.\end{array}$$
Varying $\lambda$ is equivalent to varying the radius of $\mathbb{P}^{n-1}$.  In particular, $\lambda \rightarrow 0$ is equivalent to the radius vanishing, and in our case of interest, the isomorphism (ii) of figure \ref{impression iso} becomes a module epimorphism, given in figure \ref{mono}.
\begin{figure}
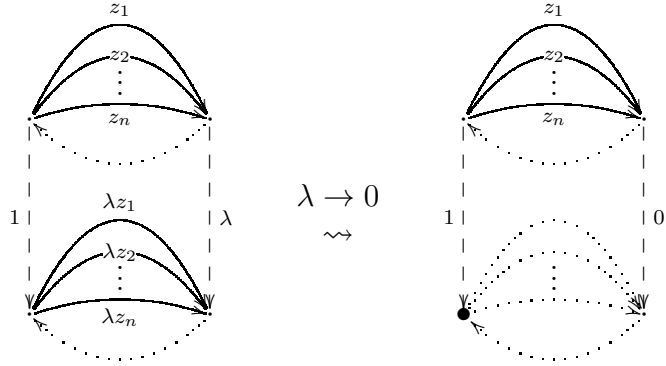

$$\xy
(-12,13)*{\cdot}="1";(12,13)*{\cdot}="2";
{\ar@/^.5pc/_{z_n}"1";"2"};{\ar@/^2pc/|-{z_2}"1";"2"};{\ar@{..>}@/^1.4pc/"2";"1"};
{\ar@/^3pc/^{z_1}"1";"2"};
(0,19)*{}="3";{\ar@{}|-{\vdots}"3";"3"};
(-12,-13)*{\cdot}="1b";(12,-13)*{\cdot}="2b";
{\ar@/^.5pc/_{\lambda z_n}"1b";"2b"};{\ar@/^2pc/|-{\lambda z_2}"1b";"2b"};{\ar@{..>}@/^1.4pc/"2b";"1b"};
{\ar@/^3pc/^{\lambda z_1}"1b";"2b"};
(0,-7)*{}="3b";{\ar@{}|-{\vdots}"3b";"3b"};
{\ar@{-->}_1"1";"1b"};{\ar@{-->}^{\lambda}"2";"2b"};
\endxy 
\ \ \ \
\begin{array}{c}
\lambda \rightarrow 0 \\
\leadsto 
\end{array}
\ \ \ \
\xy
(-12,13)*{\cdot}="1";(12,13)*{\cdot}="2";
{\ar@/^.5pc/_{z_n}"1";"2"};{\ar@/^2pc/|-{z_2}"1";"2"};{\ar@{..>}@/^1.4pc/"2";"1"};
{\ar@/^3pc/^{z_1}"1";"2"};
(0,19)*{}="3";{\ar@{}|-{\vdots}"3";"3"};
(-12,-13)*{\bullet}="1b";(12,-13)*{\cdot}="2b";
{\ar@{..>}@/^.5pc/"1b";"2b"};{\ar@{..>}@/^2pc/"1b";"2b"};{\ar@{..>}@/^1.4pc/"2b";"1b"};
{\ar@{..>}@/^3pc/"1b";"2b"};
(0,-7)*{}="3b";{\ar@{}|-{\vdots}"3b";"3b"};
{\ar@{-->}_1"1";"1b"};{\ar@{-->}^{0}"2";"2b"};
\endxy$$
\caption{Shrinking the exceptional locus to zero size.} \label{mono}
\end{figure}
The vertex simple $S_1$, which is not an almost large module, may therefore be viewed as the $\mathbb{P}^{n-1}$ shrunk to zero size.  Note that $S_1$ is the top of every module in the $\mathbb{P}^{n-1}$-family.  Moreover, even though this module corresponds to a point at the origin of $\mathbb{C}^n$, it is \textit{not} the module (isoclass) corresponding to the actual origin of $\mathbb{C}^n$, namely the isoclass given in (iii) of figure \ref{impression iso}.

\subsubsection{Socle vs.\ top}

In Conjecture \ref{conjecture} we made a choice of restricting our attention to almost large modules with isomorphic 1-dimensional socles rather than isomorphic 1-dimensional tops.  These two choices--either fixing the socle or fixing the top--appear equally suitable for the examples we will encounter in section \ref{Res sing}, but they are not equal in regards to the noncommutative tautological line bundle algebra $A$ defined in (\ref{taut}).  For consider the geometric interpretation of projective dimension: if $R$ is the (commutative) coordinate ring for an algebraic variety and $p \in \operatorname{Spec}R$ is smooth, then the projective dimension of $R_p/p_p$ equals the codimension of $p$ (that is, the codimension of the irreducible subvariety defined by $p$).  Therefore since  $\operatorname{pd}_A(S_1) = n$ and $\operatorname{pd}_A(S_2) = 1$, $S_1$ should be viewed as a zero-dimensional point in $\operatorname{Max}A$ while $S_2$ should be viewed as an $(n-1)$-dimensional ``point''.\footnote{Given any almost large module $W$ with socle $S_2$, there exists minimal projective resolutions of $W$ and the vertex simple $S_1$ that are identical except for a factor of $b$ that ``switches sides'' in the first two connecting maps.  For an explicit example, consider $n =3$.  The homomorphism $Ae_1 \otimes e_1V \stackrel{\delta}{\longrightarrow} V$, $\delta(c \otimes v)= cv$, is a projective cover for both $V=W$ and $V = S_1$.  Let $I \subset Ae_1$ be the left ideal such that $\operatorname{ker}\delta_0 = I \otimes e_1V$; then if $V=W$ (resp.\ $V=S_1$),
$$I = \left\langle c_i:=x_ia_{i+1}-x_{i+1}a_i, \ ba_i \ | \ i =1,2,3 \right\rangle = \left\langle c_1,c_2, ba_1 \right\rangle$$ $$\left(\text{resp.\ } I = \left\langle a_1, a_2, a_3 \right\rangle = \left\langle c_1, c_2, a_1 \right\rangle \right).$$
The sequence
$$0 \rightarrow Ae_2 \otimes V 
\stackrel{\cdot \left[ \begin{array}{ccc} a_1b & c_2b & c_1b \end{array}\right] \otimes 1}{\longrightarrow}
\left( Ae_1\right)^{\oplus 3} \otimes V 
\stackrel{ \cdot \left[ \begin{array}{ccc} c_2b & -c_1b & 0 \\ -a_1b & 0 & c_1\beta_2 \\ 0 & a_1b & -c_2\beta_2 \end{array} \right] \otimes 1}{\longrightarrow}$$
$$\left( \left( Ae_2 \right)^{\oplus 2} \oplus Ae_1 \right) \otimes V \stackrel{\cdot \left[ \begin{array}{c} c_1 \\ c_2 \\ \beta_1a_1 \end{array} \right] \otimes 1}{\longrightarrow} 
Ae_1 \otimes V 
\stackrel{\delta}{\longrightarrow} V \rightarrow 0$$
is a minimal projective resolution of $V=W$ (resp.\ $V=S_1$) when 
$$\left(\beta_2,\beta_1\right) = \left\{ \begin{array}{cc} (1,b) & \text{ if } V =W \\ (b,1) & \text{ if } V=S_1 \end{array} \right..$$
However, for any $n$ the projective dimension of the vertex simple $S_2$ is only 1, 
$$0 \rightarrow Ae_1 \otimes e_2S_2 \stackrel{\cdot c \otimes 1}{\longrightarrow} Ae_2 \otimes e_2S_2 \stackrel{\delta_0}{\longrightarrow} S_2 \rightarrow 0.$$}
  It follows that if the $\mathbb{P}^{n-1}$ shrinks to a zero-dimensional point, then it should shrink to $S_1$ and not $S_2$.

\section{$\mathbb{P}^n$-families} \label{Explicit Methods}

\subsection{Determining $\mathbb{P}^n$-families} \label{Determining P^n families}

We now give an explicit method for determining a $\mathbb{P}^n$-family of module isoclasses over a quiver algebra $A= kQ/I$.  Recall the notation of Definition \ref{P^n family}.\\
\\
\indent 1. \textit{Fix the support of $\sigma$.}  This may be done efficiently by fixing a \textit{pulled-apart} supporting subquiver $\widetilde{Q}$ of $Q$; given a representation $\rho: A \rightarrow \operatorname{Mat}_d(\mathbb{C})$, or the corresponding $A$-module $\mathbb{C}^d$, the quiver $\widetilde{Q}$ is defined by
$$\begin{array}{ccl}
\widetilde{Q}_0 & = & \left\{ 1, \ldots, \operatorname{rank}\rho(1) \right\},\\
\widetilde{Q}_1 & = & \bigsqcup_{a \in Q_1} \left\{ i \rightarrow j \ | \ (\rho(a))_{ji} \not = 0 \right\},
\end{array}$$
where $(\rho(a))_{ji}$ is the $ji$-th entry of the matrix $\rho(a)$.  Note that this quiver depends on a choice of basis for $\mathbb{C}^d$.  If $\rho$ has dimension vector $(1, \ldots, 1)$, then $\widetilde{Q}_0=Q_0$.

For fixed $\widetilde{Q}$, define the ideal $J_0 \subset \mathbb{C}[x_a]:=\mathbb{C}[x_a \ | \ a \in \widetilde{Q}_1 ]$ so that the map 
\begin{equation}
\label{sigma_0} 
\sigma_0: A \rightarrow \operatorname{Mat}_d(\mathbb{C}[x_a]/J_0), \ \ \ \ 
\sigma_0(a) := \left\{ \begin{array}{ll} x_aE_a & \text{ if } \ a \in \widetilde{Q}_1, \\ E_a & \text{ if } \ a \in \widetilde{Q}_0, \end{array} \right.
\end{equation}
is an algebra monomorphism, where for a path $a$ in $\widetilde{Q}$, $E_a$ denotes the matrix with a $1$ in the $(\operatorname{h}(a), \operatorname{t}(a))$-th slot and zeros elsewhere.\\
\\
\indent 2. \textit{Trivialize the ideal $J_0$.}  Suppose $\widetilde{Q}$ is a pulled-apart subquiver of $Q$ (with respect to some basis) that contains a sink at $0 \in \widetilde{Q}_0$.  We apply the following iterative procedure on $n$ to trivialize the ideal $J_0$ in (\ref{sigma_0}).  For $n \geq 1$, define
\begin{equation}
\label{sigma'}
\sigma_n: A \longrightarrow \operatorname{Mat}_d(\mathbb{C}[x_a]/J_n)
\end{equation}
as follows:

If $n =1$, let $i = 0$.

Suppose $b \in \widetilde{Q}_1e_i$.  If for each $a \in \widetilde{Q}_1e_i$ there is some $\alpha_{a} \in \mathbb{C}$ such that $x_b = \alpha_{a} x_a$ (modulo $J_{n-1}$) (in particular, if $\widetilde{Q}_1e_i = \{b\}$), then set
\begin{equation}
\label{sigma,J}\begin{array}{rcl}
\sigma_n(a) & := & \left\{ \begin{array}{cl} \alpha_{a} E_a & \text{ if } a \in \widetilde{Q}_1e_i \\ x_aE_a & \text{ otherwise } \end{array} \right.\\
J_n & := & \left\langle I_n, x_a \ | \ a \in \widetilde{Q}_1e_i \right\rangle,
\end{array}
\end{equation}
where the ideal $I_n$ is defined so that (\ref{sigma'}) is an algebra monomorphism.  Otherwise do nothing.  

Next, if $e_i\widetilde{Q}_1$ is non-empty, choose $a \in e_i\widetilde{Q}_1$ and set $j:= \operatorname{t}(a) \in \widetilde{Q}_0$.  Otherwise choose any vertex $j$ where there exists an $a \in \widetilde{Q}_1e_j$ such that $\sigma_n(a) = x_a E_a$ and $\sigma_n(b) \propto E_b$ for all $b \in \widetilde{Q}_1e_{\operatorname{h}(a)}$ (the latter condition is trivially satisfied if $j$ is a sink).

Repeat this process with $i = j$ until there does not exist such a $j$, and denote the final representation by
$$\sigma: A \longrightarrow \operatorname{Mat}_d(\mathbb{C}[x_a]/J).$$

In the examples we will consider, we will find that $\mathbb{C}[x_a]/J \cong \mathbb{C}[t_1, \ldots, t_{m}]$ for some $m$.  The following lemma says that when this is the case, it possible that the family of all modules supported on $\widetilde{Q}$ forms a $\mathbb{P}^{m-1}$-family.  Denote by $\epsilon_z: \operatorname{Mat}_d(\mathbb{C}[x_a]/J_n) \longrightarrow \operatorname{Mat}_d((\mathbb{C}[x_a]/J_n)/(x_a-z_a)) \cong \operatorname{Mat}_d(\mathbb{C})$ the evaluation map at the point $z = (z_a)_{a \in \widetilde{Q}_1} \in \mathbb{C}^{|\widetilde{Q}_1|}$.

\begin{Lemma} \label{rho cong}
If $\rho$ is a representation of $A$ with pulled-apart supporting subquiver $\widetilde{Q}$, then there exists a point $z \in (\mathbb{C}^*)^{|\widetilde{Q}_1|}$ such that 
$$\rho \cong \epsilon_z \cdot \sigma.$$
\end{Lemma}

\begin{proof} Clearly there exists some $z \in (\mathbb{C}^*)^{|\widetilde{Q}_1|}$ such that $\rho = \epsilon_{z} \cdot \sigma_0$.  We claim that given any point $z \in (\mathbb{C}^*)^{|\widetilde{Q}_1|}$ there exists a point $z' \in (\mathbb{C}^*)^{|\widetilde{Q}_1|}$ such that
\begin{equation}
\label{iso sigma}
\epsilon_{z} \cdot \sigma_{n-1} \cong \epsilon_{z'} \cdot \sigma_{n}.
\end{equation}
Let $b \in \widetilde{Q}_1$ be such that $\sigma_{n-1}(b)=bE_b$ and $\sigma_{n}(b)= E_b$, and set
$$z'_a := \left\{ \begin{array}{ll} z_bz_a & \text{ if } \ a \in \widetilde{Q}_1e_{\operatorname{h}(b)} \\ z_b^{-1}z_a & \text{ if } \ a \in e_{\operatorname{h}(b)}\widetilde{Q}_1 \\ z_a & \text{ otherwise} \end{array} \right.$$
In particular, $z'_b=1$.  The isomorphism (\ref{iso sigma}) then follows from the definition of $\sigma_n$ (\ref{sigma,J}), explicitly given by $\operatorname{diag}(1, \ldots, 1, z_a^{-1},1, \ldots, 1) \in \operatorname{GL}_d(\mathbb{C})$, with $z_a^{-1}$ in the $\operatorname{h}(a)$-th slot.  Schematically, there is an isomorphism of representations:
$$\xy
(-13,7)*{\cdot}="1";(-13,-7)*{\cdot}="2";(0,0)*{\cdot}="3";(13,0)*{\cdot}="4";
{\ar^{x_a}"1";"3"};{\ar^{x_c}"3";"2"};{\ar^{x_b}"3";"4"};
\endxy
\ \ \ \cong \ \ \
\xy
(-13,7)*{\cdot}="1";(-13,-7)*{\cdot}="2";(0,0)*{\cdot}="3";(13,0)*{\cdot}="4";
{\ar^{x_b x_a}"1";"3"};{\ar^{x_b^{-1} x_c}"3";"2"};{\ar^{1}"3";"4"};
\endxy$$
Consequently there is some $z^0, z^1, \ldots, z^N \in \mathbb{C}^{|\widetilde{Q}_1|}$ such that
$$\rho = \epsilon_{z^0} \cdot \sigma_0 \cong \epsilon_{z^1} \cdot \sigma_1 \cong \ \ \cdots \ \ \cong \epsilon_{z^N} \cdot \sigma.$$
\end{proof}

3. \textit{Solve the isomorphism parameters.}  Suppose that $\mathbb{C}[x_{\alpha}]/J \cong \mathbb{C}[t_1, \ldots, t_{m}]$.  Set $\phi: V_{(t_1, \ldots, t_{m})} \stackrel{\cong}{\longrightarrow} V_{(\lambda_1 t_1, \ldots, \lambda_{m} t_{m})}$, so that $(t_1, \ldots, t_{m}) \sim (\lambda_1 t_1, \ldots, \lambda_{m} t_{m})$, and solve the relations among the $\lambda_i$.

In the following example we demonstrate how to ``solve the isomorphism parameters'' to show that a family of modules is a $\mathbb{P}^1$-family.

\begin{Example} \rm{Consider the family of modules over the path algebra given in the second column of figure \ref{s,t}.iii.  To show that this is a $\mathbb{P}^1$-family we need to show that $\lambda = \mu$.  Denote the isomorphism parameters by
$$1,f,g \in \operatorname{GL}_1(\mathbb{C}), \ \ \left[ \begin{array}{cc} a & b \\ c & d \end{array} \right] \in \operatorname{GL}_2 (\mathbb{C}),$$
at the respective vertices $1,2,3,4 \in Q_0$; we then solve for these parameters by requiring that the relevant ``squares commute'':
$$\begin{array}{rcl}
\left[ \begin{array}{cc} 1 & 0 \end{array} \right]f = \left[ \begin{array}{cc} 1 & 0 \end{array} \right]\left[ \begin{array}{cc} a & b \\ c & d \end{array} \right] 
& \Rightarrow & b = 0 \text{ and } f = a\\ \\
\left[ \begin{array}{cc} a & 0 \\ c & d \end{array} \right] \left[ \begin{array}{c} 0 \\ 1 \end{array} \right] = \left[ \begin{array}{c} 0 \\ 1 \end{array} \right] f & \Rightarrow & d = f (=a)\\ \\
\left[ \begin{array}{cc} 0 & 1 \end{array} \right]g = \left[ \begin{array}{cc} 1 & 0 \end{array} \right]\left[ \begin{array}{cc} a & 0 \\ c & a \end{array} \right] 
& \Rightarrow & c = 0 \text{ and } g = a\\ \\
\left[ \begin{array}{cc} a & 0 \\ 0 & a \end{array} \right] \left[ \begin{array}{c} s \\ t \end{array} \right] = 1 \left[ \begin{array}{c} \lambda s \\ \mu t \end{array} \right] 
& \Rightarrow & \lambda = a = \mu
\end{array}$$
} \end{Example}

\begin{figure}
$$\begin{array}{|l|c|c|}
\hline
\begin{array}{l}
(i) \\
\xy
(0,-10)*{\cdot}="1";(-10,0)*{\cdot}="2";(0,10)*{\cdot}="3";(10,0)*{\cdot}="4";
{\ar^{s}"1";"2"};
{\ar^1"2";"3"};
{\ar_{t}"1";"4"};
{\ar_1"4";"3"};
{\ar@{..}"1";"3"};
\endxy
\end{array}
&
\begin{array}{c}
\xy
(-19,0)*{\cdot}="1";
(0,0)*{\cdot}="2";
(19,0)*{\cdot}="3";
{\ar@/^/^{\left[ \begin{array}{c} 1 \\ 0 \end{array} \right]}"1";"2"};
{\ar@/^/^{\left[ \begin{array}{cc} 0 & s \end{array} \right]}"2";"1"};
{\ar@/_/_{\left[ \begin{array}{c} 1 \\ 0 \end{array} \right]}"3";"2"};
{\ar@/_/_{\left[ \begin{array}{cc} 0 & t \end{array} \right]}"2";"3"};
\endxy
\cong
\xy
(-19,-1)*{\cdot}="1";
(0,-1)*{\cdot}="2";
(19,-1)*{\cdot}="3";
{\ar@/^/^{\left[ \begin{array}{c} 1 \\ 0 \end{array} \right]}"1";"2"};
{\ar@/^/^{\left[ \begin{array}{cc} 0 & \lambda s \end{array} \right]}"2";"1"};
{\ar@/_/_{\left[ \begin{array}{c} 1 \\ 0 \end{array} \right]}"3";"2"};
{\ar@/_/_{\left[ \begin{array}{cc} 0 & \mu t \end{array} \right]}"2";"3"};
\endxy
\\
\Rightarrow \ \lambda = \mu
\end{array}
&
\mathbb{P}^1
\\
\hline
\begin{array}{l}
(ii) \\
\xy
(0,-10)*{\cdot}="1";(-10,0)*{\cdot}="2";(0,10)*{\cdot}="3";(10,0)*{\cdot}="4";
{\ar^{s}"1";"2"};
{\ar^1"2";"3"};
{\ar_{t}"1";"4"};
{\ar_1"4";"3"};
{\ar@{..}"2";"4"};
\endxy
\end{array}
&
\begin{array}{c}
\xy
(-19,0)*{\cdot}="1";
(0,0)*{\cdot}="2";
(19,0)*{\cdot}="3";
{\ar^{\left[ \begin{array}{c} s \\ t \end{array} \right]}"1";"2"};
{\ar^{\left[ \begin{array}{cc} 1 & 1 \end{array} \right]}"2";"3"};
\endxy
\cong
\xy
(-19,0)*{\cdot}="1";
(0,0)*{\cdot}="2";
(19,0)*{\cdot}="3";
{\ar^{\left[ \begin{array}{c} \lambda s \\ \mu t \end{array} \right]}"1";"2"};
{\ar^{\left[ \begin{array}{cc} 1 & 1 \end{array} \right]}"2";"3"};
\endxy
\\
\Rightarrow \ \lambda = \mu
\end{array}
&
\mathbb{P}^1
\\
\hline
\begin{array}{l}
(iii) \\
\xy
(0,-5)*{\cdot}="1";(-10,5)*{\cdot}="2";(0,15)*{\cdot}="3";(10,5)*{\cdot}="4";
(0,-15)*{\cdot}="5";
{\ar^{s}"2";"3"};
{\ar^1"3";"4"};
{\ar_{t}"2";"1"};
{\ar^1"4";"1"};
{\ar@{..}"1";"3"};
{\ar^{1}"1";"5"};
\endxy
\end{array}
&
\begin{array}{c}
\xy
(-19,0)*+{\text{\scriptsize{$1$}}}="1";
(0,0)*+{\text{\scriptsize{$4$}}}="2";
(19,0)*+{\text{\scriptsize{$3$}}}="3";
(0,-15)*+{\text{\scriptsize{$2$}}}="4";
{\ar^{\left[ \begin{array}{c} s \\ t \end{array} \right]}"1";"2"};
{\ar@/_/_{\left[ \begin{array}{c} 0 \\ 1 \end{array} \right]}"3";"2"};
{\ar@/_/_{\left[ \begin{array}{cc} 1 & 0 \end{array} \right]}"2";"3"};
{\ar_{\left[ \begin{array}{cc} 0 & 1 \end{array} \right]}"2";"4"};
\endxy
\cong
\xy
(-19,0)*{\cdot}="1";
(0,0)*{\cdot}="2";
(19,0)*{\cdot}="3";
(0,-15)*{\cdot}="4";
{\ar^{\left[ \begin{array}{c} \lambda s \\ \mu t \end{array} \right]}"1";"2"};
{\ar@/_/_{\left[ \begin{array}{c} 0 \\ 1 \end{array} \right]}"3";"2"};
{\ar@/_/_{\left[ \begin{array}{cc} 1 & 0 \end{array} \right]}"2";"3"};
{\ar_{\left[ \begin{array}{cc} 0 & 1 \end{array} \right]}"2";"4"};
\endxy
\\
\Rightarrow \ \lambda = \mu
\end{array}
&
\mathbb{P}^1
\\
\hline
\begin{array}{l}
(iv) \\
\xy
(0,-5)*{\cdot}="1";(-10,5)*{\cdot}="2";(0,15)*{\cdot}="3";(10,5)*{\cdot}="4";
(0,-15)*{\cdot}="5";
{\ar^{s}"2";"3"};
{\ar^1"3";"4"};
{\ar_{t}"2";"1"};
{\ar^1"4";"1"};
{\ar@{..}"1";"3"};
{\ar_{1}"5";"1"};
\endxy
\end{array}
&
\begin{array}{c}
\xy
(-19,0)*{\cdot}="1";
(0,0)*{\cdot}="2";
(19,0)*{\cdot}="3";
(0,-15)*{\cdot}="4";
{\ar^{\left[ \begin{array}{c} s \\ t \end{array} \right]}"1";"2"};
{\ar@/_/_{\left[ \begin{array}{c} 0 \\ 1 \end{array} \right]}"3";"2"};
{\ar@/_/_{\left[ \begin{array}{cc} 1 & 0 \end{array} \right]}"2";"3"};
{\ar^{\left[ \begin{array}{cc} 0 & 1 \end{array} \right]}"4";"2"};
\endxy
\cong
\xy
(-19,0)*{\cdot}="1";
(0,0)*{\cdot}="2";
(19,0)*{\cdot}="3";
(0,-15)*{\cdot}="4";
{\ar^{\left[ \begin{array}{c} \lambda s \\ \mu t \end{array} \right]}"1";"2"};
{\ar@/_/_{\left[ \begin{array}{c} 0 \\ 1 \end{array} \right]}"3";"2"};
{\ar@/_/_{\left[ \begin{array}{cc} 1 & 0 \end{array} \right]}"2";"3"};
{\ar^{\left[ \begin{array}{cc} 0 & 1 \end{array} \right]}"4";"2"};
\endxy
\\
\Rightarrow \ \lambda \in \mathbb{C}^* \text{ and } \mu \in \mathbb{C} \text{ whenever } s \not = 0
\end{array}
&
\begin{array}{c}
\text{ two points: }\\
s=0, s \not = 0
\end{array}
\\
\hline
\end{array}$$
\caption{Examples of modules over path algebras, with $s,t \in \mathbb{C}$ not both zero, are given in the middle column, and their corresponding pulled-apart quivers (with respect to the standard basis) are given in the left column.  Vertices in the pulled-apart quiver connected by a dotted edge correspond to the same vertex in the quiver itself.  In (i) - (iii), the coordinates $(s:t)$ parameterize $\mathbb{P}^1$-families of almost large modules that will appear in section \ref{D_n example}, and (iv) is cautionary.}
\label{s,t}
\end{figure}

\subsection{Coordinates on resolved singularities via impressions} \label{Coordinates on P^n families}

In this section we recall the definition of an impression, a notion the author introduced in \cite[section 2.1]{B}.  An impression may be thought of as a way of placing (commutative) coordinates within an algebra that is module-finite over its center.

\begin{Definition/Lemma} \label{impression} \cite[Definition 2.1]{B}
Let $k$ be an algebraically closed field, and let $A$ be a f.g.\ $k$-algebra, module-finite over its center $Z$.  Suppose that there exists a commutative noetherian reduced $k$-algebra $B$, an open dense subset $U \subseteq \operatorname{Max}B$, and an algebra momomorphism $\tau: A \rightarrow \operatorname{End}_B \left(B^d \right)$ such that the composition
$$\tau_m: A \stackrel{\tau}{\longrightarrow} \operatorname{End}_B\left( B^d \right) \stackrel{ \epsilon_m}{\longrightarrow} \operatorname{End}_B\left( (B/m)^d \right) \cong \operatorname{End}_k \left( k^d \right)$$
is a large representation of $A$ for each $m \in U$. Then
\begin{equation} \label{U}
Z \cong \left\{f \in B \ | \ f1_d \in \operatorname{im}\tau \right\} \subset B.
\end{equation}
If the induced morphism of varieties
\begin{equation} \label{U'} 
\operatorname{Max}B \stackrel{\phi}{\rightarrow} \operatorname{Max}Z
\end{equation}
is surjective, then we call $(\tau,B)$ an \rm{impression} \textit{of $A$.}
\end{Definition/Lemma}

The following demonstrates the utility of an impression.

\begin{Proposition} \label{impression large} \cite[Proposition 2.5]{B} Let $(\tau, B)$ be an impression of a prime algebra $A$.  If $V$ is a large $A$-module, then there is some $r \in \operatorname{Max}B$ such that $V \cong (B/r)^{d}$.
\end{Proposition}

Now let $A = kQ/I$ be a quiver algebra.  For $i \in Q_0$, set $d_i := \operatorname{rank}\tau(e_i)$.  If $a \in e_jAe_i$ for some $i,j \in Q_0$, then we denote by $\bar{\tau}(a)$ the restriction of $\tau(a)$ to 
\begin{equation} \label{bar}
B^{d_i} \cong \tau(e_i)B^{d} \rightarrow B^{d_j} \cong \tau(e_j)B^{d}.
\end{equation}
For example, if the large $A$-modules have dimension vector $(1, \ldots, 1)$, then $\bar{\tau}(a) \in B$ whenever $a \in e_jAe_i$.  In sections \ref{The conifold}, \ref{Cyclic quotient surface singularities}, and \ref{A non-isolated quotient singularity}, we will consider quiver algebras that admit impressions $(\tau,B)$ satisfying\footnote{Let $A$ be a quiver algebra and suppose that $B$, $\tau$, and $U$ are as in Definition \ref{impression} with $d< \infty$, but without requiring $A$ be module-finite over its center or that $\phi$ exists.  It was shown \cite[Theorem 2.7]{B} that if (\ref{B}) holds, then $A$ and its center $Z$ are both noetherian rings, $A$ is a finitely generated $Z$-module, and 
$$Z = k \left[ \sum_{i \in Q_0} \gamma_i \in \bigoplus_{i \in Q_0}e_iAe_i \ | \ \bar{\tau}(\gamma_i)= \bar{\tau}(\gamma_j) \text{ for each } i,j \in Q_0 \right].$$
Moreover, if we only assume that there is an algebra monomorphism $\tau: A \rightarrow \operatorname{End}_B(B^d)$ such that (\ref{B}) holds, then the dimension vector $d$ of any large $A$-module is bounded by $d \leq (1,\ldots,1)$ \cite[Proposition 2.9]{B}.}
\begin{equation} \label{B}
\bar{\tau}(e_iAe_i) = \bar{\tau}(e_jAe_j) \subset B \ \text{ for each } \ i,j \in Q_0.
\end{equation}
In each of these examples, $(\tau,B)$ determines a structure sheaf $\mathcal{O}_X$ on the parameterizing space $X$ of isoclasses of large modules and almost large modules with fixed vertex simple socle, that coincides precisely with the structure sheaf obtained by blowing up the singularity.  The construction of $\mathcal{O}_X$ from $(\tau,B)$ is as follows.

For each $x \in X$, let $Q(x)$ denote the supporting subquiver of $x$, and for each Zariski-open affine subset $U \subset X$, set 
$$Q(U) := \bigcap_{x \in U} Q(x) \subseteq Q.$$
Define the new quiver
$$Q'(U) := \left\{ \begin{array}{ccl} Q_0'(U) & = & Q_0, \\ Q'_1(U) & = & Q_1 \cup \left\{ \operatorname{h}(a) \stackrel{a^*}{\longrightarrow} \operatorname{t}(a) \ | \ a \in Q_1(U) \right\}, \end{array} \right.$$
which contains $Q$ as a subquiver, and set 
$$A(U) := kQ'(U)/\left\langle I, \ aa^* - e_{\operatorname{h}(a)}, \ a^*a - e_{\operatorname{t}(a)} \ | \ a \in Q_1(U) \right\rangle,$$
which contains $A$ as a subalgebra.  Extend $\tau: A \rightarrow \operatorname{Mat}_d(B)$ to an algebra monomorphism 
$$\tau : A(U) \longrightarrow \operatorname{Mat}_d\left( \operatorname{Frac}(B) \right)$$
defined by
\begin{equation} \label{a^*}
\begin{array}{rcccll} 
\tau(a) & := & \tau(a) & = & \bar{\tau}(a) E_{\operatorname{h}(a),\operatorname{t}(a)} & \text{ for } a \in Q_1,\\
\tau\left( a^* \right) & := & & & \bar{\tau}(a)^{-1} E_{\operatorname{t}(a), \operatorname{h}(a)} & \text{ for } a \in Q_1(U),
\end{array}
\end{equation}
where $E_{ij}$ denotes the matrix whose $ij$th entry is $1$, and zeros elsewhere.  We may then define the structure sheaf $\mathcal{O}_X$ induced by the impression $(\tau,B)$ to be
\begin{equation} \label{O}
\mathcal{O}_X(U):= \bar{\tau}\left( e_i A(U) e_i \right).
\end{equation}

\begin{Remark} \rm{If the dimension vector of the large modules over a quiver algebra is not $(1, \ldots, 1)$ then it is not immediately clear how to generalize this construction, specifically (\ref{a^*}), since in general $\bar{\tau}(a)$ may not be invertible.
}\end{Remark}

\begin{figure}
$$\begin{array}{ccc}
\xy
(6.235,7.818)*{\cdot}="2";(-2.225,9.750)*{\circ}="1";(-9.010,4.339)*{\cdot}="7";(-9.010,-4.339)*{\bullet}="6";(-2.225,-9.750)*{\cdot}="5";(6.235,-7.818)*{\cdot}="4";(10,0)*{\cdot}="3";
{\ar@{<-}^x"1";"2"};{\ar@{<-}^x"2";"3"};{\ar@{<-}^x"3";"4"};{\ar@{<-}^x"4";"5"};{\ar@{<-}^x"5";"6"};{\ar@{}"6";"7"};{\ar@{}"7";"1"};
{\ar@{<-}_y"1";"7"};{\ar@{<-}_y"7";"6"};{\ar@{}"6";"5"};{\ar@{}"5";"4"};{\ar@{}"4";"3"};{\ar@{}"3";"2"};{\ar@{}"2";"1"};
\endxy
& \ \ \ \cong \ \ \ &
\xy
(6.235,7.818)*{\cdot}="2";(-2.225,9.750)*{\circ}="1";(-9.010,4.339)*{\cdot}="7";(-9.010,-4.339)*{\bullet}="6";(-2.225,-9.750)*{\cdot}="5";(6.235,-7.818)*{\cdot}="4";(10,0)*{\cdot}="3";
{\ar@{<-}^1"1";"2"};{\ar@{<-}^1"2";"3"};{\ar@{<-}^1"3";"4"};{\ar@{<-}^1"4";"5"};{\ar@{<-}^{s=x^5}"5";"6"};{\ar@{}"6";"7"};{\ar@{}"7";"1"};
{\ar@{<-}_1"1";"7"};{\ar@{<-}_{t=y^2}"7";"6"};{\ar@{}"6";"5"};{\ar@{}"5";"4"};{\ar@{}"4";"3"};{\ar@{}"3";"2"};{\ar@{}"2";"1"};
\endxy
\\
\\
(i) \ \ \ \text{coordinates } (x^5:y^2) & & (ii) \ \ \ \text{coordinates } (s:t)\\
\ \ \ \ \ \ \ \mathcal{O}_X(U):= \bar{\tau}\left( e_i A(U) e_i \right) & &
\ \ \ \ \ \ \ \ \ \ \ \ \ \ \ \ \sigma: A \rightarrow \operatorname{End}_{\mathbb{C}[s,t]}\left( \mathbb{C}[s,t]^{\oplus 7} \right)
\end{array}$$
\caption{Isomorphic labeling of arrows for the supporting subquiver $\widetilde{Q}^{5}$ of the $\mathbb{P}^1$-family of modules over the $A_7$ preprojective algebra given in figure \ref{cyclic example} below.  (i) determines coordinates of the $\mathbb{P}^1$-family from an impression of the $A_7$ preprojective algebra, and hence coordinates related to the singularity $\mathbb{C}^2/\rho(\mu_7)$, while (ii) specifies the $\mathbb{P}^1$-family (Definition \ref{P^n family}) and is necessary for the intersections of the different $\mathbb{P}^1$-families to be parameterized by the intersections of the corresponding $\mathbb{P}^1$'s in the minimal resolution of $\mathbb{C}^2/\rho(\mu_7)$.
}
\label{}
\end{figure}

\section{Resolving singularities} \label{Res sing}

In this section we verify Conjecture \ref{conjecture} in a number of examples.  In these examples the noncommutative algebra is the path algebra of a McKay quiver, modulo relations.  The McKay quiver $Q$ of a group $G$ and representation $\rho: G \rightarrow \operatorname{GL}_n(\mathbb{C})$ is defined to have a vertex for each irreducible representation $\phi_0, \phi_1, \ldots, \phi_m$ of $G$, and an arrow from $j$ to $i$ for each direct summand of $\phi_j$ in $\rho \otimes_{\mathbb{C}} \phi_i$.  In the special cases $\rho: G \rightarrow \operatorname{SL}_2(\mathbb{C})$, $Q$ is the double of any quiver whose underlying graph is the extended Dynkin graph of $G$, and McKay observed that this is the dual graph of the exceptional locus of the minimal resolution of $\mathbb{C}^2/\rho(G)$.  Our program extends this correspondence by realizing the vertex simples at the vertices of the McKay quiver as the respective irreducible components of the exceptional locus shrunk to (smooth) point-like spheres. 

\subsection{The conifold} \label{The conifold}

The well-known quiver algebra for the conifold (quadric cone) $R:=\mathbb{C}\left[ xz, xw, yz, yw \right] \cong \mathbb{C}\left[ s,t,u,v \right]/(sv-tu)$ is
$$A:=\mathbb{C}Q/\left\langle a_ib_ja_k-a_kb_ja_i, b_ia_jb_k-b_ka_jb_i \ | \ i,j,k = 1,2 \right\rangle$$
with quiver given in figure \ref{conifold}.i.  Since $A$ is a square superpotential algebra, by \cite[Theorem 3.7]{B} $A$ admits an impression $\left(\tau, \mathbb{C}\left[x,y,z,w \right] \right)$, where $\tau$ is defined by the labeling of arrows in figure \ref{conifold}.ii, namely,
\begin{equation} \label{conifold impression}
\begin{array}{c}
\tau(a_1) = \left[ \begin{array}{cc} 0 & 0 \\ x & 0 \end{array} \right], \ \ \
\tau(a_2) = \left[ \begin{array}{cc} 0 & 0 \\ y & 0 \end{array} \right], \ \ \
\tau(b_1) = \left[ \begin{array}{cc} 0 & z \\ 0 & 0 \end{array} \right], \ \ \
\tau(b_2) = \left[ \begin{array}{cc} 0 & w \\ 0 & 0 \end{array} \right], \\ \\
\tau(e_1) = \left[ \begin{array}{cc} 1 & 0 \\ 0 & 0 \end{array} \right], \ \ \ \
\tau(e_2) = \left[ \begin{array}{cc} 0 & 0 \\ 0 & 1 \end{array} \right].
\end{array}
\end{equation}
The center $Z$ of $A$ is isomorphic to $R$, and the non-Azumaya locus of $A$ is the unique singular point $0 \in \operatorname{Max}R$ \cite[Theorem 6.5]{B}.  $\operatorname{Max}R$ admits two crepant resolutions $\pi_{\pm}: Y^{\pm} \rightarrow \operatorname{Max}R$ given by the two birational transforms (with $s' = 1$),
$$sv- tu = s(v-t'u), \ \ \ \ \ \ \ sv-tu = s(v-tu').$$
The exceptional locus $\pi^{-1}(0)$ is given by $v-t'u=0$ (resp.\ $v-tu'=0$) with $s=t=u=v=0$, so since 
$s' (xw)= s't = st' = (xz) t'$ (resp.\ $s' (yz) = s'u = su'= (xz) u'$), the ratios $t'/s' = w/z$ (resp.\ $u'/s'=y/x$) are free to vary.  Thus in terms of the original coordinates $x,y,z,w$, $\pi_+^{-1}(0)= \mathbb{P}^1$ has coordinates $(z:w)$, while $\pi_-^{-1}(0) = \mathbb{P}^1$ has coordinates $(x:y)$.  We now show that these coordinates agree with those obtained from the almost large $A$-modules.
\begin{figure}
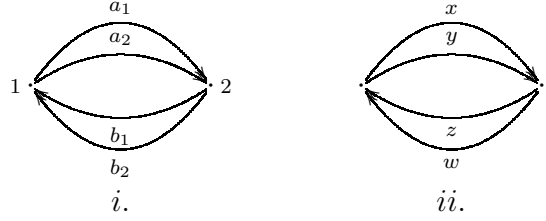

$$\begin{array}{ccc}
\xy
(-12,0)*{\cdot}="1";(12,0)*{\cdot}="2";
{\ar@/^2pc/^{a_1}"1";"2"};{\ar@/^1pc/^{a_2}"1";"2"};
{\ar@/^1pc/^{b_1}"2";"1"};{\ar@/^2pc/^{b_2}"2";"1"};
(-14,0)*{}="1'";(14,0)*{}="2'";
{\ar@{}|-1"1'";"1'"};{\ar@{}|-2"2'";"2'"};
\endxy
& \ \ \ \ \ \ \ & 
\xy
(-12,0)*{\cdot}="1";(12,0)*{\cdot}="2";
{\ar@/^2pc/^x"1";"2"};{\ar@/^1pc/^y"1";"2"};
{\ar@/^1pc/^z"2";"1"};{\ar@/^2pc/^w"2";"1"};
\endxy
\\
i. & & ii.
\end{array}$$
\caption{The conifold quiver and its impression.} \label{conifold}
\end{figure}

\begin{Proposition}
Let $A$ be the conifold quiver algebra.  Then the large $A$-module isoclasses are parameterized by the smooth locus of $\operatorname{Max}R$, while the almost large module isoclasses with socle $S_2$ (resp.\ $S_1$) are parameterized by the exceptional locus $\pi_-^{-1}(0) = \mathbb{P}^1$ (resp.\ $\pi_+^{-1}(0)$), having coordinates $(x:y)$ (resp.\ $(z:w)$).  Moreover, the coordinates on $Y^{\pm}$ obtained from the impression $(\tau,\mathbb{C}[x,y,z,w])$, namely (\ref{O}), agree with those obtained by blowing up.
\end{Proposition}

\begin{proof} 
The fact that the large modules are parameterized by the smooth locus follows from \cite[Theorem 6.5]{B}.  By \cite[Theorem 3.7]{B} the large modules have dimension vector $(1,1)$.  Denote by $\mathcal{P}$ the path-like set $Q_{\geq 0} \cup \{0\}$.  As in the proof of Proposition \ref{first proposition}, for $w_1, \ldots, w_j \in \left\{ y, z_1, \ldots, z_n \right\}$ let $P(w_1 = \cdots = w_j = 0)$ denote the $\mathcal{P}$-annihilator of a module in $\operatorname{Rep}_{(1,1)}A$ with $w_1 = \cdots = w_j = 0$ and all other arrows represented by nonzero scalars.  Noting that $\operatorname{dim}Z = 3$, there is then a maximal chain as in Definition \ref{almost large},
$$0 \subsetneq P_1(z = 0 ) \subsetneq P_2(z= w= 0) \subsetneq P_3(z = w = x = 0),$$
so any module with $\mathcal{P}$-annihilator $P(z= w=0)$, $P(z= w=x=0)$, or $P(z=w=y=0)$ is almost large with socle $S_2$, and similarly any module with $\mathcal{P}$-annihilator $P(x= y=0)$, $P(x= y=z=0)$, or $P(x=y=w=0)$ is almost large with socle $S_1$.  These two families of almost large modules form $\mathbb{P}^1$-families (recall the module isomorphism (ii) in figure \ref{impression iso}), with respective coordinates $(x:y)$ and $(z:w)$ determined from (\ref{O}) and the impression of $A$ given by (\ref{conifold impression}); see figure \ref{conifold resolution}.  The path-like set $\mathcal{P} = Q_{\geq 0} \cup \{0\}$ is sufficient for determining all almost large modules since the almost large modules with socle $S_1$ or $S_2$ obtained from $Q_{\geq 0} \cup \{0\}$ exhaust the set of all modules in $\operatorname{Rep}_{(1,1)}A$ with socle $S_1$ or $S_2$.
\begin{figure}
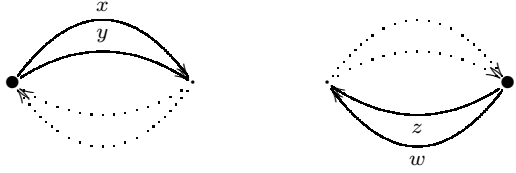

$$\begin{array}{ccc}
\xy
(-12,0)*{\bullet}="1";(12,0)*{\cdot}="2";
{\ar@/^2pc/^x"1";"2"};{\ar@/^1pc/^y"1";"2"};
{\ar@{..>}@/^2pc/"2";"1"};{\ar@{..>}@/^1pc/"2";"1"};
\endxy
& \ \ \ \ \ \ \ & 
\xy
(-12,0)*{\cdot}="1";(12,0)*{\bullet}="2";
{\ar@{..>}@/^2pc/"1";"2"};{\ar@{..>}@/^1pc/"1";"2"};
{\ar@/^1pc/^z"2";"1"};{\ar@/^2pc/^w"2";"1"};
\endxy
\end{array}$$
\caption{The exceptional loci of the two crepant resolutions of the conifold.  Each $\mathbb{P}^1$ shrinks to the vertex simple at the bold vertex.} \label{conifold resolution}
\end{figure}
\end{proof}

\subsection{Cyclic quotient surface singularities} \label{Cyclic quotient surface singularities}

Consider the linear action of the finite abelian group $\mu_r = \left\langle g \right\rangle$ of order $r$ on $\mathbb{C}[x,y]$ by the representation
$$\rho(g) = \left[ \begin{array}{cc} e^{2 \pi i/r} & 0 \\ 0 & e^{2 \pi i b/r} \end{array} \right],$$
that is, $g \cdot (x,y) = \left( e^{2 \pi i/r}x, e^{2 \pi i b/r}y \right)$.  The ring of invariants $R := \mathbb{C}[x,y]^{\rho(\mu_r)}$ is the coordinate ring for the cyclic quotient surface singularity $\mathbb{C}^2/\rho(\mu_r) := \operatorname{Max}R$ of type $\frac 1r (1,b)$.  We suppose $\mu_r$ acts freely on $\mathbb{C}^2 \setminus 0$, and so we take $\operatorname{gcd}(r,b)=1$, thus neglecting quasi-reflections.  We will find that the minimal resolution $Y \rightarrow \mathbb{C}^2/\rho(\mu_r)$ of such a singularity (the total number of irreducible components of the exceptional locus and the coordinates on each component) can be read off directly from the associated McKay quiver: this information is simply hidden within the quiver, and is extracted by determining the supporting subquivers of the almost large modules over the McKay quiver algebra. 

\begin{Lemma} \label{cyclic impression} 
Let $Q$ be the McKay quiver of $(\mu_r, \rho)$, so for each $i \in Q_0 = \{1, \ldots, r\}$ there are arrows
$$e_i \stackrel{a_i}{\longrightarrow} e_{i+1}, \ \ \ \ e_i \stackrel{b_i}{\longrightarrow}e_{i+b}.$$
Then the associated McKay quiver algebra
$$A:=\mathbb{C}Q/\left\langle b_{i+1}a_i- a_{i+r}b_i \ | \ i \in Q_0 \right\rangle$$
admits an impression $\left(\tau, \mathbb{C} \left[ x,y \right] \right)$, where $\tau$ is defined by the labeling 
\begin{equation} \label{cyclic label}
\bar{\tau}(a_i)=x, \ \ \ \ \bar{\tau}(b_i)=y,
\end{equation}
for each $i \in Q_0$.
\end{Lemma}

\begin{proof}
Since the corner rings $e_iAe_i$ are commutative, the algebra homomorphism $\tau: A \rightarrow \operatorname{End}_{\mathbb{C}[x,y]}(\mathbb{C}[x,y])$ defined by (\ref{cyclic label}) is a monomorphism.  Thus the large $A$-modules have dimension vector $(1, \ldots, 1)$ by \cite[Proposition 2.9]{B}.  Take $U = \mathbb{C}^2 \setminus 0$.  Since $r,b$ are coprime, $V_{\tau_m}$ will be a large module for each $m \in U$.  Since $Z \cong \mathbb{C}[x,y]^{\mu_r}$, the canonical morphism $\phi: \operatorname{Max}B \rightarrow \operatorname{Max}Z$ is a surjection.
\end{proof}

The following theorem extends the fact that the large $A$-modules are parameterized by the smooth locus of $\mathbb{C}^2/\rho(\mu_r)$.

\begin{Theorem} \label{cyclic theorem} Let $\mathbb{C}^2/\rho(\mu_r)$ be a cyclic quotient surface singularity, $Y \rightarrow \mathbb{C}^2/\rho(\mu_r)$ its minimal (Hirzebruch-Jung) resolution, and $A$ the associated McKay quiver algebra.  Then for each $i \in Q_0$, the set of almost large modules with socle $S_i$ are parameterized by the exceptional locus of $Y$.  Moreover, the coordinates on $Y$ obtained from the impression $(\tau,\mathbb{C}[x,y])$, namely (\ref{O}), agree with those obtained from the Hirzebruch-Jung resolution.
\end{Theorem}

\begin{proof} As noted above, the large modules have dimension vector $(1, \ldots, 1)$, so we may fix any vertex $0 \in Q_0 = \{0, \ldots, r-1 \}$ and consider the isoclasses of almost large modules with socle isomorphic to the vertex simple $S_0$.

Let $L$ denote the lattice $\mathbb{Z}^2 + \mathbb{Z}\cdot \frac 1r (1,b) \subset \mathbb{R}^2$.  For $m \in \{1, \ldots, r-1\}$, let $p \in e_0Q_{\geq 1}$ the unique path satisfying $\bar{\tau}(p)=x^m$, that is, $p = a_1a_2\cdots a_m$, and let $q \in e_0Q_{\geq 1}e_{\operatorname{t}(p)}$ be the unique path satisfying $\bar{\tau}(q)=y^n$ for some $n \in \{ 1, \ldots r-1 \}$.  Then $m = \overline{nb}$, so $\frac 1r (n,m) = \frac 1r \left(n, \overline{nb} \right) \in L$ is in the unit square of $\mathbb{R}^2$.

Let $Q^m \subset Q$ be the subquiver defined by
$$Q_i^m  = \left\{ a \in Q_i \ | \ a \text{ is a subpath of } p \text{ or } q \right\}, \ \ \ i = 0,1.$$
Note that $0 \in Q^m_0$ is a sink for $Q^m$ and 
$$j:= \operatorname{t}(p) = \operatorname{t}(q) \in Q^m_0$$
is a source (denoted by the bold vertices in figure \ref{cyclic example}).

Consider two subquivers $Q^{m}$ and $Q^{m'}$ of $Q$ where $m = \overline{nb}$ and $m'= \overline{n'b}$ with $1 \leq n,n' \leq r-1$.  
If $n < n'$ and $m' < m$ then clearly $Q^m \not \subset Q^{m'}$ and $Q^{m'} \not \subset Q^m$.
Now the boundary lattice points of the convex hull of $L \subset \mathbb{R}^2$ in the positive quadrant, excluding the origin, are precisely the points $\frac 1r (n',m')$ for which $n < n'$ implies $m' < m$ (and these points are in 1-1 correspondence with the irreducible components of the exceptional locus in $Y$).  There is thus a 1-1 correspondence between the maximal chains of subquivers
\begin{equation} \label{maximal chain} 
Q^{m_1} \subsetneq Q^{m_2} \subsetneq \cdots \subsetneq Q^{m_{\ell}}
\end{equation}
and the boundary lattice points.

Now let $Q^m$ be the minimal term in a maximal chain (\ref{maximal chain}), and construct the subquiver $\widetilde{Q}^m \supset Q^m$ of $Q$ by adding the arrows $a_i$ and $b_i$ to $Q^m$ for each $i \not \in Q^m$ (these are denoted by the dotted arrows in figure \ref{cyclic example}).

Since $\operatorname{dim}Z = 2$, we must determine all maximal chains $0 \subsetneq P_1 \subsetneq P_2$ as in Definition \ref{almost large}.  Denote by $\mathcal{P}$ the path-like set $Q_{\geq 0} \cup \{ 0 \}$.\\
\\
(i) \textit{If $Q^m$ is the minimal term in a maximal chain (\ref{maximal chain}), then $p$ and $q$ cannot have a common vertex subpath $e_k$ different from the sink and source of $Q^{m}$, namely $e_0$ and $e_j$.}  Suppose otherwise; let $p_1$ and $q_1$ be the (unique) subpaths of $p$ and $q$ respectively, satisfying $p_1,q_1 \in e_kQ_{\geq 1}e_j$.  Then there are subpaths $p_2$ and $q_2$ of $p$ and $q$ such that $p_2,q_2 \in e_1Q_{\geq 1}e_{\operatorname{t}(p_2)}$.  The subquiver corresponding to $p_2$ and $q_2$ is then a subquiver of $Q^{m}$, contracting the minimality of $Q^{m}$ in a maximal chain.  It follows that $p$ and $q$ have no common vertex subpaths other than the source and sink of $Q^{m}$.  $\Box$\\
\\
(ii) \textit{$\widetilde{Q}^m$ supports an $A$-module with dimension vector $(1,\ldots,1)$.}  It suffices to show that if $a_i, b_{\operatorname{h}(a_i)} \in \widetilde{Q}^m_1$ then $b_i, a_{\operatorname{h}(b_i)} \in \widetilde{Q}^m_1$ as well, since the relation $b_{\operatorname{h}(a_i)}a_i = a_{\operatorname{h}(b_i)}b_i$ must hold.  

If $i \in Q^m_0$ and $a_i \in \widetilde{Q}^m_1$ then $b_{\operatorname{h}(a_i)} \not \in \widetilde{Q}^m_1$ by (i), so it must be that $i \not \in Q^m_0$.  But then $b_i \in \widetilde{Q}^m_1$ by construction of $\widetilde{Q}^m$, so we just need to show that $a_{\operatorname{h}(b_i)} \in \widetilde{Q}^m_1$ as well.  If $\operatorname{h}(b_i) \not \in Q^m_0$ then $a_{\operatorname{h}(b_i)} \in \widetilde{Q}^m_1$, again by construction.  Otherwise suppose $\operatorname{h}(b_i) \in Q^m_0$.  Since $i \not \in Q^m_0$, $b_i \not \in Q^m_1$, so $e_{\operatorname{h}(b_i)}$ cannot be a subpath of $q$ different from $e_j$ since there is only one $b$ arrow whose head is at any given vertex, and thus $e_{\operatorname{h}(b_i)}$ must be a subpath of $p$.  Moreover, $\operatorname{h}(b_i) \not = 0$ since $q$ contains the $b$ arrow whose head is at $0$.  But then $a_{\operatorname{h}(b_i)} \in Q^m_1 \subseteq \widetilde{Q}^m_1$, proving our claim.  $\Box$\\
\\
(iii) \textit{Any module $V \in \operatorname{Rep}_{(1,\ldots,1)}A$ supported on $\widetilde{Q}^m$ has socle $S_0$ and therefore is not simple.}  Since $0 \in \widetilde{Q}^m_0$, $a_0$ and $b_0$ will not be added to $Q^m$ to form $\widetilde{Q}^m_0$, and so $0$ is a sink in $\widetilde{Q}^m$.  It therefore suffices to show that for each $i \in \widetilde{Q}_0^m = Q_0$ there is a path $s$ in $Q$ from $i$ to $0$ that is contained in $\widetilde{Q}^m$ (that is, $s$ does not annihilate $V$) since the dimension vector of $V$ is $(1, \ldots, 1)$.   We claim there exists a path $s = r a_{k_t} \cdots a_{k_2}a_{k_1}$, where $r$ is a subpath of $p$ or $q$ with head at $0$.  

For $1 \leq u \leq t$, if $\operatorname{h}(a_{k_u}) \in Q^m_0$ then $u = t$; otherwise $\operatorname{h}(a_{k_u}) \not \in Q^m_0$, in which case there exists arrows $a_{\operatorname{h}(a_{k_u})}$ and $b_{\operatorname{h}(a_{k_u})}$ in $\widetilde{Q}^m_1$ by construction, so $a_{k_{u+1}}a_{k_u}$ is a path in $\widetilde{Q}^m$.  Now $a_{k_{u+1}}$ cannot be a subpath of $a_{k_{u}} \cdots a_{k_1}$ since $r$ and $b$ are coprime and $0$ is a sink, and it follows that $t$ must exist since the number of vertices is finite.  $\Box$\\
\\
(iv) \textit{If $Q'$ supports an $A$-module with dimension vector $(1,\ldots,1)$ and $\widetilde{Q}^m \subsetneq Q' \subseteq Q$ then $Q' = Q$.}  Suppose $b_i \in Q'_1 \setminus \widetilde{Q}^m_1$.  Then $i \in \widetilde{Q}^m_0 \setminus \{ j  \}$, specifically $i$ is a vertex subpath of $p = a_1a_2 \cdots a_m$.

Now if $\operatorname{h}(b_i) \not \in Q^m_0$ then $a_{\operatorname{h}(b_i)} \in \widetilde{Q}^m_1 \subset Q'_1$, while if $\operatorname{h}(b_i) \in Q^m_0$ then $a_{\operatorname{h}(b_i)} \in Q^m_1 \subseteq Q'_1$ (otherwise the head of $b_i$ would coincide with the head of a $b$ arrow in $Q^m_1$, and since there is precisely one $b$ arrow with head at any given vertex then $b_i$ would be in $Q^m_1$, contrary to our original assumption).  Therefore in either case $a_{\operatorname{h}(b_i)} \in Q'_1$.  Since $Q'$ supports an $A$-module and $a_{\operatorname{h}(b_i)}$ and $b_i$ are both in $Q'_1$, the relation
$$a_{\operatorname{h}(b_i)}b_i = b_{\operatorname{h}(a_i)}a_i$$
implies $b_{\operatorname{h}(a_i)}$ is also in $Q'_1$.  We can apply this argument iteratively (next with $b_{\operatorname{h}(a_i)}$ in place of $b_i$) to deduce that
$$b_{\operatorname{h}(a_1 a_2 \cdots a_{i-1}a_i)} = b_0$$
is in $Q_1'$.  A similar argument with the $a$ arrows then implies $Q'_1 = Q_1$, and hence $Q' = Q$. $\Box$\\
\\
(v) \textit{Any module $V \in \operatorname{Rep}_{(1,\ldots,1)}A$ supported on $\widetilde{Q}^m$ is an $\ell_{\mathcal{P}} = 1$ almost large module.}  By (ii) $\widetilde{Q}^m$ supports an $A$-module; by (iii) $V$ is not simple; and by (iv) the chain $0 \subsetneq P_1$, where $P_1$ is the $\mathcal{P}$-annihilator of $V$, is maximal. $\Box$\\
\\
(vi) \textit{$\widetilde{Q}^m$ supports a $\mathbb{P}^1$-family, minus the two points where one of the coordinates is zero.}  By (i) $p$ and $q$ have no common vertex subpaths and so clearly $Q^m$ supports a $\mathbb{P}^1$-family (minus two points); see the upper diagram in figure \ref{P^1 family}.  By (iii) any module $V$ supported on $\widetilde{Q}^m$ will have socle $S_0$, and so together with the ``commutation'' relations from $I$ this implies that $V$ is isomorphic to a module in which all the $a$ arrows in $\widetilde{Q}^m_1$ are represented by the same scalar, and all the $b$ arrows are represented by the same scalar.  The claim then follows since the subquiver $Q^m$ of $\widetilde{Q}^m$ supports a $\mathbb{P}^1$-family. $\Box$\\
\\
(vii) \textit{If $Q'$ supports an $\ell_{\mathcal{P}} = 1$ almost large module with socle $S_0$ then $Q' = \widetilde{Q}^m$ for some $m$.}  By our assumptions on $Q'$, $Q'$ must contain as a subquiver a minimal term $Q^{m}$ in some maximal chain (\ref{maximal chain}), and by (iv) we may assume that $Q'$ does not properly contain $\widetilde{Q}^m$, for otherwise it would equal $Q$.  In addition, by assuming $\ell_{\mathcal{P}} = 1$, $Q'$ cannot be properly contained in $\widetilde{Q}^m$.  Suppose that $a_i \in Q'_1 \setminus Q_1^m$, where $i \not = j$ is a vertex subpath of $q$.  By the argument in (iv), $a_0$ must then also be in $Q_1'$, and so the socle of any module supported on $Q'_1$ would not be $S_0$.  Similarly $b_i \not \in Q_1'$ if $i \not = j$ is a vertex subpath of $p$.  Thus if $a_i$ or $b_i$ is in $Q'_1 \setminus Q_1^m$ then $i$ must be not be in $Q^m_0$, so by the construction of $\widetilde{Q}^m$ we have $Q'_1 \subseteq \widetilde{Q}^m_1$ and hence $Q' = \widetilde{Q}^m$. $\Box$\\
\\
\indent We have now characterized the $\ell_{\mathcal{P}} = 1$ almost large module isoclasses with socle $S_0$, and now we characterize the $\ell_{\mathcal{P}} = 2$ almost large modules.

Set
$$\begin{array}{rcl}
\alpha & := & \left\{ a_i \in \widetilde{Q}^m_1 \ | \ i = j \ \text{ or } \ b_k \cdots b_{\operatorname{h}(b_i)} b_i \in e_j\widetilde{Q}^m_{\geq 1} \right\},
\\ \\
\beta & := & \left\{ b_i \in \widetilde{Q}^m_1 \ | \ i = j \ \text{ or } \ a_k \cdots a_{\operatorname{h}(a_i)} a_i \in e_j \widetilde{Q}^m_{\geq 1} \right\}.
\end{array}$$
Consider the subquivers $\widetilde{Q}^{m,a}$ and $\widetilde{Q}^{m,b}$ of $\widetilde{Q}^m$ with vertex sets $Q_0$ and arrow sets $\widetilde{Q}^m_1 \setminus \alpha$ and $\widetilde{Q}^m_1 \setminus \beta$, respectively. \\
\\
(viii) \textit{The subquivers $\widetilde{Q}^{m,a}$ and $\widetilde{Q}^{m,b}$ support $A$-modules with dimension vector $(1,\ldots,1)$.}  Let $\rho \in \operatorname{Rep}_{(1,\ldots,1)}A$ be a representation supported on $\widetilde{Q}^{m,a}$, and suppose $a_i \in \alpha$, so $\rho(a_i) = 0$.  It suffices to show that the relations 
\begin{equation} \label{star}
\rho(a_{\operatorname{h}(b_i)}b_i ) = \rho( b_{\operatorname{h}(a_i)}a_i) = 0
\end{equation}
and
\begin{equation} \label{starstar}
\rho(b_{\operatorname{h}(a_t)}a_t ) = \rho(a_ib_t) = 0
\end{equation}
hold, where $\operatorname{h}(b_t) = i$.

In the first case, if $i = j$ then by (i) $a_{\operatorname{h}(b_i)} \not \in Q^m_1$, hence $a_{\operatorname{h}(b_i)} \not \in \widetilde{Q}^m_1$ since $\operatorname{h}(b_i) \in Q^m_0$, so (\ref{star}) holds.  

If $i \not = j$ then $a_i \in \alpha$ implies $b_k \cdots b_{\operatorname{h}(b_i)}b_i \in e_j\widetilde{Q}^m_{\geq 1}$.  If the length of the path $b_k \cdots b_i$ is 1 (so the path is really just $b_i$), then $\operatorname{h}(b_i) = j$, so $a_{\operatorname{h}(b_i)} = a_j \in \alpha$, so (\ref{star}) holds.  Otherwise if the length of the path $b_k \cdots b_i$ is at least 2 then $b_k \cdots b_{\operatorname{h}(b_i)} \in e_j \widetilde{Q}^m_{\geq 1}$ as well, in which case $a_{\operatorname{h}(b_i)} \in \alpha$, hence (\ref{star}) holds.

Now in the second case, first suppose $b_t \in \widetilde{Q}^m_1$.  If $i = j$ then $b_t \in e_j \widetilde{Q}^m_1$, so $a_t \in \alpha$, hence (\ref{starstar}) holds.  If $i \not = j$ then $a_i \in \alpha$ implies $b_k \cdots b_i \in e_j \widetilde{Q}^m_{\geq 1}$, hence $b_k \cdots b_i b_t \in e_j \widetilde{Q}^m_{\geq 1}$ since $\operatorname{h}(b_t) = i$, and so $a_t \in \alpha$, hence (\ref{starstar}) holds.  

Otherwise suppose $b_t \not \in \widetilde{Q}^m_1$.  Then $t \not \in Q^m_0$, so it must be that $a_t \in Q^m_1$, and by (i), $b_{\operatorname{h}(a_t)} \not \in Q^m_1$, hence $b_{\operatorname{h}(a_t)} \not \in \widetilde{Q}^m_1$ since $\operatorname{h}(a_t) \in Q^m_0$, and so (\ref{starstar}) holds.

We have shown that in all cases the relations (\ref{star}) and (\ref{starstar}) are satisfied, so $Q^{m,a}$ supports an $A$-module, and similarly for $Q^{m,b}$.  $\Box$\\
\\
(ix) \textit{Any module in $\operatorname{Rep}_{(1,\ldots,1)}A$ supported on $Q^{m,a}$ or $Q^{m,b}$ has socle $S_0$.}  Since $0$ is a sink in $\widetilde{Q}^m$ by (iii), it is also a sink in $\widetilde{Q}^{m,a}$, and so it suffices to show that for any vertex $k \in \widetilde{Q}^{m,a}_0 = Q_0$ there exists a path $s$ from $k$ to $0$ in $\widetilde{Q}^{m,a}$.  No $b$ arrows are removed from $\widetilde{Q}^m$ to form $\widetilde{Q}^{m,a}$, and so by (iii) we may take $s$ to be $rb_{j_t} \cdots b_{j_2}b_{j_1}$, where $r$ is a subpath of $p$ or $q$ with head at $0$ ($q$ if $\operatorname{h}(b_{j_t}) = j$).  We therefore only need to show that if $a_i \in \alpha$ and $i \not = j$, then $a_i$ is not a subpath of $p$.  Suppose otherwise; since $a_i \in \alpha$ and $i \not = j$, $b_k \cdots b_i$ is a path in $\widetilde{Q}^m$, so $b_i$ is a path in $\widetilde{Q}^m$, and hence a path in $Q^m$ since $i \in Q^m_0$, which is a contradiction by (i).  $\Box$\\
\\
(x) \textit{If $Q'$ supports an $\ell_{\mathcal{P}} = 2$ almost large module with socle $S_0$, then $Q' = \widetilde{Q}^{m,a}$ or $Q' = \widetilde{Q}^{m,b}$.}  Suppose $b_i \not \in Q'_1$; then we claim that $b_j$ is also not in $Q'_1$, hence $Q' \subseteq \widetilde{Q}^{m,b}$ since $Q'$ supports an $A$-module, so $Q' = \widetilde{Q}^{m,b}$ since $Q'$ supports an $\ell_{\mathcal{P}} = 2$ almost large module.  First suppose $i \in Q^m_0 \setminus \{ j \}$.  Since $Q' \subset \widetilde{Q}^m$ and (i) holds, the vertex $i$ would be a sink of $Q'$, and so $S_i$ would be a direct summand of the socle of any module supported on $Q'$, contrary to our assumption.  So suppose $i \not \in Q^m_0$.  Since $Q'$ supports a module with socle $S_0$, there exists a path $r$ from $\operatorname{h}(b_i)$ to $0$.  Since $Q' \subset \widetilde{Q}^m$ and (i) holds, $r = r'p$ or $r = r'q$ for some path $r'$ in $Q'$ from $\operatorname{h}(b_i)$ to $j$.  Thus for any $\rho \in \operatorname{Rep}_{(1,\ldots,1)}A$ supported on $Q'$, $\rho(b_jr') = \rho(r''b_i) = 0$ for some path $r''$ in $Q$.  But $\rho(r') \not = 0$ since $r'$ is a path in $Q'$, so it must be that $\rho(b_j) = 0$, hence $b_j \not \in Q'_1$, proving our claim.  $\Box$\\
\\
(xi) \textit{There is only one module in $\operatorname{Rep}_{(1,\ldots,1)}A$ supported respectively on $Q^{m,a}$ and $Q^{m,b}$, up to isomorphism.}  Suppose to the contrary that the underlying graph of $\widetilde{Q}^{m,a}$ contains a cycle.  Since $a_0, b_0 \not \in \widetilde{Q}^m_1$, $\widetilde{Q}^{m,a}$ contains no oriented cycles, so there must be a vertex $i$ for which both $a_i$ and $b_i$ are in $\widetilde{Q}^{m,a}_1$.  Since $\widetilde{Q}^{m,a}$ supports a representation $\rho \in \operatorname{Rep}_{(1,\ldots,1)}A$ with socle $S_0$ by (ix), there exists a path $r$ from $\operatorname{h}(b_i)$ to $j$ in $\widetilde{Q}^{m,a}$ (recall the proof of (x)).  But $a_j \not \in \widetilde{Q}^{m,a}$ implies $\rho(b_jra_i) = \rho(a_jrb_i) = 0$, so $\rho(a_i)=0$ since $\rho(b_jr) \not = 0$, hence $a_i \not \in \widetilde{Q}^{m,a}$, a contradiction.  Similarly the underlying graph of $Q^{m,b}$ contains no cycles.  The claim then follows since we are considering modules with dimension vector $(1,\ldots,1)$.  $\Box$\\
\\
(xii) \textit{There is an equality of subquivers $Q^{m_i,b} = Q^{m_{i+1},a}$.}  Denote by $j_i$, $p_i$, and $q_i$, the source $j$ and respective paths $p$ and $q$ of $Q^{m_i}$.  Suppose to the contrary that $b_{j_i} \in \widetilde{Q}^{m_{i+1}}_1$.  Since $p_i$ is a subpath of $p_{i+1}$, it follows that $a_{j_i} \in Q^{m_{i+1}}_1$, hence $a_{j_i} \in \widetilde{Q}^{m_{i+1}}_1$ as well, so it must be that $j_i \not \in Q^{m_{i+1}}_0$ by (i) since then both $a_{j_i}$ and $b_{j_i}$ are in $Q^{m_{i+1}}_1$ and $j_i \not = j_{i+1}$.  But $a_{j_i} \in Q^{m_{i+1}}_1$ implies $j_i \in Q^{m_{i+1}}_0$, a contradiction.

Similarly suppose to the contrary that $a_{j_{i+1}} \in \widetilde{Q}^{m_i}_1$.  Since $p_i$ is a subpath of $p_{i+1}$, by (i) it must be that $q_{i+1}$ is a subpath of $q_i$.  Therefore $b_{j_{i+1}} \in Q^{m_i}_1$, hence $b_{j_{i+1}} \in \widetilde{Q}^{m_i}_1$ as well, so it must be that $j_{i+1} \not \in Q^{m_i}_0$ by (i) and $j_{i+1} \not = j_i$.  But $b_{j_{i+1}} \in Q^{m_i}_1$ implies $j_{i+1} \in Q^{m_i}_0$, a contradiction.

Since $b_{j_i} \not \in \widetilde{Q}^{m_{i+1}}_1$ and $\widetilde{Q}^{m_{i+1}}$ supports an $A$-module, we have $\widetilde{Q}^{m_{i+1}} \subseteq \widetilde{Q}^{m_i}$.  Similarly since $a_{j_{i+1}} \not \in \widetilde{Q}^{m_i}$ we have $\widetilde{Q}^{m_i} \subseteq \widetilde{Q}^{m_{i+1}}$, proving our claim.  $\Box$\\
\\
(xiii) \textit{If $V \in \operatorname{Rep}_{(1,\ldots,1)}A$ has socle $S_0$ then $V$ is an almost large module.}  Suppose $Q'$ supports $V$.  Since $Q'$ supports a module with dimension vector $(1,\ldots,1)$ and socle $S_0$, for each $i \in Q_0$ there must be a path in $Q'$ from $i$ to $0$.  Therefore there is some $m$ such that $Q^m_1 \setminus \{a_j\}$ or $Q^m_1 \setminus \{b_j\}$ is a subset of $Q'_1$.

First suppose $Q^m$ is a subquiver of $Q'$, and suppose $a_i \in Q'_1 \setminus Q^m_1$.  Then since $Q'$ supports an $A$-module with socle $S_0$, there is a path $p$ from $i$ to $j$, say $p = a_i p'$ with $p'$ a path.  By the relations of $A$, $ap'b_j = b_ip''$ for some path $p''$, so $b_i \in Q'_1$.  Since $i \not \in Q^m_0$ was arbitrary, $Q' = \widetilde{Q}^m$ by the construction of $\widetilde{Q}^m$.

Now suppose $Q^m_1 \setminus \{a_j\}$ is a subset of $Q'_1$, but $Q^m_1$ is not.  By the proof of (x), $Q' \subseteq \widetilde{Q}^{m,a}$.  Suppose to the contrary that this containment is proper, that is, there is some $a_i$ or $b_i$ in $\widetilde{Q}^{m,a}_1 \setminus Q'_1$ with $i \not \in Q^m_0$.  If $a_i \not \in Q'_1$ then there must exist a path consisting entirely of $b$ arrows from $i$ to $0$, for if $p$ is a path from $i$ to $0$ containing an $a$ arrow then by the relations of $A$, $p = a_ip'$ for some path $p'$.  But by the construction of $\widetilde{Q}^{m,a}$, there is no path consisting entirely of $b$ arrows from $i$ to $0$ since $a_i \in \widetilde{Q}^{m,a}_1$.  Similarly, if $b_i \not \in Q'_1$ then there must exist a path consisting entirely of $a$ arrows from $i$ to $0$ in $Q'$.  But there is no such path in $\widetilde{Q}^{m,a}$ since such a path would necessarily contain $a_j$, and $a_j \not \in \widetilde{Q}^{m,a}_1$, so $a_j \not \in Q'_1$ as well.  Thus the containment cannot be proper and so $Q' = \widetilde{Q}^{m,a}$.  The case where $Q^m_1 \setminus \{b_j\}$ is a subset of $Q'_1$ is similar.  $\Box$\\
\\
\indent The path-like set $\mathcal{P} = Q_{\geq 0} \cup \{0\}$ is sufficient for determining all almost large modules since the almost large modules with socle $S_0$ obtained from $Q_{\geq 0} \cup \{0\}$ exhaust the set of all modules in $\operatorname{Rep}_{(1,\ldots,1)}A$ with socle $S_0$ by (xiii).  By (x) and (xi), the $\ell_{\mathcal{P}} = 2$ almost large module isoclasses with socle $S_0$ are parameterized by the points in the $\mathbb{P}^1$-families where one coordinate is zero, namely $(0:1)$ or $(1:0)$.  Together with (vi) and (vii), it follows that there is a 1-1 correspondence between the supporting subquivers of almost large modules with socle $S_0$ (each of which supports a $\mathbb{P}^1$-family) and the boundary lattice points of $L$, and hence the irreducible components of the exceptional locus, each of which is a $\mathbb{P}^1$ \cite[Proposition 2.2, Theorem 3.2]{Reid}.  Furthermore, by (xii) the intersections of the irreducible components parameterize the intersections of the $\mathbb{P}^1$-families of almost large modules.

Finally, if $Q^{m,n}$ is the minimal term in the chain (\ref{maximal chain}), then $Q^{m,n}$, and hence $\widetilde{Q}^{m,n}$, supports a $\mathbb{P}^1$-family with homogeneous coordinates $(x^n:y^m)$, obtained from (\ref{O}) and the impression $(\tau, \mathbb{C}[x,y])$ of $A$ given in lemma \ref{cyclic impression}.  But these are precisely the coordinates obtained from the Hirzebruch-Jung resolution; see for example \cite[Theorem 3.2]{Reid} and references therein.
\end{proof}

\begin{figure}
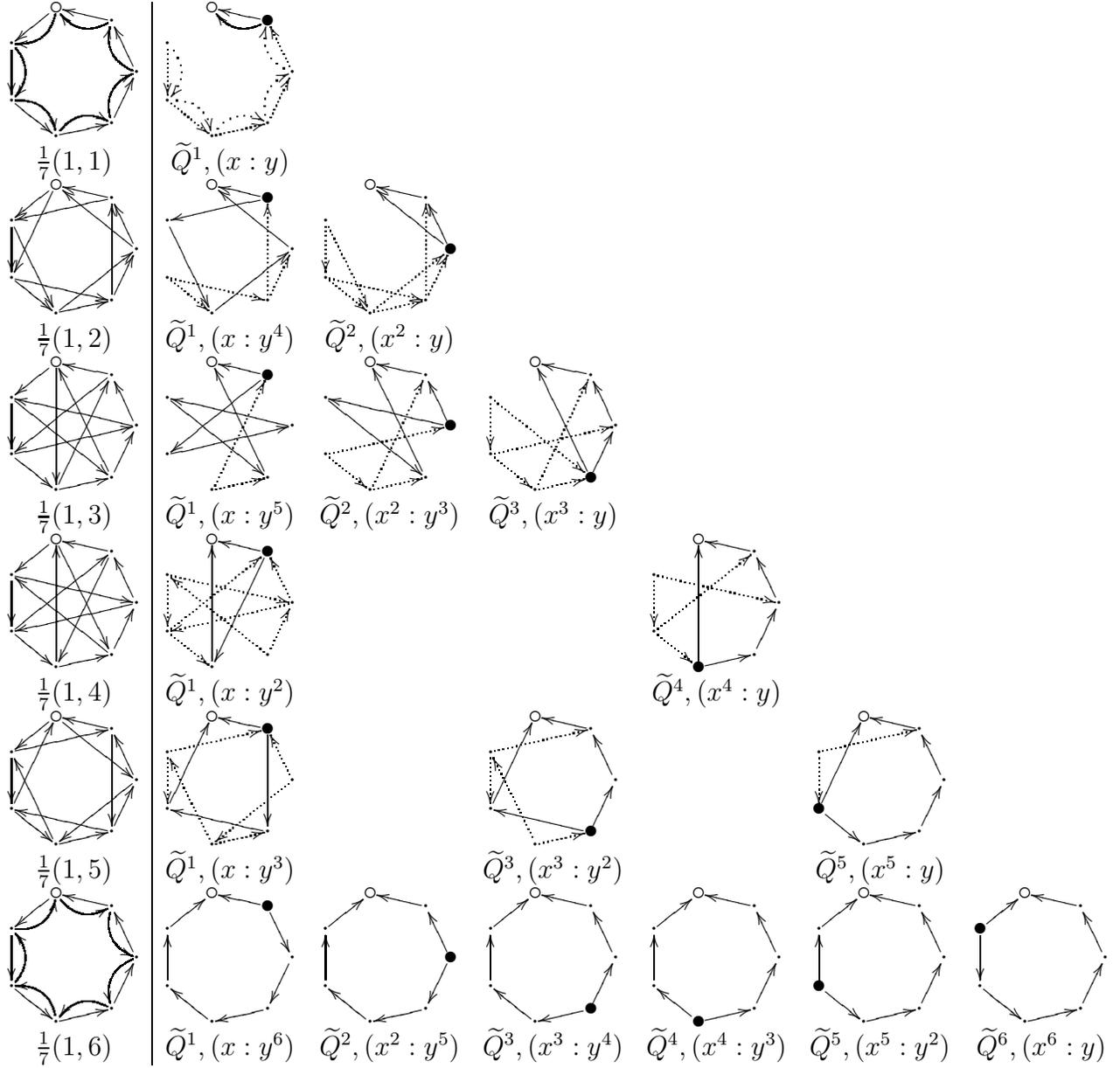

$$\begin{array}{c|cccccc}
\xy
(6.235,7.818)*{\cdot}="2";(-2.225,9.750)*{\circ}="1";(-9.010,4.339)*{\cdot}="7";(-9.010,-4.339)*{\cdot}="6";(-2.225,-9.750)*{\cdot}="5";(6.235,-7.818)*{\cdot}="4";(10,0)*{\cdot}="3";
{\ar@{<-}"1";"2"};{\ar@{<-}"2";"3"};{\ar@{<-}"3";"4"};{\ar@{<-}"4";"5"};{\ar@{<-}"5";"6"};{\ar@{<-}"6";"7"};{\ar@{<-}"7";"1"};
{\ar@{<-}@/_/"1";"2"};{\ar@{<-}@/_/"2";"3"};{\ar@{<-}@/_/"3";"4"};{\ar@{<-}@/_/"4";"5"};{\ar@{<-}@/_/"5";"6"};{\ar@{<-}@/_/"6";"7"};{\ar@{<-}@/_/"7";"1"};
\endxy
&
\xy
(6.235,7.818)*{\bullet}="2";(-2.225,9.750)*{\circ}="1";(-9.010,4.339)*{\cdot}="7";(-9.010,-4.339)*{\cdot}="6";(-2.225,-9.750)*{\cdot}="5";(6.235,-7.818)*{\cdot}="4";(10,0)*{\cdot}="3";
{\ar@{<-}"1";"2"};{\ar@{<..}"2";"3"};{\ar@{<..}"3";"4"};{\ar@{<..}"4";"5"};{\ar@{<..}"5";"6"};{\ar@{<..}"6";"7"};{\ar@{}"7";"1"};
{\ar@{<-}@/_/"1";"2"};{\ar@{<..}@/_/"2";"3"};{\ar@{<..}@/_/"3";"4"};{\ar@{<..}@/_/"4";"5"};{\ar@{<..}@/_/"5";"6"};{\ar@{<..}@/_/"6";"7"};{\ar@{}"7";"1"};\endxy
&&&&&
\\
\frac 17 (1,1) & \widetilde{Q}^{1},(x:y) &&&&&
\\
\xy
(6.235,7.818)*{\cdot}="2";(-2.225,9.750)*{\circ}="1";(-9.010,4.339)*{\cdot}="7";(-9.010,-4.339)*{\cdot}="6";(-2.225,-9.750)*{\cdot}="5";(6.235,-7.818)*{\cdot}="4";(10,0)*{\cdot}="3";
{\ar@{<-}"1";"2"};{\ar@{<-}"2";"3"};{\ar@{<-}"3";"4"};{\ar@{<-}"4";"5"};{\ar@{<-}"5";"6"};{\ar@{<-}"6";"7"};{\ar@{<-}"7";"1"};
{\ar@{<-}"1";"3"};{\ar@{<-}"3";"5"};{\ar@{<-}"5";"7"};{\ar@{<-}"7";"2"};{\ar@{<-}"2";"4"};{\ar@{<-}"4";"6"};{\ar@{<-}"6";"1"};
\endxy
&
\xy
(6.235,7.818)*{\bullet}="2";(-2.225,9.750)*{\circ}="1";(-9.010,4.339)*{\cdot}="7";(-9.010,-4.339)*{\cdot}="6";(-2.225,-9.750)*{\cdot}="5";(6.235,-7.818)*{\cdot}="4";(10,0)*{\cdot}="3";
{\ar@{<-}"1";"2"};{\ar@{}"2";"3"};{\ar@{}"3";"4"};{\ar@{}"4";"5"};{\ar@{}"5";"6"};{\ar@{}"6";"7"};{\ar@{}"7";"1"};
{\ar@{<-}"1";"3"};{\ar@{<-}"3";"5"};{\ar@{<-}"5";"7"};{\ar@{<-}"7";"2"};{\ar@{}"2";"4"};{\ar@{}"4";"6"};{\ar@{}"6";"1"};
{\ar@{..>}"6";"5"};{\ar@{..>}"6";"4"};{\ar@{..>}"4";"2"};{\ar@{..>}"4";"3"};
\endxy
&
\xy
(6.235,7.818)*{\cdot}="2";(-2.225,9.750)*{\circ}="1";(-9.010,4.339)*{\cdot}="7";(-9.010,-4.339)*{\cdot}="6";(-2.225,-9.750)*{\cdot}="5";(6.235,-7.818)*{\cdot}="4";(10,0)*{\bullet}="3";
{\ar@{<-}"1";"2"};{\ar@{<-}"2";"3"};{\ar@{<..}"3";"4"};{\ar@{<..}"4";"5"};{\ar@{<..}"5";"6"};{\ar@{<..}"6";"7"};{\ar@{}"7";"1"};
{\ar@{<-}"1";"3"};{\ar@{<..}"3";"5"};{\ar@{<..}"5";"7"};{\ar@{}"7";"2"};{\ar@{<..}"2";"4"};{\ar@{<..}"4";"6"};{\ar@{}"6";"1"};
\endxy 
&&&& 
\\ 
\frac 17 (1,2) & \widetilde{Q}^{1}, (x:y^4)  & \widetilde{Q}^{2}, (x^2:y) &&&&
\\
\xy
(6.235,7.818)*{\cdot}="2";(-2.225,9.750)*{\circ}="1";(-9.010,4.339)*{\cdot}="7";(-9.010,-4.339)*{\cdot}="6";(-2.225,-9.750)*{\cdot}="5";(6.235,-7.818)*{\cdot}="4";(10,0)*{\cdot}="3";
{\ar@{<-}"1";"2"};{\ar@{<-}"2";"3"};{\ar@{<-}"3";"4"};{\ar@{<-}"4";"5"};{\ar@{<-}"5";"6"};{\ar@{<-}"6";"7"};{\ar@{<-}"7";"1"};
{\ar@{<-}"1";"4"};{\ar@{<-}"4";"7"};{\ar@{<-}"7";"3"};{\ar@{<-}"3";"6"};{\ar@{<-}"6";"2"};{\ar@{<-}"2";"5"};{\ar@{<-}"5";"1"};
\endxy
&
\xy
(6.235,7.818)*{\bullet}="2";(-2.225,9.750)*{\circ}="1";(-9.010,4.339)*{\cdot}="7";(-9.010,-4.339)*{\cdot}="6";(-2.225,-9.750)*{\cdot}="5";(6.235,-7.818)*{\cdot}="4";(10,0)*{\cdot}="3";
{\ar@{<-}"1";"2"};{\ar@{}"2";"3"};{\ar@{}"3";"4"};{\ar@{}"4";"5"};{\ar@{}"5";"6"};{\ar@{}"6";"7"};{\ar@{}"7";"1"};
{\ar@{<-}"1";"4"};{\ar@{<-}"4";"7"};{\ar@{<-}"7";"3"};{\ar@{<-}"3";"6"};{\ar@{<-}"6";"2"};{\ar@{}"2";"5"};{\ar@{}"5";"1"};
{\ar@{..>}"5";"4"};{\ar@{..>}"5";"2"};
\endxy
&
\xy
(6.235,7.818)*{\cdot}="2";(-2.225,9.750)*{\circ}="1";(-9.010,4.339)*{\cdot}="7";(-9.010,-4.339)*{\cdot}="6";(-2.225,-9.750)*{\cdot}="5";(6.235,-7.818)*{\cdot}="4";(10,0)*{\bullet}="3";
{\ar@{<-}"1";"2"};{\ar@{<-}"2";"3"};{\ar@{}"3";"4"};{\ar@{}"4";"5"};{\ar@{}"5";"6"};{\ar@{}"6";"7"};{\ar@{}"7";"1"};
{\ar@{<-}"1";"4"};{\ar@{<-}"4";"7"};{\ar@{<-}"7";"3"};{\ar@{}"3";"6"};{\ar@{}"6";"2"};{\ar@{}"2";"5"};{\ar@{}"5";"1"};
{\ar@{..>}"5";"4"};{\ar@{..>}"5";"2"};{\ar@{..>}"6";"5"};{\ar@{..>}"6";"3"};
\endxy
&
\xy
(6.235,7.818)*{\cdot}="2";(-2.225,9.750)*{\circ}="1";(-9.010,4.339)*{\cdot}="7";(-9.010,-4.339)*{\cdot}="6";(-2.225,-9.750)*{\cdot}="5";(6.235,-7.818)*{\bullet}="4";(10,0)*{\cdot}="3";
{\ar@{<-}"1";"2"};{\ar@{<-}"2";"3"};{\ar@{<-}"3";"4"};{\ar@{<..}"4";"5"};{\ar@{<..}"5";"6"};{\ar@{<..}"6";"7"};{\ar@{}"7";"1"};
{\ar@{<-}"1";"4"};{\ar@{<..}"4";"7"};{\ar@{}"7";"3"};{\ar@{<..}"3";"6"};{\ar@{}"6";"2"};{\ar@{<..}"2";"5"};{\ar@{}"5";"1"};
\endxy &&&
\\
\frac 17 (1,3) & \widetilde{Q}^{1}, (x:y^5) & \widetilde{Q}^{2}, (x^2:y^3) & \widetilde{Q}^{3}, (x^3:y) &&&
\\
\xy
(6.235,7.818)*{\cdot}="2";(-2.225,9.750)*{\circ}="1";(-9.010,4.339)*{\cdot}="7";(-9.010,-4.339)*{\cdot}="6";(-2.225,-9.750)*{\cdot}="5";(6.235,-7.818)*{\cdot}="4";(10,0)*{\cdot}="3";
{\ar@{<-}"1";"2"};{\ar@{<-}"2";"3"};{\ar@{<-}"3";"4"};{\ar@{<-}"4";"5"};{\ar@{<-}"5";"6"};{\ar@{<-}"6";"7"};{\ar@{<-}"7";"1"};
{\ar@{<-}"1";"5"};{\ar@{<-}"5";"2"};{\ar@{<-}"2";"6"};{\ar@{<-}"6";"3"};{\ar@{<-}"3";"7"};{\ar@{<-}"7";"4"};{\ar@{<-}"4";"1"};
\endxy
&
\xy
(6.235,7.818)*{\bullet}="2";(-2.225,9.750)*{\circ}="1";(-9.010,4.339)*{\cdot}="7";(-9.010,-4.339)*{\cdot}="6";(-2.225,-9.750)*{\cdot}="5";(6.235,-7.818)*{\cdot}="4";(10,0)*{\cdot}="3";
{\ar@{<-}"1";"2"};{\ar@{<..}"2";"3"};{\ar@{<..}"3";"4"};{\ar@{}"4";"5"};{\ar@{<..}"5";"6"};{\ar@{<..}"6";"7"};{\ar@{}"7";"1"};
{\ar@{<-}"1";"5"};{\ar@{<-}"5";"2"};{\ar@{<..}"2";"6"};{\ar@{<..}"6";"3"};{\ar@{<..}"3";"7"};{\ar@{<..}"7";"4"};{\ar@{}"4";"1"};
\endxy
& &&
\xy
(6.235,7.818)*{\cdot}="2";(-2.225,9.750)*{\circ}="1";(-9.010,4.339)*{\cdot}="7";(-9.010,-4.339)*{\cdot}="6";(-2.225,-9.750)*{\bullet}="5";(6.235,-7.818)*{\cdot}="4";(10,0)*{\cdot}="3";
{\ar@{<-}"1";"2"};{\ar@{<-}"2";"3"};{\ar@{<-}"3";"4"};{\ar@{<-}"4";"5"};{\ar@{}"5";"6"};{\ar@{}"6";"7"};{\ar@{}"7";"1"};
{\ar@{<-}"1";"5"};{\ar@{}"5";"2"};{\ar@{}"2";"6"};{\ar@{}"6";"3"};{\ar@{}"3";"7"};{\ar@{}"7";"4"};{\ar@{}"4";"1"};
{\ar@{..>}"7";"6"};{\ar@{..>}"7";"3"};{\ar@{..>}"6";"5"};{\ar@{..>}"6";"2"};
\endxy
&&
\\
\frac 17 (1,4) & \widetilde{Q}^{1}, (x:y^2) & &&\widetilde{Q}^{4}, (x^4:y) &&
\\
\xy
(6.235,7.818)*{\cdot}="2";(-2.225,9.750)*{\circ}="1";(-9.010,4.339)*{\cdot}="7";(-9.010,-4.339)*{\cdot}="6";(-2.225,-9.750)*{\cdot}="5";(6.235,-7.818)*{\cdot}="4";(10,0)*{\cdot}="3";
{\ar@{<-}"1";"2"};{\ar@{<-}"2";"3"};{\ar@{<-}"3";"4"};{\ar@{<-}"4";"5"};{\ar@{<-}"5";"6"};{\ar@{<-}"6";"7"};{\ar@{<-}"7";"1"};
{\ar@{<-}"1";"6"};{\ar@{<-}"6";"4"};{\ar@{<-}"4";"2"};{\ar@{<-}"2";"7"};{\ar@{<-}"7";"5"};{\ar@{<-}"5";"3"};{\ar@{<-}"3";"1"};
\endxy
&
\xy
(6.235,7.818)*{\bullet}="2";(-2.225,9.750)*{\circ}="1";(-9.010,4.339)*{\cdot}="7";(-9.010,-4.339)*{\cdot}="6";(-2.225,-9.750)*{\cdot}="5";(6.235,-7.818)*{\cdot}="4";(10,0)*{\cdot}="3";
{\ar@{<-}"1";"2"};{\ar@{<..}"2";"3"};{\ar@{}"3";"4"};{\ar@{<..}"4";"5"};{\ar@{}"5";"6"};{\ar@{<..}"6";"7"};{\ar@{}"7";"1"};
{\ar@{<-}"1";"6"};{\ar@{<-}"6";"4"};{\ar@{<-}"4";"2"};{\ar@{<..}"2";"7"};{\ar@{<..}"7";"5"};{\ar@{<..}"5";"3"};{\ar@{}"3";"1"};
\endxy
& &
\xy
(6.235,7.818)*{\cdot}="2";(-2.225,9.750)*{\circ}="1";(-9.010,4.339)*{\cdot}="7";(-9.010,-4.339)*{\cdot}="6";(-2.225,-9.750)*{\cdot}="5";(6.235,-7.818)*{\bullet}="4";(10,0)*{\cdot}="3";
{\ar@{<-}"1";"2"};{\ar@{<-}"2";"3"};{\ar@{<-}"3";"4"};{\ar@{}"4";"5"};{\ar@{}"5";"6"};{\ar@{}"6";"7"};{\ar@{}"7";"1"};
{\ar@{<-}"1";"6"};{\ar@{<-}"6";"4"};{\ar@{}"4";"2"};{\ar@{}"2";"7"};{\ar@{}"7";"5"};{\ar@{}"5";"3"};{\ar@{}"3";"1"};
{\ar@{..>}"7";"6"};{\ar@{..>}"7";"2"};{\ar@{..>}"5";"7"};{\ar@{..>}"5";"4"};
\endxy
& &
\xy
(6.235,7.818)*{\cdot}="2";(-2.225,9.750)*{\circ}="1";(-9.010,4.339)*{\cdot}="7";(-9.010,-4.339)*{\bullet}="6";(-2.225,-9.750)*{\cdot}="5";(6.235,-7.818)*{\cdot}="4";(10,0)*{\cdot}="3";
{\ar@{<-}"1";"2"};{\ar@{<-}"2";"3"};{\ar@{<-}"3";"4"};{\ar@{<-}"4";"5"};{\ar@{<-}"5";"6"};{\ar@{}"6";"7"};{\ar@{}"7";"1"};
{\ar@{<-}"1";"6"};{\ar@{}"6";"4"};{\ar@{}"4";"2"};{\ar@{}"2";"7"};{\ar@{}"7";"5"};{\ar@{}"5";"3"};{\ar@{}"3";"1"};
{\ar@{..>}"7";"6"};{\ar@{..>}"7";"2"};
\endxy &
\\
\frac 17 (1,5) & \widetilde{Q}^{1}, (x:y^3) & & \widetilde{Q}^{3}, (x^3:y^2) & & \widetilde{Q}^{5}, (x^5:y) &
\\
\xy
(6.235,7.818)*{\cdot}="2";(-2.225,9.750)*{\circ}="1";(-9.010,4.339)*{\cdot}="7";(-9.010,-4.339)*{\cdot}="6";(-2.225,-9.750)*{\cdot}="5";(6.235,-7.818)*{\cdot}="4";(10,0)*{\cdot}="3";
{\ar@{<-}"1";"2"};{\ar@{<-}"2";"3"};{\ar@{<-}"3";"4"};{\ar@{<-}"4";"5"};{\ar@{<-}"5";"6"};{\ar@{<-}"6";"7"};{\ar@{<-}"7";"1"};
{\ar@{<-}@/^/"1";"7"};{\ar@{<-}@/^/"7";"6"};{\ar@{<-}@/^/"6";"5"};{\ar@{<-}@/^/"5";"4"};{\ar@{<-}@/^/"4";"3"};{\ar@{<-}@/^/"3";"2"};{\ar@{<-}@/^/"2";"1"};
\endxy
&
\xy
(6.235,7.818)*{\bullet}="2";(-2.225,9.750)*{\circ}="1";(-9.010,4.339)*{\cdot}="7";(-9.010,-4.339)*{\cdot}="6";(-2.225,-9.750)*{\cdot}="5";(6.235,-7.818)*{\cdot}="4";(10,0)*{\cdot}="3";
{\ar@{<-}"1";"2"};{\ar@{}"2";"3"};{\ar@{}"3";"4"};{\ar@{}"4";"5"};{\ar@{}"5";"6"};{\ar@{}"6";"7"};{\ar@{}"7";"1"};
{\ar@{<-}"1";"7"};{\ar@{<-}"7";"6"};{\ar@{<-}"6";"5"};{\ar@{<-}"5";"4"};{\ar@{<-}"4";"3"};{\ar@{<-}"3";"2"};{\ar@{}"2";"1"};
\endxy
&
\xy
(6.235,7.818)*{\cdot}="2";(-2.225,9.750)*{\circ}="1";(-9.010,4.339)*{\cdot}="7";(-9.010,-4.339)*{\cdot}="6";(-2.225,-9.750)*{\cdot}="5";(6.235,-7.818)*{\cdot}="4";(10,0)*{\bullet}="3";
{\ar@{<-}"1";"2"};{\ar@{<-}"2";"3"};{\ar@{}"3";"4"};{\ar@{}"4";"5"};{\ar@{}"5";"6"};{\ar@{}"6";"7"};{\ar@{}"7";"1"};
{\ar@{<-}"1";"7"};{\ar@{<-}"7";"6"};{\ar@{<-}"6";"5"};{\ar@{<-}"5";"4"};{\ar@{<-}"4";"3"};{\ar@{}"3";"2"};{\ar@{}"2";"1"};
\endxy
&
\xy
(6.235,7.818)*{\cdot}="2";(-2.225,9.750)*{\circ}="1";(-9.010,4.339)*{\cdot}="7";(-9.010,-4.339)*{\cdot}="6";(-2.225,-9.750)*{\cdot}="5";(6.235,-7.818)*{\bullet}="4";(10,0)*{\cdot}="3";
{\ar@{<-}"1";"2"};{\ar@{<-}"2";"3"};{\ar@{<-}"3";"4"};{\ar@{}"4";"5"};{\ar@{}"5";"6"};{\ar@{}"6";"7"};{\ar@{}"7";"1"};
{\ar@{<-}"1";"7"};{\ar@{<-}"7";"6"};{\ar@{<-}"6";"5"};{\ar@{<-}"5";"4"};{\ar@{}"4";"3"};{\ar@{}"3";"2"};{\ar@{}"2";"1"};
\endxy
&
\xy
(6.235,7.818)*{\cdot}="2";(-2.225,9.750)*{\circ}="1";(-9.010,4.339)*{\cdot}="7";(-9.010,-4.339)*{\cdot}="6";(-2.225,-9.750)*{\bullet}="5";(6.235,-7.818)*{\cdot}="4";(10,0)*{\cdot}="3";
{\ar@{<-}"1";"2"};{\ar@{<-}"2";"3"};{\ar@{<-}"3";"4"};{\ar@{<-}"4";"5"};{\ar@{}"5";"6"};{\ar@{}"6";"7"};{\ar@{}"7";"1"};
{\ar@{<-}"1";"7"};{\ar@{<-}"7";"6"};{\ar@{<-}"6";"5"};{\ar@{}"5";"4"};{\ar@{}"4";"3"};{\ar@{}"3";"2"};{\ar@{}"2";"1"};
\endxy
&
\xy
(6.235,7.818)*{\cdot}="2";(-2.225,9.750)*{\circ}="1";(-9.010,4.339)*{\cdot}="7";(-9.010,-4.339)*{\bullet}="6";(-2.225,-9.750)*{\cdot}="5";(6.235,-7.818)*{\cdot}="4";(10,0)*{\cdot}="3";
{\ar@{<-}"1";"2"};{\ar@{<-}"2";"3"};{\ar@{<-}"3";"4"};{\ar@{<-}"4";"5"};{\ar@{<-}"5";"6"};{\ar@{}"6";"7"};{\ar@{}"7";"1"};
{\ar@{<-}"1";"7"};{\ar@{<-}"7";"6"};{\ar@{}"6";"5"};{\ar@{}"5";"4"};{\ar@{}"4";"3"};{\ar@{}"3";"2"};{\ar@{}"2";"1"};
\endxy
&
\xy
(6.235,7.818)*{\cdot}="2";(-2.225,9.750)*{\circ}="1";(-9.010,4.339)*{\bullet}="7";(-9.010,-4.339)*{\cdot}="6";(-2.225,-9.750)*{\cdot}="5";(6.235,-7.818)*{\cdot}="4";(10,0)*{\cdot}="3";
{\ar@{<-}"1";"2"};{\ar@{<-}"2";"3"};{\ar@{<-}"3";"4"};{\ar@{<-}"4";"5"};{\ar@{<-}"5";"6"};{\ar@{<-}"6";"7"};{\ar@{}"7";"1"};
{\ar@{<-}"1";"7"};{\ar@{}"7";"6"};{\ar@{}"6";"5"};{\ar@{}"5";"4"};{\ar@{}"4";"3"};{\ar@{}"3";"2"};{\ar@{}"2";"1"};
\endxy
\\
\frac 17 (1,6) & \widetilde{Q}^{1}, (x:y^6)  & \widetilde{Q}^{2}, (x^2:y^5) & \widetilde{Q}^{3}, (x^3:y^4) & \widetilde{Q}^{4}, (x^4:y^3) & \widetilde{Q}^{5}, (x^5:y^2) & \widetilde{Q}^{6}, (x^6:y)
\end{array}$$
\caption{The supporting subquivers of the $\ell_{\mathcal{P}} = 1$ almost large modules over the McKay quiver algebras of type $\frac 17 (1,b)$.}
\label{cyclic example}
\end{figure}
%
%
\begin{Example} \label{cyclic} \rm{The supporting subquivers $\widetilde{Q}^m$ of the $\ell_{\mathcal{P}} = 1$ almost large modules over the $\frac 17 (1,b)$ McKay quiver algebra $A$ with $1 \leq b \leq 6$ are shown in figure \ref{cyclic example}.
} \end{Example}  

\begin{Remark} \rm{ 
Ishii showed that for small finite subgroups $G \subset \operatorname{GL}_2(\mathbb{C})$, the $G$-Hilbert scheme $\operatorname{Hilb}^G(\mathbb{C}^2)$ coincides with the minimal resolution of $\mathbb{C}^2/G$ using Wunram's special representations of $G$ \cite[Theorem 3.1]{I}.  It would therefore be interesting to understand how special representations are related to almost large modules.
}\end{Remark}

A cyclic quotient surface singularity is Gorenstein if and only if it is of type $\frac 1n (1,-1)$, in which case the McKay quiver algebra coincides with the $A_n$ preprojective algebra.

\begin{Corollary} Let $A$ be the $A_n$ preprojective algebra, and let $\pi: Y \rightarrow \mathbb{C}^2/\mu_n$ be the minimal resolution of the $A_n$ surface singularity.  The irreducible component $E_i$ of $\pi^{-1}(0)$, associated to the vertex $i \in Q_0$ by the McKay correspondence, shrinks to the vertex simple $A$-module $S_i$.
\end{Corollary}

\subsection{$D_n$ and $E_6$ surface singularities} \label{D_n example}

Consider the linear action of the binary dihedral group of order $4n$, 
$$\operatorname{BD}_{4n}:= \left\langle g,j \ | \ g^{2n} = e, g^n = j^2, gjg = j \right\rangle,$$
on $\mathbb{C}[x,y]$ by the representation 
$$\rho(g) = \left[ \begin{array}{cc} e^{\pi i/n} & 0 \\ 0 & e^{-\pi i/n} \end{array} \right], \ \ \ \ \ \rho(j) = \left[ \begin{array}{cc} 0 & 1 \\ -1 & 0 \end{array} \right],$$
that is, $g \cdot (x,y) = \left( e^{\pi i/n}x, e^{-\pi i/n}y \right)$ and $j \cdot (x,y) = (y,-x)$.  Similarly, consider the linear action on $\mathbb{C}[x,y]$ of the binary tetrahedral group 
$$\operatorname{BT} := \left\{ \pm 1, \pm i, \pm j, \pm k, \frac 12 \left( 1\pm i \pm j \pm k \right) \right\} \subset \mathbb{H},$$
where all possible sign combinations occur, by the representation
$$\rho(i) =  \left[ \begin{array}{cc} i & 0 \\ 0 &-i \end{array} \right], \ \rho(j) = \left[ \begin{array}{cc} 0 & 1 \\ -1 & 0 \end{array} \right], \ \rho(k) = \left[ \begin{array}{cc} 0 & i \\ i & 0 \end{array} \right].$$
The ring of invariants $R = \mathbb{C}[x,y]^{\rho(\operatorname{BD}_{4n})}$ and $R = \mathbb{C}[x,y]^{\rho(\operatorname{BT})}$ are the respective coordinate rings for the $D_{n+2}$ and $E_6$ Kleinian singularities $\operatorname{Max}R := \mathbb{C}^2/\rho(\operatorname{BD}_{4n})$ and $\mathbb{C}^2/\rho(\operatorname{BT})$ .

Denote by $Q$ the McKay quiver of $(\operatorname{BD_{4n}}, \rho)$ (resp.\ $(\operatorname{E_{6}}, \rho)$), shown in figure \ref{D_n singularity} (resp.\ figure \ref{E_6 singularity}), and let $A$ be the preprojective algebra $A = \mathbb{C}Q/\left\langle \sum_i [a_i, \bar{a}_i ] \right\rangle$.  $A$ is module-finite over its center $Z$, and $Z \cong R$ (this follows since $A$ is Morita equivalent to the corresponding skew group ring $\mathbb{C}[x,y] * \operatorname{BD}_{4n}$ or $\mathbb{C}[x,y] * \operatorname{BT}$ \cite[proof of Proposition 2.13]{RV} which has center $R$, and Morita equivalent rings have isomorphic centers).  
Moreover, the smooth locus of $\operatorname{Max}R$ parameterizes the large $A$-modules,\footnote{This follows, for example, since the moduli space of $\theta$-stable modules with $\theta =0$ and dimension vector $d$ coincides with the smooth locus of $\operatorname{Max}R$, and the only nonzero stable modules with $\theta =0$ are simple.} and this fact is extended in Theorem \ref{D_n theorem}, where we give strong evidence that Conjecture \ref{conjecture} holds for the $D_{n+2}$ and $E_6$ surface singularities and their respective noncommutative coordinate rings $A$.

The following lemma is known, but we give a proof for completeness.

\begin{Lemma} \label{large lemma}
Let $A = \mathbb{C}Q/\left\langle \sum_i [a_i, \bar{a}_i ] \right\rangle$ be the preprojective algebra of a quiver $Q'$ whose underlying graph is extended Dynkin.  Let $d_i$ be the dimension of the irreducible representation of $G$ corresponding to vertex $i$.  Then the dimension vector of any large $A$-module is $d = ( d_i )_{i \in Q_0}$.
\end{Lemma}

\begin{proof}
By \cite[p.\ 18, (1) and (3)]{CB2} when $Q'$ is extended Dynkin the real roots of the corresponding (positive semi-definite) Tits form $q: \mathbb{Z}^{|Q_0|} \rightarrow \mathbb{Z}$ are the coordinate vectors $\epsilon_i$, while the imaginary roots are the nonzero integer multiples of $d$.  Thus for $m \in \mathbb{Z}$, $p(m d) := 1+ q(m d) = 1$, and so for $m \geq 2$ we have $p(m d) = 1 < m = m p(d)$.  Apply \cite[Theorem 1.2]{CB1}.
\end{proof}

\begin{Lemma} \label{socle S_k}
Let $A = \mathbb{C}Q/I$ be a quiver algebra and let $V$ be an $A$-module with pulled-apart supporting subquiver $\widetilde{Q}$ (with respect to some basis).  Suppose that (i) $k \in \widetilde{Q}_0$ is a sink in $\widetilde{Q}$, (ii) there is a path in $\widetilde{Q}$ from each vertex $i \in \widetilde{Q}_0$ to $k$, and (iii) if $\operatorname{h}(a) = \operatorname{h}(b)$ for distinct $a, b \in \widetilde{Q}_1$ then $a$ and $b$ correspond to distinct arrows in $Q_1$.  Then the socle of $V$ is isomorphic to the vertex simple $S_k$. 
\end{Lemma}

\begin{proof}
Let $j \in Q_0$ and consider a nonzero vector $v \in e_jV$ that is not in the 1-dimensional vector space at $k$.  If $a \in Q_1e_j$ satisfies $0 \not = av =: w$ then consider $w$ in place of $v$; otherwise if $av = 0$ then by (ii) and (iii) there exists a $b \in Q_1e_j$ such that $0 \not = bv =: w$.  Since $Q$ is finite, by (ii) we may iterate this process a finite number of times until $w \not = 0$ is in the 1-dimensional vector space at $k$.  The isomorphism $\operatorname{Soc}V \cong S_k$ then follows by (i).
(Note that if the dimension vector of $V$ is not $(1,\ldots, 1)$ then (i) and (ii) alone are not sufficient to imply $\operatorname{Soc}V \cong S_k$.)
\end{proof}

\begin{Lemma} \label{cycle lemma}
Let $A = \mathbb{C}Q/I$ be the preprojective algebra of an extended Dynkin quiver, and let $d = (d_i)_{i \in Q_0}$ be the dimension vector of a large $A$-module.  Suppose $V \in \operatorname{Rep}_dA$ and $d_k = 1$.  If there is a cycle $c \in e_{\operatorname{t}(c)}A$ such that $c^n \not \in \operatorname{ann}_AV$ for all $n \geq 1$ then $\operatorname{Soc}V$ cannot be isomorphic to the vertex simple $S_k$.
\end{Lemma}

\begin{proof}
If $c^nV \not = 0$ for all $n \geq 1$ then there is an eigenvector $v \in e_{\operatorname{t}(c)}V \subset V$ such that $c^mv = \gamma v$ for some $m \geq 1$ and $\gamma \in \mathbb{C}^*$.  But then for sufficiently large $r$, $e_k$ is a subpath of a term of $(c^m)^r$ modulo $I$ (using the preprojective relations $I$), and the lemma follows since $c^{mr}v = \gamma^r v \not = 0$.
\end{proof}

\begin{figure}
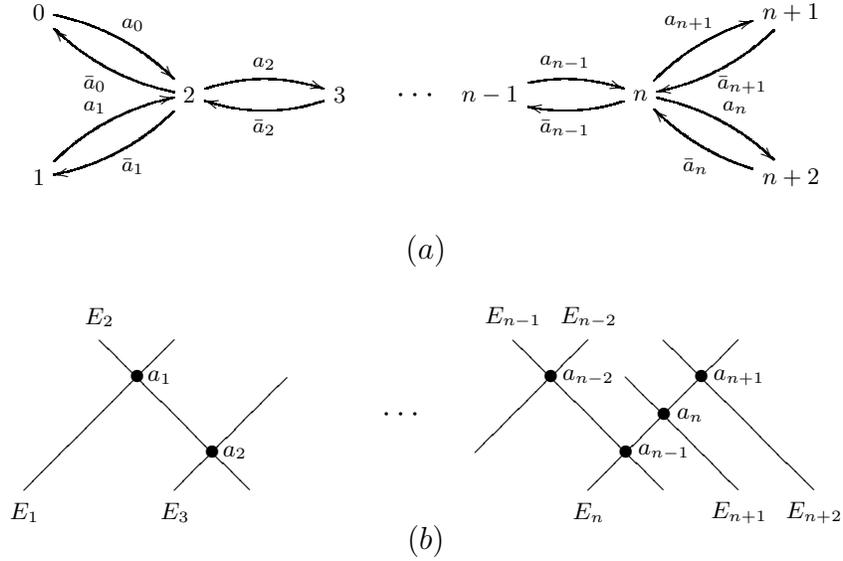

$$\begin{array}{c}
\xy
(-50,11)*+{\text{\scriptsize{$0$}}}="0";
(-50,-11)*+{\text{\scriptsize{$1$}}}="1";
(-30,0)*+{\text{\scriptsize{$2$}}}="2";
(-10,0)*+{\text{\scriptsize{$3$}}}="3";
(0,0)*{\cdots}="";
(10,0)*+{\text{\scriptsize{$n-1$}}}="n-1";
(30,0)*+{\text{\scriptsize{$n$}}}="n";
(50,-11)*+{\text{\scriptsize{$n+2$}}}="n+1";
(50,11)*+{\text{\scriptsize{$n+1$}}}="n+2";
{\ar@/^/^{a_0}"0";"2"};
{\ar@/^/^{a_1}"1";"2"};
{\ar@/^/^{a_2}"2";"3"};
{\ar@/^/^{a_{n-1}}"n-1";"n"};
{\ar@/^/^{a_n}"n";"n+1"};
{\ar@/^/^{a_{n+1}}"n";"n+2"};
{\ar@/^/^{\bar{a}_0}"2";"0"};
{\ar@/^/^{\bar{a}_1}"2";"1"};
{\ar@/^/^{\bar{a}_2}"3";"2"};
{\ar@/^/^{\bar{a}_{n-1}}"n";"n-1"};
{\ar@/^/^{\bar{a}_n}"n+1";"n"};
{\ar@/^/^{\bar{a}_{n+1}}"n+2";"n"};
\endxy\\
\\
(a)\\
\\
\xy 
(-50,-10)*{}="1";
(-40,10)*{}="2";
(-35,5)*{\bullet}="3";
(-32,5)*{\text{\scriptsize{$a_1$}}}="";
(-30,10)*{}="4";
(-30,-10)*{}="5";
(-25,-5)*{\bullet}="6";
(-22,-5)*{\text{\scriptsize{$a_2$}}}="";
(-15,5)*{}="7";
(10,-5)*{}="8";
(15,10)*{}="9";
(20,5)*{\bullet}="10";
(25,5)*{\text{\scriptsize{$a_{n-2}$}}}="";
(25,10)*{}="11";
(25,-10)*{}="12";
(30,-5)*{\bullet}="13";
(35,-5)*{\text{\scriptsize{$a_{n-1}$}}}="";
(30,5)*{}="14";
(35,0)*{\bullet}="15";
(38.5,0)*{\text{\scriptsize{$a_n$}}}="";
(35,10)*{}="16";
(40,5)*{\bullet}="17";
(45,5)*{\text{\scriptsize{$a_{n+1}$}}}="";
(45,10)*{}="18";
(45,-10)*{}="19";
(55,-10)*{}="20";
(-20,-10)*{}="21";
(35,-10)*{}="22";
(-50,-13)*{\text{\scriptsize{$E_1$}}}="";
(-40,13)*{\text{\scriptsize{$E_2$}}}="";
(-30,-13)*{\text{\scriptsize{$E_3$}}}="";
(15,13)*{\text{\scriptsize{$E_{n-1}$}}}="";
(25,13)*{\text{\scriptsize{$E_{n-2}$}}}="";
(25,-13)*{\text{\scriptsize{$E_n$}}}="";
(45,-13)*{\text{\scriptsize{$E_{n+1}$}}}="";
(55,-13)*{\text{\scriptsize{$E_{n+2}$}}}="";
{\ar@{-}"1";"4"};
{\ar@{-}"2";"21"};
{\ar@{-}"5";"7"};
{\ar@{-}"8";"11"};
{\ar@{-}"9";"22"};
{\ar@{-}"12";"18"};
{\ar@{-}"14";"19"};
{\ar@{-}"16";"20"};
(0,0)*{\cdots}="";
\endxy
\\
(b)
\end{array}$$
\caption{(a) The $D_{n+2}$ McKay quiver $Q$.  (b) The exceptional locus of the minimal resolution of the $D_{n+2}$ singularity (each edge is a $\mathbb{P}^1$).  By the McKay correspondence, there is an arrow $i \rightarrow j$ in $Q$ 
iff the intersection $E_i \cap E_j$ is nonempty.}
\label{D_n singularity}
\end{figure}

\begin{figure}
$$\begin{array}{c}
\xy
(0,-21.5)*+{0}="0";
(0,-4)*+{1}="4";
(0,13.5)*+{2}="1";
(-21,13.5)*+{3}="2";(21,13.5)*+{4}="3";
(-42,13.5)*+{5}="2'";(42,13.5)*+{6}="3'";
{\ar@/^/^{a_0}"0";"4"};
{\ar@/^/^{\bar{a}_0}"4";"0"};
{\ar@/^/^{a_1}"4";"1"};
{\ar@/^/^{\bar{a}_1}"1";"4"};
{\ar@/^/^{a_{5}}"2'";"2"};
{\ar@/^/^{\bar{a}_{5}}"2";"2'"};
{\ar@/^/^{a_3}"2";"1"};
{\ar@/^/^{\bar{a}_{3}}"1";"2"};
{\ar@/_/_{a_4}"3";"1"};
{\ar@/_/_{\bar{a}_4}"1";"3"};
{\ar@/_/_{a_{6}}"3'";"3"};
{\ar@/_/_{\bar{a}_{6}}"3";"3'"};
\endxy\\
\\
(a)\\
\\
\xy
(0,20)*{}="1";
(0,5)*{\bullet}="2";
(0,-8)*{}="3";
(-20,5)*{\bullet}="5";
(20,5)*{\bullet}="16";
(-30,-10)*{\bullet}="8";
(30,-10)*{\bullet}="13";
(-15,-25)*{}="10";
(15,-25)*{}="11";
(-29,5)*{}="6";
(29,5)*{}="15";
(-34,-16)*{}="9";
(34,-16)*{}="12";
(-16,11)*{}="4";
(16,11)*{}="17";
(-35,-5)*{}="7";
(35,-5)*{}="14";
{\ar@{-}"1";"3"};
{\ar@{-}"6";"15"};
{\ar@{-}"4";"9"};
{\ar@{-}"17";"12"};
{\ar@{-}"7";"10"};
{\ar@{-}"14";"11"};
(4,20)*{E_1}="";
(32,5)*{E_2}="";
(-13,13)*{E_3}="";
(13,13)*{E_4}="";
(-12,-28)*{E_{5}}="";
(12,-28)*{E_{6}}="";
(2.5,3)*{a_1}="";
(18.5,3)*{a_4}="";
(-18.5,3)*{a_3}="";
(-26.5,-10)*{a_{5}}="";
(26.5,-10)*{a_{6}}="";
\endxy
\\
(b)
\end{array}$$
\caption{(a) The $E_6$ McKay quiver $Q$.  (b) The exceptional locus of the minimal resolution of the $E_6$ singularity (each edge is a $\mathbb{P}^1$).} \label{E_6 singularity}
\end{figure}

Denote by $\mathcal{P}$ the path-like set $Q_{\geq 0} \cup \{0\}$.  In the following theorem, let $P_1$ denote the $\mathcal{P}$-annihilator of an $A$-module with pulled-apart supporting subquiver given in figure \ref{D_n pulled-apart sub} for the $D_{n+2}$ case and in figure \ref{E_6 pulled-apart sub} for the $E_6$ case.  We will assume that the chain
\begin{equation} \label{D_n chain}
0 \subsetneq P_1
\end{equation}
is maximal in the sense of Definition \ref{almost large}, which is expected by Lemma \ref{cycle lemma}.

\begin{Theorem} \label{D_n theorem} 
Let $A=\mathbb{C}Q/I$ be the $D_{n+2}$ (resp.\ $E_6$) preprojective algebra, let 
$$\pi: Y \rightarrow \mathbb{C}^2/\rho(\operatorname{BD}_{4n}) \ \ \ \ \ \left(\text{resp.\ \ } \pi: Y \rightarrow \mathbb{C}^2/\rho(\operatorname{BT})\right)$$ 
be the minimal resolution of the Gorenstein $D_{n+2}$ (resp.\ $E_6$) surface singularity, and fix a vertex $k \in \left\{ 0,1,n+1,n+2 \right\}$ (resp.\ $k \in \left\{ 0,5,6 \right\}$).  If the chain (\ref{D_n chain}) is maximal (which is expected), then the exceptional locus $\pi^{-1}(0)$ parameterizes the almost large $A$-modules with socles isomorphic to the vertex simple $S_k$.  Furthermore, the irreducible component $E_i$ of $\pi^{-1}(0)$, associated to the vertex $i \in Q_0$ by the McKay correspondence, shrinks to the vertex simple $A$-module $S_i$.
\end{Theorem}

\begin{proof} 
Denote by $\widetilde{Y}$ the space that parameterizes the isoclasses of almost large modules whose socles are isomorphic to $S_k$.\\
\\
\textit{Claim I:  $Y \subseteq \widetilde{Y}$.}\\
(i) \textit{$\mathbb{P}^1$-families.}  Each $\widetilde{Q}^i$ in figure \ref{D_n pulled-apart sub} (resp.\ figure \ref{E_6 pulled-apart sub}) is the support of a $\mathbb{P}^1$-family, minus the two points $(1:0)$ and $(0:1)$:\ apply the method ``Trivialize $J_0$'' in section \ref{Determining P^n families} to determine the monomorphism
$$\sigma^i: A \rightarrow \operatorname{Mat}_{2n}(\mathbb{C}[s_i,t_i]),$$
$$\left( \text{resp.\ \ \ } \sigma^i: A \rightarrow \operatorname{Mat}_{12}(\mathbb{C}[s_i,t_i]) \ \right)$$   
which is given by the labeling of $\widetilde{Q}^i$ in figure \ref{D_n pulled-apart sub} (resp.\ figure \ref{E_6 pulled-apart sub}).  Here the unlabeled arrows are represented by $\pm 1$, the sign being chosen so that the preprojective relations hold.  By lemma \ref{rho cong}, given any representation $\rho$ supported on $\widetilde{Q}^i$ there is some $z \in \mathbb{C}^2$ such that $\rho$ is isomorphic to $\epsilon_{z} \cdot \sigma$.  In the $D_{n+2}$ case: it is straightforward to check that the parameters $(s,t)$ in the example given in figures \ref{s,t}.i and \ref{s,t}.iii coincide schematically with the respective parameters $(s_i,t_i)$ for $2 \leq i \leq n$ and $i = 1,n+1,n+2$.  In the $E_6$ case: one may check that the parameters $(s,t)$ in the example given in figure \ref{s,t}.iii coincides schematically with the parameters $(s_i,t_i)$ for $i = 1, \ldots, 6$; specifically, the two dimensional vector space in figure \ref{s,t}.iii sits inside the vector space at vertex $2,3,2,3,4 \in Q_0$ respectively. $\Box$\\
\\
(ii) \textit{$\ell_{\mathcal{P}} = 1$ almost large modules.}  By lemma \ref{large lemma}, the almost large modules have dimension vector $d = (1,1,2, \ldots, 2, 1,1)$ (resp.\ $(1,2,3,2,2,1,1)$).  By lemma \ref{socle S_k}, any module supported on a pulled-apart subquiver given in figure \ref{D_n pulled-apart sub} (resp.\ figure \ref{E_6 pulled-apart sub}) has socle $S_k$.  Here we assume the chain (\ref{D_n chain}) is maximal by Lemma \ref{cycle lemma}.  $\Box$\\
\\
(iii) \textit{$\ell_{\mathcal{P}} = 2$ almost large modules.} Each intersection point $E_i \cap E_j$ in the minimal resolution corresponds to a (unique) almost large module isoclass $V$ that belongs to two $\mathbb{P}^1$-families.  Although these two families have different pulled-apart supporting subquivers, namely $\widetilde{Q}^i$ and $\widetilde{Q}^j$, $V$ is parameterized by the vanishing of a coordinate in each $\mathbb{P}^1$-family, and so the support of $V$ is properly contained in both $\widetilde{Q}^i$ and $\widetilde{Q}^j$, as shown below.\\
In the $D_{n+2}$ case:
$$\begin{array}{rll}
a_1 = E_1 \cap E_2: & \widetilde{Q}_1^1 \setminus \{ \stackrel{t_1}{\longrightarrow} \} \ = \ \widetilde{Q}_1^2 \setminus \{ \stackrel{s_2}{\longleftarrow}, \stackrel{s_2}{\longrightarrow} \}, & t_1=s_2=0\\
& \vdots & \vdots \\
a_{j} = E_j \cap E_{j+1}: & \widetilde{Q}_1^j \setminus \{ \stackrel{t_j}{\longrightarrow} \} \ = \ \widetilde{Q}_1^{j+1} \setminus \{ \stackrel{s_{j+1}}{\longleftarrow}, \stackrel{s_{j+1}}{\longrightarrow} \}, & t_j=s_{j+1}=0\\
& \vdots & \vdots \\
a_{n-1} = E_{n-1} \cap E_n: & \widetilde{Q}_1^{n-1} \setminus \{ \stackrel{t_{n-1}}{\longrightarrow} \} \ = \ \widetilde{Q}_1^{n} \setminus \{ \stackrel{s_{n}}{\longleftarrow} \}, & t_{n-1}=s_{n}=0 \\
a_{n} = E_n \cap E_{n+1}: & \widetilde{Q}_1^{n} \setminus \{ \stackrel{t_{n}}{\longrightarrow} \} \ = \ \widetilde{Q}_1^{n+1} \setminus \{ \stackrel{s_{n+1}}{\longleftarrow} \}, & t_{n} = s_{n+1}=0 \\
a_
{n+1} = E_n \cap E_{n+2}: & \widetilde{Q}_1^{n} \setminus \{ \stackrel{-s_{n}-t_{n}}{\longrightarrow} \} \ = \ \widetilde{Q}_1^{n+2} \setminus \{ \stackrel{s_{n+2}}{\longleftarrow} \}, & s_{n}+t_{n} =s_{n+2}=0
\end{array}$$
In the $E_6$ case:
$$\begin{array}{rll}
a_1 = E_1 \cap E_2: & \widetilde{Q}_1^1 \setminus \{ \stackrel{s_1}{\longrightarrow} \} \ = \ \widetilde{Q}_1^2 \setminus \{ \stackrel{s_2 + t_2}{\longrightarrow} \}, & s_1 = s_2+t_2 =0 \\
a_3 = E_3 \cap E_2: & \widetilde{Q}_1^3 \setminus \{ \stackrel{s_3}{\longrightarrow} \} \ = \ \widetilde{Q}_1^2 \setminus \{ \stackrel{s_2}{\longrightarrow} \}, & s_3= s_2=0 \\
a_4 = E_4 \cap E_2: & \widetilde{Q}_1^4 \setminus \{ \stackrel{t_4}{\longrightarrow} \} \ = \ \widetilde{Q}_1^2 \setminus \{ \stackrel{t_2}{\longrightarrow} \}, & t_4 =t_2=0 \\
a_5 = E_3 \cap E_5: & \widetilde{Q}^3_1 \setminus \{ \stackrel{t_3}{\longleftarrow}, \ \stackrel{t_3}{\longrightarrow} \} \ = \ \widetilde{Q}_1^5 \setminus \{ \stackrel{s_5}{\longrightarrow} \}, & t_3 =s_5 = 0 \\
a_6 = E_4 \cap E_6: & \widetilde{Q}^4_1 \setminus \{ \stackrel{s_4}{\longleftarrow}, \ \stackrel{s_4}{\longrightarrow} \} \ = \ \widetilde{Q}_1^6 \setminus \{ \stackrel{t_6}{\longrightarrow} \}, & s_4 = t_6 = 0.
\end{array}$$
\textit{Claim II: $Y \supseteq \widetilde{Y}$.}\\
Consider the moduli space $\mathcal{M}_d^{\theta}(A)$ of stable $A$-modules with generic stability parameter $\theta = \left( -1 + \sum_{i \in Q_0}d_i, -1, \ldots, -1 \right) \in \mathbb{Z}^{|Q_0|}$, where the first component is $\theta_k$.  This choice of $\theta$ is equivalent to restricting to modules in $\operatorname{Rep}_d(A)$ whose socles are isomorphic to $S_k$, and so any almost large module with socle $S_k$ will be $\theta$-stable.  But $\mathcal{M}_d^{\theta}(A)$ is precisely $Y$ by \cite[Corollary 3.12]{K}, proving our claim.  This is also implies that the path-like set $\mathcal{P} = Q_{\geq 0} \cup \{0\}$ is sufficient for determining all almost large modules, since the almost large modules with socle $S_k$ obtained from $Q_{\geq 0} \cup \{0\}$ exhaust the set of all modules in $\operatorname{Rep}_dA$ with socle $S_k$.
\begin{figure}
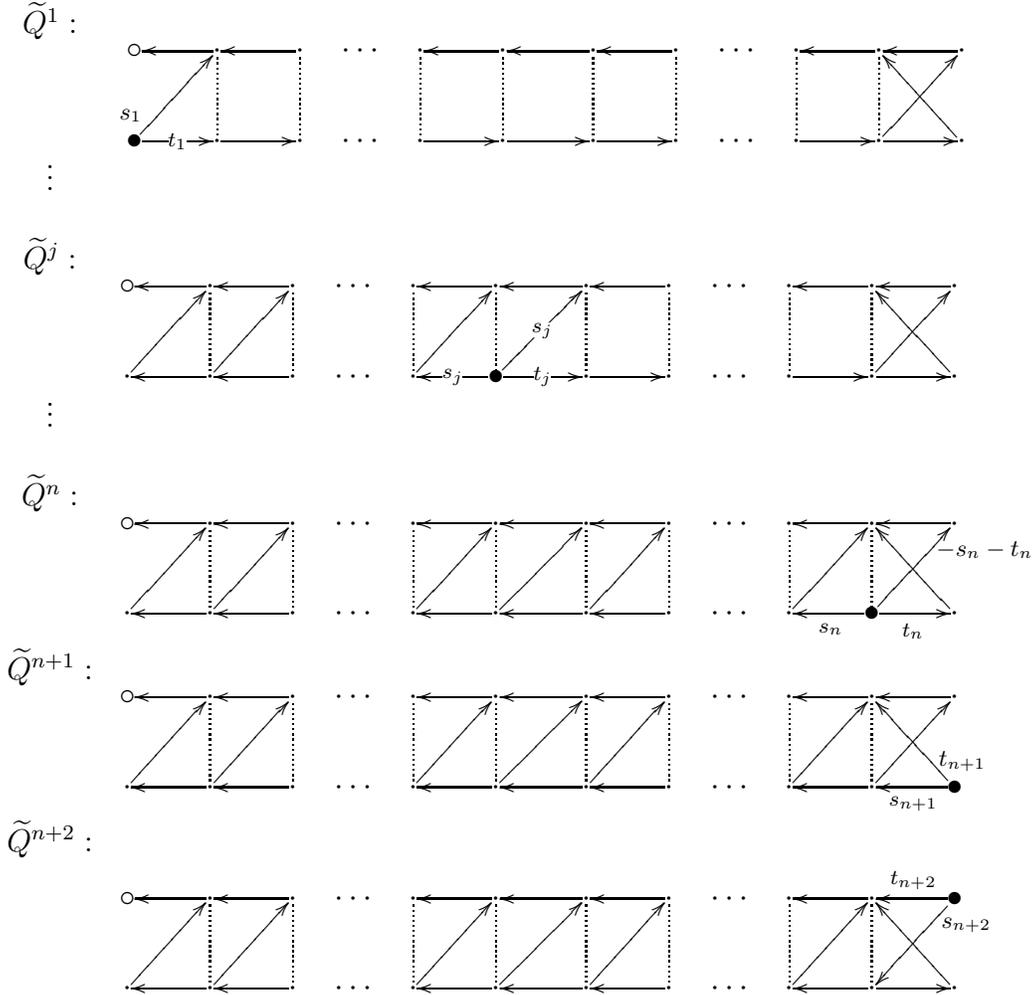

$$\begin{array}{cl}
\widetilde{Q}^1: & \\
& \xy
(-55,6)*{\circ}="1";(-44,6)*{\cdot}="2";(-33,6)*{\cdot}="3";(-17,6)*{\cdot}="4";(-6,6)*{\cdot}="5";
(6,6)*{\cdot}="6";(17,6)*{\cdot}="7";(33,6)*{\cdot}="8";(44,6)*{\cdot}="9";(55,6)*{\cdot}="10";
(-55,-6)*{\bullet}="11";(-44,-6)*{\cdot}="12";(-33,-6)*{\cdot}="13";(-17,-6)*{\cdot}="14";(-6,-6)*{\cdot}="15";
(6,-6)*{\cdot}="16";(17,-6)*{\cdot}="17";(33,-6)*{\cdot}="18";(44,-6)*{\cdot}="19";(55,-6)*{\cdot}="20";
(25,6)*{\cdots}="";(-25,6)*{\cdots}="";(25,-6)*{\cdots}="";(-25,-6)*{\cdots}="";
{\ar"2";"1"};{\ar"3";"2"};{\ar"5";"4"};{\ar"6";"5"};{\ar"7";"6"};{\ar"9";"8"};{\ar"10";"9"};
{\ar|-{t_1}"11";"12"};{\ar"12";"13"};{\ar"14";"15"};{\ar"15";"16"};{\ar"16";"17"};{\ar"18";"19"};{\ar"19";"20"};
{\ar"11";"2"};
{\ar"19";"10"};{\ar"20";"9"};
{\ar@{..}"2";"12"};{\ar@{..}"3";"13"};{\ar@{..}"4";"14"};{\ar@{..}"5";"15"};{\ar@{..}"6";"16"};{\ar@{..}"7";"17"};{\ar@{..}"8";"18"};{\ar@{..}"9";"19"};
(-55.5,-2.5)*{\text{\scriptsize{$s_1$}}}="";
\endxy
\\
\vdots & \\
\\
\widetilde{Q}^j: & \\
& \xy
(-55,6)*{\circ}="1";(-44,6)*{\cdot}="2";(-33,6)*{\cdot}="3";(-17,6)*{\cdot}="4";(-6,6)*{\cdot}="5";
(6,6)*{\cdot}="6";(17,6)*{\cdot}="7";(33,6)*{\cdot}="8";(44,6)*{\cdot}="9";(55,6)*{\cdot}="10";
(-55,-6)*{\cdot}="11";(-44,-6)*{\cdot}="12";(-33,-6)*{\cdot}="13";(-17,-6)*{\cdot}="14";(-6,-6)*{\bullet}="15";
(6,-6)*{\cdot}="16";(17,-6)*{\cdot}="17";(33,-6)*{\cdot}="18";(44,-6)*{\cdot}="19";(55,-6)*{\cdot}="20";
(25,6)*{\cdots}="";(-25,6)*{\cdots}="";(25,-6)*{\cdots}="";(-25,-6)*{\cdots}="";
{\ar"2";"1"};{\ar"3";"2"};{\ar"5";"4"};{\ar"6";"5"};{\ar"7";"6"};{\ar"9";"8"};{\ar"10";"9"};
{\ar"12";"11"};{\ar"13";"12"};{\ar|-{s_j}"15";"14"};{\ar|-{t_j}"15";"16"};{\ar"16";"17"};{\ar"18";"19"};{\ar"19";"20"};
{\ar"11";"2"};{\ar"12";"3"};{\ar"14";"5"};{\ar|-{s_j}"15";"6"};
{\ar"19";"10"};{\ar"20";"9"};
{\ar@{..}"2";"12"};{\ar@{..}"3";"13"};{\ar@{..}"4";"14"};{\ar@{..}"5";"15"};{\ar@{..}"6";"16"};{\ar@{..}"7";"17"};{\ar@{..}"8";"18"};{\ar@{..}"9";"19"};
\endxy
\\
\vdots & \\
\\
\widetilde{Q}^{n}: & \\
& \xy
(-55,6)*{\circ}="1";(-44,6)*{\cdot}="2";(-33,6)*{\cdot}="3";(-17,6)*{\cdot}="4";(-6,6)*{\cdot}="5";
(6,6)*{\cdot}="6";(17,6)*{\cdot}="7";(33,6)*{\cdot}="8";(44,6)*{\cdot}="9";(55,6)*{\cdot}="10";
(-55,-6)*{\cdot}="11";(-44,-6)*{\cdot}="12";(-33,-6)*{\cdot}="13";(-17,-6)*{\cdot}="14";(-6,-6)*{\cdot}="15";
(6,-6)*{\cdot}="16";(17,-6)*{\cdot}="17";(33,-6)*{\cdot}="18";(44,-6)*{\bullet}="19";(55,-6)*{\cdot}="20";
(25,6)*{\cdots}="";(-25,6)*{\cdots}="";(25,-6)*{\cdots}="";(-25,-6)*{\cdots}="";
{\ar"2";"1"};{\ar"3";"2"};{\ar"5";"4"};{\ar"6";"5"};{\ar"7";"6"};{\ar"9";"8"};{\ar"10";"9"};
{\ar"12";"11"};{\ar"13";"12"};{\ar"15";"14"};{\ar"16";"15"};{\ar"17";"16"};{\ar^{s_{n}}"19";"18"};{\ar_{t_{n}}"19";"20"};
{\ar"11";"2"};{\ar"12";"3"};{\ar"14";"5"};{\ar"15";"6"};{\ar"16";"7"};{\ar"18";"9"};
{\ar"19";"10"};{\ar"20";"9"};
{\ar@{..}"2";"12"};{\ar@{..}"3";"13"};{\ar@{..}"4";"14"};{\ar@{..}"5";"15"};{\ar@{..}"6";"16"};{\ar@{..}"7";"17"};{\ar@{..}"8";"18"};{\ar@{..}"9";"19"};
(59,2.5)*{\text{\scriptsize{$-s_{n}-t_{n}$}}}="";
\endxy
\\
\widetilde{Q}^{n+1}: & \\
& \xy
(-55,6)*{\circ}="1";(-44,6)*{\cdot}="2";(-33,6)*{\cdot}="3";(-17,6)*{\cdot}="4";(-6,6)*{\cdot}="5";
(6,6)*{\cdot}="6";(17,6)*{\cdot}="7";(33,6)*{\cdot}="8";(44,6)*{\cdot}="9";(55,6)*{\cdot}="10";
(-55,-6)*{\cdot}="11";(-44,-6)*{\cdot}="12";(-33,-6)*{\cdot}="13";(-17,-6)*{\cdot}="14";(-6,-6)*{\cdot}="15";
(6,-6)*{\cdot}="16";(17,-6)*{\cdot}="17";(33,-6)*{\cdot}="18";(44,-6)*{\cdot}="19";(55,-6)*{\bullet}="20";
(25,6)*{\cdots}="";(-25,6)*{\cdots}="";(25,-6)*{\cdots}="";(-25,-6)*{\cdots}="";
{\ar"2";"1"};{\ar"3";"2"};{\ar"5";"4"};{\ar"6";"5"};{\ar"7";"6"};{\ar"9";"8"};{\ar"10";"9"};
{\ar"12";"11"};{\ar"13";"12"};{\ar"15";"14"};{\ar"16";"15"};{\ar"17";"16"};{\ar"19";"18"};{\ar^{s_{n+1}}"20";"19"};
{\ar"11";"2"};{\ar"12";"3"};{\ar"14";"5"};{\ar"15";"6"};{\ar"16";"7"};{\ar"18";"9"};
{\ar"19";"10"};{\ar"20";"9"};
(56,-2.5)*{\text{\scriptsize{$t_{n+1}$}}}="";
{\ar@{..}"2";"12"};{\ar@{..}"3";"13"};{\ar@{..}"4";"14"};{\ar@{..}"5";"15"};{\ar@{..}"6";"16"};{\ar@{..}"7";"17"};{\ar@{..}"8";"18"};{\ar@{..}"9";"19"};
\endxy
\\
\widetilde{Q}^{n+2}: & \\
& \xy
(-55,6)*{\circ}="1";(-44,6)*{\cdot}="2";(-33,6)*{\cdot}="3";(-17,6)*{\cdot}="4";(-6,6)*{\cdot}="5";
(6,6)*{\cdot}="6";(17,6)*{\cdot}="7";(33,6)*{\cdot}="8";(44,6)*{\cdot}="9";(55,6)*{\bullet}="10";
(-55,-6)*{\cdot}="11";(-44,-6)*{\cdot}="12";(-33,-6)*{\cdot}="13";(-17,-6)*{\cdot}="14";(-6,-6)*{\cdot}="15";
(6,-6)*{\cdot}="16";(17,-6)*{\cdot}="17";(33,-6)*{\cdot}="18";(44,-6)*{\cdot}="19";(55,-6)*{\cdot}="20";
(25,6)*{\cdots}="";(-25,6)*{\cdots}="";(25,-6)*{\cdots}="";(-25,-6)*{\cdots}="";
{\ar"2";"1"};{\ar"3";"2"};{\ar"5";"4"};{\ar"6";"5"};{\ar"7";"6"};{\ar"9";"8"};{\ar_{t_{n+2}}"10";"9"};
{\ar"12";"11"};{\ar"13";"12"};{\ar"15";"14"};{\ar"16";"15"};{\ar"17";"16"};{\ar"19";"18"};{\ar"19";"20"};
{\ar"11";"2"};{\ar"12";"3"};{\ar"14";"5"};{\ar"15";"6"};{\ar"16";"7"};{\ar"18";"9"};
{\ar"10";"19"};{\ar"20";"9"};
(56.5,2.5)*{\text{\scriptsize{$s_{n+2}$}}}="";
{\ar@{..}"2";"12"};{\ar@{..}"3";"13"};{\ar@{..}"4";"14"};{\ar@{..}"5";"15"};{\ar@{..}"6";"16"};{\ar@{..}"7";"17"};{\ar@{..}"8";"18"};{\ar@{..}"9";"19"};
\endxy
\\
\end{array}$$
\caption{The $n+2$ pulled-apart supporting subquivers of the almost large modules over the $D_{n+2}$ preprojective algebra, up to isomorphism.  Vertices connected by a dotted edge correspond to the same vertex in $Q_0$.  Vertex $k$ is denoted $\circ$, and each $\mathbb{P}^1$ shrinks to the vertex simple at the vertex denoted $\bullet$.}
\label{D_n pulled-apart sub}
\end{figure}
\begin{figure}
$$\begin{array}{cl}
\widetilde{Q}^1: & \\
& \xy
(0,-21.5)*+{\circ}="0";
(-7,-6.5)*+{\bullet}="1";(7,-6.5)*+{\cdot}="1'";
(0,13.5)*+{\cdot}="3";
(-14,13.5)*+{\cdot}="2";(14,13.5)*+{\cdot}="2'";
(-28,6.5)*+{\cdot}="4";(28,6.5)*+{\cdot}="4'";(-28,20.5)*+{\cdot}="5";(28,20.5)*+{\cdot}="5'";
(-42,13.5)*+{\cdot}="6";(42,13.5)*+{\cdot}="6'";
{\ar@{..}"1";"1'"};
{\ar@{..}"2";"3"};
{\ar@{..}"3";"2'"};
{\ar@{..}"5";"4"};
{\ar@{..}"5'";"4'"};
{\ar|-{}"1'";"0"};
{\ar|-{s_1}"1";"2"};
{\ar|-{t_1}"1";"3"};
{\ar|-{}"4";"6"};
{\ar|-{}"6";"5"};
{\ar|-{}"4'";"6'"};
{\ar|-{}"6'";"5'"};
{\ar|-{}"5'";"2'"};
{\ar|-{}"2";"4"};
{\ar@/_/|-{}"4";"3"};
{\ar@/^/|-{}"4'";"3"};
{\ar@/_/|-{}"3";"5"};
{\ar@/^/|-{}"3";"5'"};
{\ar@/^1.3pc/|-{}"5";"2'"};
{\ar@/_1.3pc/|-{}"2";"4'"};
{\ar|-{}"2'";"1'"};
(-3.5,0)*{}="";
\endxy
\\
\widetilde{Q}^2: & \\
& \xy
(0,-21.5)*+{\circ}="0";
(-7,-6.5)*+{\cdot}="1";(7,-6.5)*+{\cdot}="1'";
(0,13.5)*+{\cdot}="3";
(-14,13.5)*+{\bullet}="2";(14,13.5)*+{\cdot}="2'";
(-28,6.5)*+{\cdot}="4";(28,6.5)*+{\cdot}="4'";(-28,20.5)*+{\cdot}="5";(28,20.5)*+{\cdot}="5'";
(-42,13.5)*+{\cdot}="6";(42,13.5)*+{\cdot}="6'";
{\ar@{..}"1";"1'"};
{\ar@{..}"2";"3"};
{\ar@{..}"3";"2'"};
{\ar@{..}"5";"4"};
{\ar@{..}"5'";"4'"};
{\ar|-{}"1'";"0"};
{\ar|-{-s_2-t_2}"2";"1"};
{\ar"1";"3"};
{\ar|-{}"4";"6"};
{\ar|-{}"6";"5"};
{\ar|-{}"4'";"6'"};
{\ar|-{}"6'";"5'"};
{\ar|-{}"5'";"2'"};
{\ar|-{s_2}"2";"4"};
{\ar@/_/|-{}"4";"3"};
{\ar@/^/|-{}"4'";"3"};
{\ar@/_/|-{}"3";"5"};
{\ar@/^/|-{}"3";"5'"};
{\ar@/^1.3pc/|-{}"5";"2'"};
{\ar@/_1.3pc/|-{ \ t_2}"2";"4'"};
{\ar|-{}"2'";"1'"};
(-3.5,0)*{}="";
\endxy
\\
\widetilde{Q}^3: & \\
& \xy
(0,-21.5)*+{\circ}="0";
(-7,-6.5)*+{\cdot}="1";(7,-6.5)*+{\cdot}="1'";
(0,13.5)*+{\cdot}="3";
(-14,13.5)*+{\cdot}="2";(14,13.5)*+{\cdot}="2'";
(-28,6.5)*+{\bullet}="4";(28,6.5)*+{\cdot}="4'";(-28,20.5)*+{\cdot}="5";(28,20.5)*+{\cdot}="5'";
(-42,13.5)*+{\cdot}="6";(42,13.5)*+{\cdot}="6'";
{\ar@{..}"1";"1'"};
{\ar@{..}"2";"3"};
{\ar@{..}"3";"2'"};
{\ar@{..}"5";"4"};
{\ar@{..}"5'";"4'"};
{\ar|-{}"1'";"0"};
{\ar|-{}"2";"1"};
{\ar"1";"3"};
{\ar|-{t_3}"4";"6"};
{\ar|-{}"6";"5"};
{\ar|-{}"4'";"6'"};
{\ar|-{}"6'";"5'"};
{\ar|-{}"5'";"2'"};
{\ar|-{s_3}"4";"2"};
{\ar@/_/|-{t_3}"4";"3"};
{\ar@/^/|-{}"4'";"3"};
{\ar@/_/|-{}"3";"5"};
{\ar@/^/|-{}"3";"5'"};
{\ar@/^1.3pc/|-{}"5";"2'"};
{\ar@/_1.3pc/|-{}"2";"4'"};
{\ar|-{}"2'";"1'"};
(-3.5,0)*{}="";
\endxy
\end{array}$$
\end{figure}

\begin{figure}
$$\begin{array}{cl}
\widetilde{Q}^4: & \\
& \xy
(0,-21.5)*+{\circ}="0";
(-7,-6.5)*+{\cdot}="1";(7,-6.5)*+{\cdot}="1'";
(0,13.5)*+{\cdot}="3";
(-14,13.5)*+{\cdot}="2";(14,13.5)*+{\cdot}="2'";
(-28,6.5)*+{\cdot}="4";(28,6.5)*+{\bullet}="4'";(-28,20.5)*+{\cdot}="5";(28,20.5)*+{\cdot}="5'";
(-42,13.5)*+{\cdot}="6";(42,13.5)*+{\cdot}="6'";
{\ar@{..}"1";"1'"};
{\ar@{..}"2";"3"};
{\ar@{..}"3";"2'"};
{\ar@{..}"5";"4"};
{\ar@{..}"5'";"4'"};
{\ar|-{}"1'";"0"};
{\ar|-{}"2";"1"};
{\ar"1";"3"};
{\ar|-{}"4";"6"};
{\ar|-{}"6";"5"};
{\ar|-{s_4}"4'";"6'"};
{\ar|-{}"6'";"5'"};
{\ar|-{}"5'";"2'"};
{\ar|-{}"2";"4"};
{\ar@/_/|-{}"4";"3"};
{\ar@/^/|-{s_4}"4'";"3"};
{\ar@/_/|-{}"3";"5"};
{\ar@/^/|-{}"3";"5'"};
{\ar@/^1.3pc/|-{}"5";"2'"};
{\ar@/^1.3pc/|-{ \ t_4}"4'";"2"};
{\ar|-{}"2'";"1'"};
(-3.5,0)*{}="";
\endxy
\\
\widetilde{Q}^{5}: & \\
& \xy
(0,-21.5)*+{\circ}="0";
(-7,-6.5)*+{\cdot}="1";(7,-6.5)*+{\cdot}="1'";
(0,13.5)*+{\cdot}="3";
(-14,13.5)*+{\cdot}="2";(14,13.5)*+{\cdot}="2'";
(-28,6.5)*+{\cdot}="4";(28,6.5)*+{\cdot}="4'";(-28,20.5)*+{\cdot}="5";(28,20.5)*+{\cdot}="5'";
(-42,13.5)*+{\bullet}="6";(42,13.5)*+{\cdot}="6'";
{\ar@{..}"1";"1'"};
{\ar@{..}"2";"3"};
{\ar@{..}"3";"2'"};
{\ar@{..}"5";"4"};
{\ar@{..}"5'";"4'"};
{\ar@/^/"3";"5'"};
{\ar"1'";"0"};
{\ar"2'";"1'"};{\ar"2";"1"};{\ar"1";"3"};
{\ar|-{s_5}"6";"5"};{\ar|-{t_5}"6";"4"};
{\ar@/^1.3pc/"5";"2'"};{\ar@/_/"3";"5"};{\ar"4";"2"};
{\ar"4'";"6'"};{\ar"6'";"5'"};
{\ar"5'";"2'"};{\ar@/^/"4'";"3"};{\ar@/_1.3pc/"2";"4'"};
\endxy
\\
\widetilde{Q}^{6}: & \\
& \xy
(0,-21.5)*+{\circ}="0";
(7,-6.5)*+{\cdot}="1";(-7,-6.5)*+{\cdot}="1'";
(0,13.5)*+{\cdot}="3";
(14,13.5)*+{\cdot}="2";(-14,13.5)*+{\cdot}="2'";
(28,6.5)*+{\cdot}="5";(-28,6.5)*+{\cdot}="5'";(28,20.5)*+{\cdot}="4";(-28,20.5)*+{\cdot}="4'";
(42,13.5)*+{\bullet}="6";(-42,13.5)*+{\cdot}="6'";
{\ar@{..}"1";"1'"};
{\ar@{..}"2";"3"};
{\ar@{..}"3";"2'"};
{\ar@{..}"5";"4"};
{\ar@{..}"5'";"4'"};
{\ar@/^/|-{}"3";"5'"};
{\ar"1'";"0"};
{\ar"2'";"1'"};{\ar"2";"1"};{\ar"1";"3"};
{\ar|-{t_6}"6";"5"};{\ar|-{s_6}"6";"4"};
{\ar@/^1.3pc/"5";"2'"};{\ar@/_/"3";"5"};{\ar"4";"2"};
{\ar"4'";"6'"};{\ar"6'";"5'"};
{\ar"5'";"2'"};{\ar@/^/"4'";"3"};{\ar@/_1.3pc/"2";"4'"};
\endxy
\end{array}$$
\caption{The $6$ pulled-apart supporting subquivers of the almost large modules over the $E_6$ preprojective algebra, up to isomorphism.  Vertices connected by a dotted edge correspond to the same vertex in $Q_0$.  Vertex $k$ is denoted $\circ$, and each $\mathbb{P}^1$ shrinks to the vertex simple at the vertex denoted $\bullet$.} 
\label{E_6 pulled-apart sub}
\end{figure}
\end{proof}

\subsection{A non-isolated quotient singularity} \label{A non-isolated quotient singularity}

Consider the linear action of the finite abelian group $G = \mu_r^{\oplus 2} = \left\langle g_1, g_2 \right\rangle$ on $\mathbb{C}[x,y,z]$ by the representation 
$$\rho(G) = \left\langle \rho(g_1) = \operatorname{diag}\left(\omega, \omega^{-1},1 \right), \rho(g_2) = \operatorname{diag}\left(1, \omega^{-1}, \omega \right)
 \right\rangle \subset \operatorname{SU}_3(\mathbb{C}),$$
where $\omega$ is a primitive $r$th root of unity.  The ring of invariants $R:= \mathbb{C}[x,y,z]^{\rho(G)}$ is the coordinate ring for the non-isolated quotient singularity $\mathbb{C}^3/\rho(G) := \operatorname{Max}R$, which is a 3 dimensional version of the $A_n$ singularity (see \cite[Example 2.2]{Reid2}).  Here we take $r = 4$.

We will find that the resolution of $\mathbb{C}^3/\rho(G)$ determined by the basic triangulation of its toric diagram, given in figure \ref{toric diagram}, parameterizes the large modules and almost large modules with isomorphic 1-dimensional socles over the McKay quiver algebra of $(G,\rho)$.\footnote{Note that the 3 regular hexagons in the triangulation correspond to 3 del Pezzo surfaces of degree 6, that is, 3 $\mathbb{P}^2$'s blown up at 3 points not on a line.}  The McKay quiver $Q$ of $(G, \rho)$ is determined by noting that there are $r^2$ irreducible representations $\rho_{ij}$ of $G$, all of which are 1-dimensional,
$$\rho_{ij}(g_1) = \omega^i, \ \ \ \ \rho_{ij}(g_2) = \omega^j.$$
$Q$ may be drawn on a two-torus as shown in figure \ref{McKay for G}.  Denote by $a_i,b_i,c_i \in Q_1e_i$ the respective arrows that head up, right, and downward to the left, and set $a:= \sum_{i \in Q_0}a_i$, $b:= \sum_{i \in Q_0}b_i$, and $c:=\sum_{i \in Q_0}c_i$.  The McKay quiver algebra of $(G, \rho)$ is then
$$A = \mathbb{C}Q/\left\langle ab-ba, bc-cb, ca-ac \right\rangle.$$  
By \cite[Theorem 3.7, with $(x,y,z) = (x_1,y_1,x_2y_2)$]{B}, the large $A$-modules have dimension vector $(1,\ldots,1)$, and an impression $(\tau, \mathbb{C}[x,y,z])$ of $A$ is given by the labeling of arrows
$$\bar{\tau}(a_i) = x, \ \ \ \bar{\tau}(b_i) = y, \ \ \ \bar{\tau}(c_i) = z.$$

The following proposition extends the fact that the large $A$-modules are parameterized by the smooth locus of $\mathbb{C}^3/\rho(G)$.

\begin{figure}
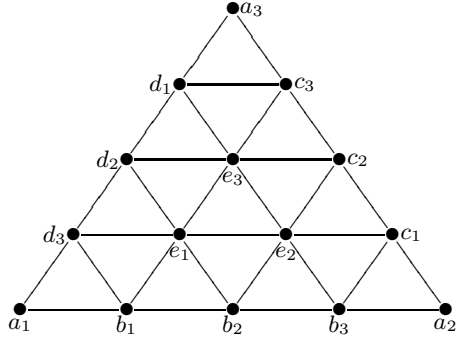

$$\xy
(-28.284,-20)*{\bullet}="1";(-14.142,-20)*{\bullet}="2";(0,-20)*{\bullet}="3";(14.142,-20)*{\bullet}="4";(28.284,-20)*{\bullet}="5";
(-21.213,-10)*{\bullet}="6";(-7.071,-10)*{\bullet}="7";(7.071,-10)*{\bullet}="8";(21.213,-10)*{\bullet}="9";
(-14.142,0)*{\bullet}="10";(0,0)*{\bullet}="11";(14.142,0)*{\bullet}="12";
(-7.071,10)*{\bullet}="13";(7.071,10)*{\bullet}="14";
(0,20)*{\bullet}="15";
{\ar@{-}"1";"2"};{\ar@{-}"2";"3"};{\ar@{-}"3";"4"};{\ar@{-}"4";"5"};
{\ar@{-}"6";"7"};{\ar@{-}"7";"8"};{\ar@{-}"8";"9"};
{\ar@{-}"10";"11"};{\ar@{-}"11";"12"};
{\ar@{-}"13";"14"};
{\ar@{-}"1";"6"};{\ar@{-}"2";"6"};{\ar@{-}"2";"7"};{\ar@{-}"3";"7"};{\ar@{-}"3";"8"};{\ar@{-}"4";"8"};{\ar@{-}"4";"9"};{\ar@{-}"8";"9"};
{\ar@{-}"6";"10"};{\ar@{-}"7";"10"};{\ar@{-}"7";"11"};{\ar@{-}"8";"11"};{\ar@{-}"8";"12"};{\ar@{-}"9";"12"};
{\ar@{-}"10";"13"};{\ar@{-}"11";"13"};{\ar@{-}"11";"14"};{\ar@{-}"12";"14"};
{\ar@{-}"13";"15"};{\ar@{-}"14";"15"};{\ar@{-}"5";"9"};
(-14.142,-22)*{\text{\scriptsize{$b_1$}}}="";(0,-22)*{\text{\scriptsize{$b_2$}}}="";(14.142,-22)*{\text{\scriptsize{$b_3$}}}="";
(23.713,-10)*{\text{\scriptsize{$c_1$}}}="";(16.642,0)*{\text{\scriptsize{$c_2$}}}="";(9.571,10)*{\text{\scriptsize{$c_3$}}}="";
(-23.713,-10)*{\text{\scriptsize{$d_3$}}}="";(-16.642,0)*{\text{\scriptsize{$d_2$}}}="";(-9.571,10)*{\text{\scriptsize{$d_1$}}}="";
(-7.071,-12.5)*{\text{\scriptsize{$e_1$}}}="";(7.071,-12.5)*{\text{\scriptsize{$e_2$}}}="";(0,-2.5)*{\text{\scriptsize{$e_3$}}}="";
(-28.284,-22)*{\text{\scriptsize{$a_1$}}}="";(28.284,-22)*{\text{\scriptsize{$a_2$}}}="";(2.5,20)*{\text{\scriptsize{$a_3$}}}="";
\endxy$$
\caption{The basic triangulation of the toric diagram for the resolution of $\mathbb{C}^3/\rho(G)$ that parameterizes the large modules and almost large modules with isomorphic 1-dimensional socles.}
\label{toric diagram}
\end{figure}

\begin{Proposition} \label{3d A} 
Let $A=\mathbb{C}Q/I$ be the McKay quiver algebra for $\left(\mu_4^{\oplus 2}, \rho \right)$, and let $\pi: Y \rightarrow \mathbb{C}^3/\rho(\mu_4^{\oplus 2})$ be the resolution determined by the basic triangulation of the toric diagram in figure \ref{toric diagram}.  Then the exceptional locus $E$ parameterizes the almost large $A$-modules with socle isomorphic to any fixed vertex simple.
\end{Proposition}

\begin{proof}
Recall that the large $A$-modules have dimension vector $(1,\ldots, 1)$.  Denote by $\mathcal{P}$ the path-like set $Q_{\geq 0} \cup \{0\}$.  Since $\operatorname{dim}Z = 3$, we must consider $\ell_{\mathcal{P}} =1,2,3$ almost large modules.  Fix a vertex $0 \in Q_0$, denoted $\circ$ in figures \ref{G1} - \ref{G3}; here each subquiver is drawn on a two-torus.

\textit{$\ell_{\mathcal{P}} = 1$ almost large modules.}  The supporting subquivers for the $\ell_{\mathcal{P}} =1$ large modules are displayed in figure \ref{G1}, while the supporting subquivers for $\ell_{\mathcal{P}} = 1$ almost large modules with socle $S_0$ are displayed in figures \ref{G2} and \ref{G3}, where by a ``$\mathbb{P}^n$-family'' we really mean a family parameterized by $\mathbb{P}^n$ minus the $n+1$ points of where one of the coordinates is zero.  These subquivers are determined as follows: Let $V \in \operatorname{Rep}_{(1,\ldots,1)}A$ be an $\ell_{\mathcal{P}} = 1$ almost large module.  Then there is some arrow $a$ that annihilates $V$ since the dimension vector of $V$ is $(1,\ldots,1)$.  For each $i \in Q_0$, denote by $\gamma_i \in e_iAe_i$ the unique cycle (modulo $I$) at vertex $i$ of length 3.  Since $a$ annihilates $V$, the cycle $\gamma_{j}$ containing $a$ as a subpath also annihilates $V$, and since $\sum_{i \in Q_0}\gamma_i$ is in the center of $A$ by \cite[Theorem 2.7]{B}, each cycle $\gamma_i$ must annihilate $V$.  But again since the dimension vector of $V$ is $(1, \ldots,1)$, at least one arrow in each cycle $\gamma_i$ must annihilate $V$.  Thus the supporting subquivers for the $\ell_{\mathcal{P}} = 1$ modules have at least one arrow removed from each cycle of length 3.  If a cycle of length 3 has two arrows removed, then we will find below that such a subquiver supports an $\ell_{\mathcal{P}} = 2$ almost large module.

\textit{$\ell_{\mathcal{P}} = 2$ almost large modules.}  The supporting subquivers for all $\ell_{\mathcal{P}} = 2$ almost large modules with socle $S_0$ are also displayed in figures \ref{G1} - \ref{G3}, and they are obtained as follows.  Suppose two vertices in the toric diagram (figure \ref{toric diagram}), say $g$ and $h$, are connected by an edge.  Then the irreducible components of the exceptional locus corresponding to $g$ and $h$ have nonempty intersection, and an open subset of this intersection parameterizes the $\ell_{\mathcal{P}} = 2$ almost large module isoclasses with supporting subquivers having vertex set $Q_0$ and arrow set
$$g \cap h: Q^g_1 \setminus \left\{ \text{ arrows labeled by } i \right\} = Q^h_1 \setminus \left\{ \text{ arrows labeled by } j \right\},$$
where $Q^g$, $Q^h$, and the labels $i$ and $j$ are displayed in figures \ref{G1} - \ref{G3}.  The following table verifies this explicitly.
$$\begin{array}{|cccc|cccc|cccc|cccc|cccc|}
g & h & i & j & g & h & i & j & g & h & i & j & g & h & i & j & g & h & i & j \\
\hline
a_1 & d_3 & 2 & 2 &
b_2 & e_2 & 4 & 6 &
c_1 & e_2 & 4 & 2 &
c_3 & d_1 & 4 & 4 &
d_2 & e_1 & 3 & 2 \\
a_1 & b_1 & 1 & 1 &
b_2 & b_3 & 2 & 1 &
c_1 & c_2 & 3 & 1 &
c_3 & a_3 & 3 & 1 &
d_2 & d_3 & 2 & 1 \\
b_1 & d_3 & 3 & 3 &
b_3 & e_2 & 3 & 3 &
c_2 & e_2 & 2 & 5 &
a_3 & d_1 & 2 & 1 &
d_3 & e_1 & 4 & 4 \\
b_1 & e_1 & 4 & 3 &
b_3 & c_1 & 4 & 2 &
c_2 & e_3 & 4 & 5 &
d_1 & e_3 & 3 & 3 &
e_1 & e_2 & 5 & 1 \\
b_1 & b_2 & 2 & 1 &
b_3 & a_2 & 2 & 1 &
c_2 & c_3 & 3 & 1 &
d_1 & d_2 & 2 & 1 &
e_2 & e_3 & 4 & 4 \\
b_2 & e_1 & 3 & 1 &
a_2 & c_1 & 2 & 1 &
c_3 & e_3 & 2 & 2 &
d_2 & e_3 & 4 & 6 &
e_3 & e_1 & 1 & 6
\end{array}$$
One may check that all other $\ell_{\mathcal{P}} = 2$ almost large modules do not have socle $S_0$.

\textit{$\ell_{\mathcal{P}} = 3$ almost large modules.}  There are 8 $\ell_{\mathcal{P}} = 3$ almost large module isoclasses, and these correspond to the faces of the basic triangles in the toric diagram.  The supporting subquiver for such a module is obtained by intersecting the three subquivers corresponding to the vertices of the corresponding basic triangle, and all other $\ell_{\mathcal{P}} =3$ almost large modules do not have socle $S_0$.  We leave the verification to the reader.

It is clear that the almost large modules with socle $S_0$ obtained from the path-like set $\mathcal{P} = Q_{\geq 0} \cup \{0\}$ exhaust the set of all modules in $\operatorname{Rep}_{(1,\cdots,1)}A$ with socle $S_0$, and so no other path-like set need be considered.
\end{proof}

\begin{figure}
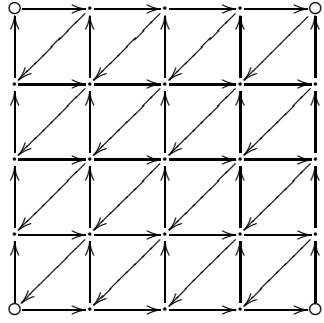

$$\xy
(-20,-20)*{\circ}="1";(-10,-20)*{\cdot}="2";(0,-20)*{\cdot}="3";(10,-20)*{\cdot}="4";(20,-20)*{\circ}="5";
(-20,-10)*{\cdot}="6";(-10,-10)*{\cdot}="7";(0,-10)*{\cdot}="8";(10,-10)*{\cdot}="9";(20,-10)*{\cdot}="10";
(-20,0)*{\cdot}="11";(-10,0)*{\cdot}="12";(0,0)*{\cdot}="13";(10,0)*{\cdot}="14";(20,0)*{\cdot}="15";
(-20,10)*{\cdot}="16";(-10,10)*{\cdot}="17";(0,10)*{\cdot}="18";(10,10)*{\cdot}="19";(20,10)*{\cdot}="20";
(-20,20)*{\circ}="21";(-10,20)*{\cdot}="22";(0,20)*{\cdot}="23";(10,20)*{\cdot}="24";(20,20)*{\circ}="25";
{\ar"1";"2"};{\ar"2";"3"};{\ar"3";"4"};{\ar"4";"5"};
{\ar"6";"7"};{\ar"7";"8"};{\ar"8";"9"};{\ar"9";"10"};
{\ar"11";"12"};{\ar"12";"13"};{\ar"13";"14"};{\ar"14";"15"};
{\ar"16";"17"};{\ar"17";"18"};{\ar"18";"19"};{\ar"19";"20"};
{\ar"21";"22"};{\ar"22";"23"};{\ar"23";"24"};{\ar"24";"25"};
{\ar"1";"6"};{\ar"6";"11"};{\ar"11";"16"};{\ar"16";"21"};
{\ar"2";"7"};{\ar"7";"12"};{\ar"12";"17"};{\ar"17";"22"};
{\ar"3";"8"};{\ar"8";"13"};{\ar"13";"18"};{\ar"18";"23"};
{\ar"4";"9"};{\ar"9";"14"};{\ar"14";"19"};{\ar"19";"24"};
{\ar"5";"10"};{\ar"10";"15"};{\ar"15";"20"};{\ar"20";"25"};
{\ar"22";"16"};{\ar"23";"17"};{\ar"17";"11"};
{\ar"24";"18"};{\ar"18";"12"};{\ar"12";"6"};
{\ar"25";"19"};{\ar"19";"13"};{\ar"13";"7"};{\ar"7";"1"};
{\ar"20";"14"};{\ar"14";"8"};{\ar"8";"2"};
{\ar"15";"9"};{\ar"9";"3"};
{\ar"10";"4"};
\endxy$$
\caption{The McKay quiver for $(G, \rho)$, drawn on a two-torus.  The vertices denoted $\circ$ are identified.}
\label{McKay for G}
\end{figure}

\begin{figure}
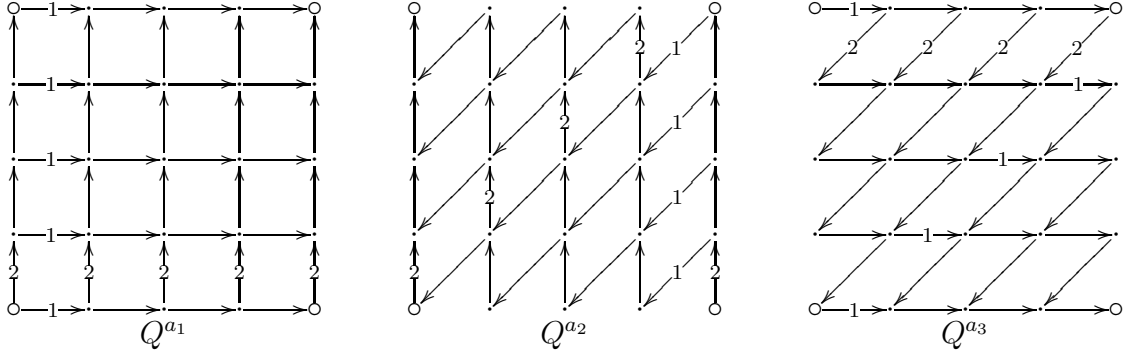

$$\begin{array}{ccccc}
\xy
(-20,-20)*{\circ}="1";(-10,-20)*{\cdot}="2";(0,-20)*{\cdot}="3";(10,-20)*{\cdot}="4";(20,-20)*{\circ}="5";
(-20,-10)*{\cdot}="6";(-10,-10)*{\cdot}="7";(0,-10)*{\cdot}="8";(10,-10)*{\cdot}="9";(20,-10)*{\cdot}="10";
(-20,0)*{\cdot}="11";(-10,0)*{\cdot}="12";(0,0)*{\cdot}="13";(10,0)*{\cdot}="14";(20,0)*{\cdot}="15";
(-20,10)*{\cdot}="16";(-10,10)*{\cdot}="17";(0,10)*{\cdot}="18";(10,10)*{\cdot}="19";(20,10)*{\cdot}="20";
(-20,20)*{\circ}="21";(-10,20)*{\cdot}="22";(0,20)*{\cdot}="23";(10,20)*{\cdot}="24";(20,20)*{\circ}="25";
{\ar|-{1}"1";"2"};{\ar"2";"3"};{\ar"3";"4"};{\ar"4";"5"};
{\ar|-{1}"6";"7"};{\ar"7";"8"};{\ar"8";"9"};{\ar"9";"10"};
{\ar|-{1}"11";"12"};{\ar"12";"13"};{\ar"13";"14"};{\ar"14";"15"};
{\ar|-{1}"16";"17"};{\ar"17";"18"};{\ar"18";"19"};{\ar"19";"20"};
{\ar|-{1}"21";"22"};{\ar"22";"23"};{\ar"23";"24"};{\ar"24";"25"};
{\ar|-{2}"1";"6"};{\ar"6";"11"};{\ar"11";"16"};{\ar"16";"21"};
{\ar|-{2}"2";"7"};{\ar"7";"12"};{\ar"12";"17"};{\ar"17";"22"};
{\ar|-{2}"3";"8"};{\ar"8";"13"};{\ar"13";"18"};{\ar"18";"23"};
{\ar|-{2}"4";"9"};{\ar"9";"14"};{\ar"14";"19"};{\ar"19";"24"};
{\ar|-{2}"5";"10"};{\ar"10";"15"};{\ar"15";"20"};{\ar"20";"25"};
\endxy
& \ \ \ 
&
\xy
(-20,-20)*{\circ}="1";(-10,-20)*{\cdot}="2";(0,-20)*{\cdot}="3";(10,-20)*{\cdot}="4";(20,-20)*{\circ}="5";
(-20,-10)*{\cdot}="6";(-10,-10)*{\cdot}="7";(0,-10)*{\cdot}="8";(10,-10)*{\cdot}="9";(20,-10)*{\cdot}="10";
(-20,0)*{\cdot}="11";(-10,0)*{\cdot}="12";(0,0)*{\cdot}="13";(10,0)*{\cdot}="14";(20,0)*{\cdot}="15";
(-20,10)*{\cdot}="16";(-10,10)*{\cdot}="17";(0,10)*{\cdot}="18";(10,10)*{\cdot}="19";(20,10)*{\cdot}="20";
(-20,20)*{\circ}="21";(-10,20)*{\cdot}="22";(0,20)*{\cdot}="23";(10,20)*{\cdot}="24";(20,20)*{\circ}="25";
{\ar|-{2}"1";"6"};{\ar"6";"11"};{\ar"11";"16"};{\ar"16";"21"};
{\ar"2";"7"};{\ar|-{2}"7";"12"};{\ar"12";"17"};{\ar"17";"22"};
{\ar"3";"8"};{\ar"8";"13"};{\ar|-{2}"13";"18"};{\ar"18";"23"};
{\ar"4";"9"};{\ar"9";"14"};{\ar"14";"19"};{\ar|-{2}"19";"24"};
{\ar|-{2}"5";"10"};{\ar"10";"15"};{\ar"15";"20"};{\ar"20";"25"};
{\ar"22";"16"};{\ar"23";"17"};{\ar"17";"11"};
{\ar"24";"18"};{\ar"18";"12"};{\ar"12";"6"};
{\ar|-{1}"25";"19"};{\ar"19";"13"};{\ar"13";"7"};{\ar"7";"1"};
{\ar|-{1}"20";"14"};{\ar"14";"8"};{\ar"8";"2"};
{\ar|-{1}"15";"9"};{\ar"9";"3"};
{\ar|-{1}"10";"4"};
\endxy
&
\ \ \ 
&
\xy
(-20,-20)*{\circ}="1";(-10,-20)*{\cdot}="2";(0,-20)*{\cdot}="3";(10,-20)*{\cdot}="4";(20,-20)*{\circ}="5";
(-20,-10)*{\cdot}="6";(-10,-10)*{\cdot}="7";(0,-10)*{\cdot}="8";(10,-10)*{\cdot}="9";(20,-10)*{\cdot}="10";
(-20,0)*{\cdot}="11";(-10,0)*{\cdot}="12";(0,0)*{\cdot}="13";(10,0)*{\cdot}="14";(20,0)*{\cdot}="15";
(-20,10)*{\cdot}="16";(-10,10)*{\cdot}="17";(0,10)*{\cdot}="18";(10,10)*{\cdot}="19";(20,10)*{\cdot}="20";
(-20,20)*{\circ}="21";(-10,20)*{\cdot}="22";(0,20)*{\cdot}="23";(10,20)*{\cdot}="24";(20,20)*{\circ}="25";
{\ar|-{1}"1";"2"};{\ar"2";"3"};{\ar"3";"4"};{\ar"4";"5"};
{\ar"6";"7"};{\ar|-{1}"7";"8"};{\ar"8";"9"};{\ar"9";"10"};
{\ar"11";"12"};{\ar"12";"13"};{\ar|-{1}"13";"14"};{\ar"14";"15"};
{\ar"16";"17"};{\ar"17";"18"};{\ar"18";"19"};{\ar|-{1}"19";"20"};
{\ar|-{1}"21";"22"};{\ar"22";"23"};{\ar"23";"24"};{\ar"24";"25"};
{\ar|-{2}"22";"16"};{\ar|-{2}"23";"17"};{\ar"17";"11"};
{\ar|-{2}"24";"18"};{\ar"18";"12"};{\ar"12";"6"};
{\ar|-{2}"25";"19"};{\ar"19";"13"};{\ar"13";"7"};{\ar"7";"1"};
{\ar"20";"14"};{\ar"14";"8"};{\ar"8";"2"};
{\ar"15";"9"};{\ar"9";"3"};
{\ar"10";"4"};
\endxy
\\
Q^{a_1} & & Q^{a_2} & & Q^{a_3}
\end{array}$$
\caption{
$Q^{a_1}$ (resp.\ $Q^{a_2}$, $Q^{a_3}$) supports the $\mathbb{C}^* \times \mathbb{C}^*$-family of $\ell_{\mathcal{P}} =1$ large modules parameterized by the vanishing of the single coordinate $z=0$ (resp.\ $x=0$, $y=0$) in the smooth locus of $\operatorname{Max}Z$, corresponding to the vertex $a_1$ (resp.\ $a_2$, $a_3$) in the toric diagram (figure \ref{toric diagram}).
For each $1 \leq i \leq 2$, the subquiver obtained by removing all arrows from $Q^{a_j}$ labeled $i$ supports the $\mathbb{C}^*$-family of $\ell_{\mathcal{P}} = 2$ almost large modules corresponding to an edge emanating from $a_j$ in the toric diagram.}
\label{G1}
\end{figure}

\begin{figure}
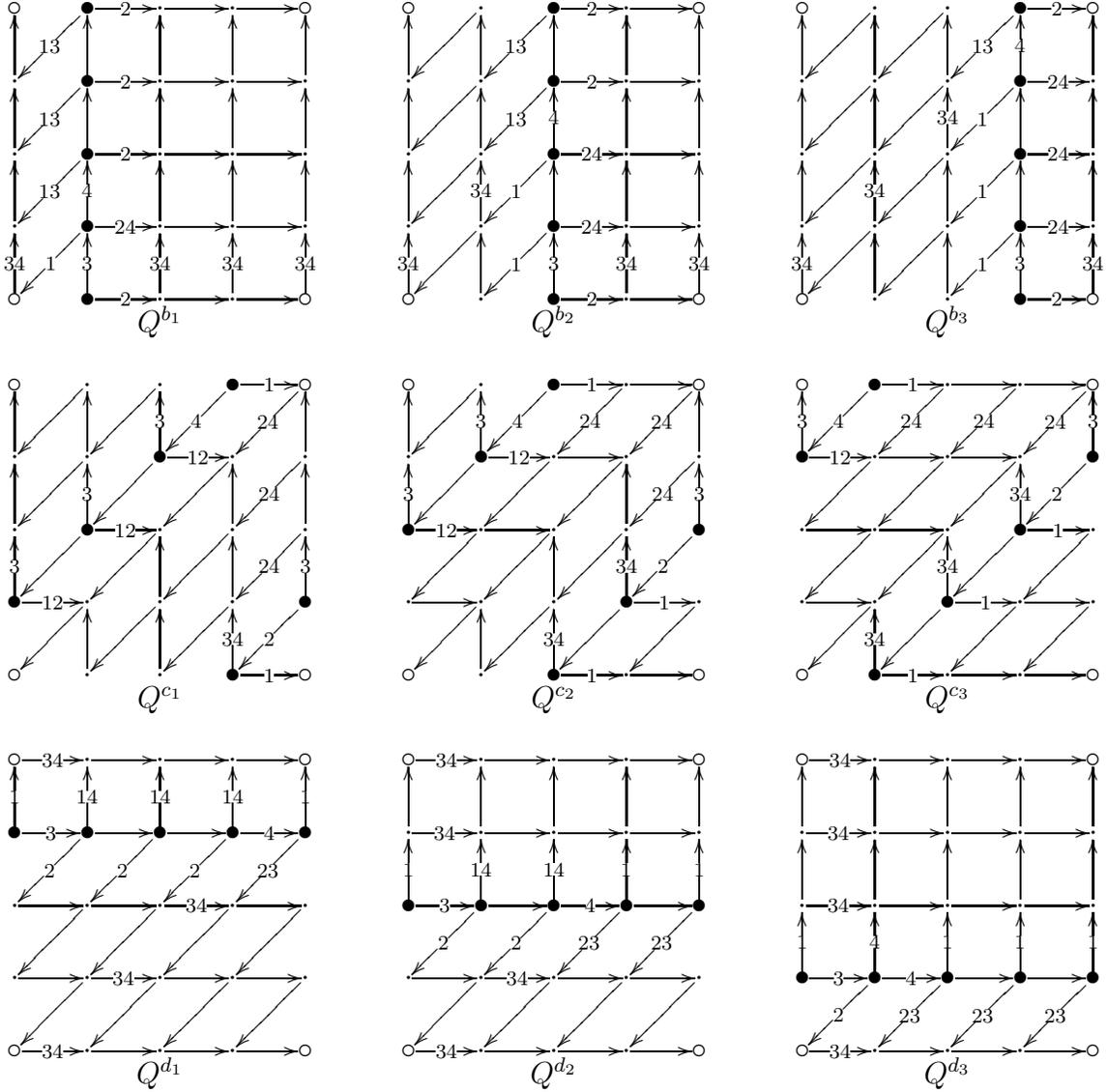

$$\begin{array}{ccccc}
\xy
(-20,-20)*{\circ}="1";(-10,-20)*{\bullet}="2";(0,-20)*{\cdot}="3";(10,-20)*{\cdot}="4";(20,-20)*{\circ}="5";
(-20,-10)*{\cdot}="6";(-10,-10)*{\bullet}="7";(0,-10)*{\cdot}="8";(10,-10)*{\cdot}="9";(20,-10)*{\cdot}="10";
(-20,0)*{\cdot}="11";(-10,0)*{\bullet}="12";(0,0)*{\cdot}="13";(10,0)*{\cdot}="14";(20,0)*{\cdot}="15";
(-20,10)*{\cdot}="16";(-10,10)*{\bullet}="17";(0,10)*{\cdot}="18";(10,10)*{\cdot}="19";(20,10)*{\cdot}="20";
(-20,20)*{\circ}="21";(-10,20)*{\bullet}="22";(0,20)*{\cdot}="23";(10,20)*{\cdot}="24";(20,20)*{\circ}="25";
{\ar|-{34}"1";"6"};{\ar"6";"11"};{\ar"11";"16"};{\ar"16";"21"};
{\ar|-{3}"2";"7"};{\ar|-{4}"7";"12"};{\ar"12";"17"};{\ar"17";"22"};
{\ar|-{34}"3";"8"};{\ar"8";"13"};{\ar"13";"18"};{\ar"18";"23"};
{\ar|-{34}"4";"9"};{\ar"9";"14"};{\ar"14";"19"};{\ar"19";"24"};
{\ar|-{34}"5";"10"};{\ar"10";"15"};{\ar"15";"20"};{\ar"20";"25"};
{\ar|-{2}"2";"3"};{\ar"3";"4"};{\ar"4";"5"};
{\ar|-{24}"7";"8"};{\ar"8";"9"};{\ar"9";"10"};
{\ar|-{2}"12";"13"};{\ar"13";"14"};{\ar"14";"15"};
{\ar|-{2}"17";"18"};{\ar"18";"19"};{\ar"19";"20"};
{\ar|-{2}"22";"23"};{\ar"23";"24"};{\ar"24";"25"};
{\ar|-{1}"7";"1"};{\ar|-{13}"12";"6"};{\ar|-{13}"17";"11"};{\ar|-{13}"22";"16"};
\endxy
&
\ \ \ 
&
\xy
(-20,-20)*{\circ}="1";(-10,-20)*{\cdot}="2";(0,-20)*{\bullet}="3";(10,-20)*{\cdot}="4";(20,-20)*{\circ}="5";
(-20,-10)*{\cdot}="6";(-10,-10)*{\cdot}="7";(0,-10)*{\bullet}="8";(10,-10)*{\cdot}="9";(20,-10)*{\cdot}="10";
(-20,0)*{\cdot}="11";(-10,0)*{\cdot}="12";(0,0)*{\bullet}="13";(10,0)*{\cdot}="14";(20,0)*{\cdot}="15";
(-20,10)*{\cdot}="16";(-10,10)*{\cdot}="17";(0,10)*{\bullet}="18";(10,10)*{\cdot}="19";(20,10)*{\cdot}="20";
(-20,20)*{\circ}="21";(-10,20)*{\cdot}="22";(0,20)*{\bullet}="23";(10,20)*{\cdot}="24";(20,20)*{\circ}="25";
{\ar|-{34}"1";"6"};{\ar"6";"11"};{\ar"11";"16"};{\ar"16";"21"};
{\ar"2";"7"};{\ar|-{34}"7";"12"};{\ar"12";"17"};{\ar"17";"22"};
{\ar|-{3}"3";"8"};{\ar"8";"13"};{\ar|-{4}"13";"18"};{\ar"18";"23"};
{\ar|-{34}"4";"9"};{\ar"9";"14"};{\ar"14";"19"};{\ar"19";"24"};
{\ar|-{34}"5";"10"};{\ar"10";"15"};{\ar"15";"20"};{\ar"20";"25"};
{\ar|-{2}"3";"4"};{\ar"4";"5"};
{\ar|-{24}"8";"9"};{\ar"9";"10"};
{\ar|-{24}"13";"14"};{\ar"14";"15"};
{\ar|-{2}"18";"19"};{\ar"19";"20"};
{\ar|-{2}"23";"24"};{\ar"24";"25"};
{\ar"7";"1"};{\ar"12";"6"};{\ar"17";"11"};{\ar"22";"16"};
{\ar|-{1}"8";"2"};{\ar|-{1}"13";"7"};{\ar|-{13}"18";"12"};{\ar|-{13}"23";"17"};
\endxy
&
\ \ \ 
&
\xy
(-20,-20)*{\circ}="1";(-10,-20)*{\cdot}="2";(0,-20)*{\cdot}="3";(10,-20)*{\bullet}="4";(20,-20)*{\circ}="5";
(-20,-10)*{\cdot}="6";(-10,-10)*{\cdot}="7";(0,-10)*{\cdot}="8";(10,-10)*{\bullet}="9";(20,-10)*{\cdot}="10";
(-20,0)*{\cdot}="11";(-10,0)*{\cdot}="12";(0,0)*{\cdot}="13";(10,0)*{\bullet}="14";(20,0)*{\cdot}="15";
(-20,10)*{\cdot}="16";(-10,10)*{\cdot}="17";(0,10)*{\cdot}="18";(10,10)*{\bullet}="19";(20,10)*{\cdot}="20";
(-20,20)*{\circ}="21";(-10,20)*{\cdot}="22";(0,20)*{\cdot}="23";(10,20)*{\bullet}="24";(20,20)*{\circ}="25";
{\ar|-{34}"1";"6"};{\ar"6";"11"};{\ar"11";"16"};{\ar"16";"21"};
{\ar"2";"7"};{\ar|-{34}"7";"12"};{\ar"12";"17"};{\ar"17";"22"};
{\ar"3";"8"};{\ar"8";"13"};{\ar|-{34}"13";"18"};{\ar"18";"23"};
{\ar|-{3}"4";"9"};{\ar"9";"14"};{\ar"14";"19"};{\ar|-{4}"19";"24"};
{\ar|-{34}"5";"10"};{\ar"10";"15"};{\ar"15";"20"};{\ar"20";"25"};
{\ar|-{2}"4";"5"};
{\ar|-{24}"9";"10"};
{\ar|-{24}"14";"15"};
{\ar|-{24}"19";"20"};
{\ar|-{2}"24";"25"};
{\ar"7";"1"};{\ar"12";"6"};{\ar"17";"11"};{\ar"22";"16"};
{\ar"8";"2"};{\ar"13";"7"};{\ar"18";"12"};{\ar"23";"17"};
{\ar|-{1}"9";"3"};{\ar|-{1}"14";"8"};{\ar|-{1}"19";"13"};{\ar|-{13}"24";"18"};
\endxy
\\
Q^{b_1} &&  Q^{b_2} && Q^{b_3}\\
\\
\xy
(-20,-20)*{\circ}="1";(-10,-20)*{\cdot}="2";(0,-20)*{\cdot}="3";(10,-20)*{\bullet}="4";(20,-20)*{\circ}="5";
(-20,-10)*{\bullet}="6";(-10,-10)*{\cdot}="7";(0,-10)*{\cdot}="8";(10,-10)*{\cdot}="9";(20,-10)*{\bullet}="10";
(-20,0)*{\cdot}="11";(-10,0)*{\bullet}="12";(0,0)*{\cdot}="13";(10,0)*{\cdot}="14";(20,0)*{\cdot}="15";
(-20,10)*{\cdot}="16";(-10,10)*{\cdot}="17";(0,10)*{\bullet}="18";(10,10)*{\cdot}="19";(20,10)*{\cdot}="20";
(-20,20)*{\circ}="21";(-10,20)*{\cdot}="22";(0,20)*{\cdot}="23";(10,20)*{\bullet}="24";(20,20)*{\circ}="25";
{\ar"22";"16"};{\ar"23";"17"};{\ar"17";"11"};
{\ar|-{4}"24";"18"};{\ar"18";"12"};{\ar"12";"6"};
{\ar|-{24}"25";"19"};{\ar"19";"13"};{\ar"13";"7"};{\ar"7";"1"};
{\ar|-{24}"20";"14"};{\ar"14";"8"};{\ar"8";"2"};
{\ar|-{24}"15";"9"};{\ar"9";"3"};{\ar|-{2}"10";"4"};
{\ar"2";"7"};{\ar"8";"13"};{\ar"14";"19"};{\ar"20";"25"};
{\ar"3";"8"};{\ar"9";"14"};{\ar"15";"20"};
{\ar|-{34}"4";"9"};{\ar|-{3}"10";"15"};{\ar|-{1}"4";"5"};
{\ar"16";"21"};{\ar"11";"16"};{\ar"17";"22"};
{\ar|-{3}"6";"11"};{\ar|-{3}"12";"17"};{\ar|-{3}"18";"23"};
{\ar|-{12}"6";"7"};{\ar|-{12}"12";"13"};{\ar|-{12}"18";"19"};{\ar|-{1}"24";"25"};
\endxy
& &
\xy
(-20,-20)*{\circ}="1";(-10,-20)*{\cdot}="2";(0,-20)*{\bullet}="3";(10,-20)*{\cdot}="4";(20,-20)*{\circ}="5";
(-20,-10)*{\cdot}="6";(-10,-10)*{\cdot}="7";(0,-10)*{\cdot}="8";(10,-10)*{\bullet}="9";(20,-10)*{\cdot}="10";
(-20,0)*{\bullet}="11";(-10,0)*{\cdot}="12";(0,0)*{\cdot}="13";(10,0)*{\cdot}="14";(20,0)*{\bullet}="15";
(-20,10)*{\cdot}="16";(-10,10)*{\bullet}="17";(0,10)*{\cdot}="18";(10,10)*{\cdot}="19";(20,10)*{\cdot}="20";
(-20,20)*{\circ}="21";(-10,20)*{\cdot}="22";(0,20)*{\bullet}="23";(10,20)*{\cdot}="24";(20,20)*{\circ}="25";
{\ar"22";"16"};{\ar|-{4}"23";"17"};{\ar"17";"11"};
{\ar|-{24}"24";"18"};{\ar"18";"12"};{\ar"12";"6"};
{\ar|-{24}"25";"19"};{\ar"19";"13"};{\ar"13";"7"};{\ar"7";"1"};
{\ar|-{24}"20";"14"};{\ar"14";"8"};{\ar"8";"2"};
{\ar|-{2}"15";"9"};{\ar"9";"3"};{\ar"10";"4"};
{\ar"2";"7"};{\ar"8";"13"};{\ar"14";"19"};{\ar"20";"25"};
{\ar|-{34}"3";"8"};{\ar|-{34}"9";"14"};{\ar|-{3}"15";"20"};
{\ar"4";"5"};
{\ar"16";"21"};{\ar|-{3}"11";"16"};{\ar|-{3}"17";"22"};
{\ar"6";"7"};{\ar"12";"13"};{\ar"18";"19"};{\ar"24";"25"};
{\ar|-{12}"11";"12"};{\ar|-{12}"17";"18"};{\ar|-{1}"23";"24"};{\ar|-{1}"3";"4"};{\ar|-{1}"9";"10"};
\endxy
& &
\xy
(-20,-20)*{\circ}="1";(-10,-20)*{\bullet}="2";(0,-20)*{\cdot}="3";(10,-20)*{\cdot}="4";(20,-20)*{\circ}="5";
(-20,-10)*{\cdot}="6";(-10,-10)*{\cdot}="7";(0,-10)*{\bullet}="8";(10,-10)*{\cdot}="9";(20,-10)*{\cdot}="10";
(-20,0)*{\cdot}="11";(-10,0)*{\cdot}="12";(0,0)*{\cdot}="13";(10,0)*{\bullet}="14";(20,0)*{\cdot}="15";
(-20,10)*{\bullet}="16";(-10,10)*{\cdot}="17";(0,10)*{\cdot}="18";(10,10)*{\cdot}="19";(20,10)*{\bullet}="20";
(-20,20)*{\circ}="21";(-10,20)*{\bullet}="22";(0,20)*{\cdot}="23";(10,20)*{\cdot}="24";(20,20)*{\circ}="25";
{\ar|-{4}"22";"16"};{\ar|-{24}"23";"17"};{\ar"17";"11"};
{\ar|-{24}"24";"18"};{\ar"18";"12"};{\ar"12";"6"};
{\ar|-{24}"25";"19"};{\ar"19";"13"};{\ar"13";"7"};{\ar"7";"1"};
{\ar|-{2}"20";"14"};{\ar"14";"8"};{\ar"8";"2"};
{\ar"15";"9"};{\ar"9";"3"};{\ar"10";"4"};
{\ar|-{34}"2";"7"};{\ar|-{34}"8";"13"};{\ar|-{34}"14";"19"};{\ar|-{3}"20";"25"};
{\ar"4";"5"};
{\ar|-{3}"16";"21"};
{\ar"6";"7"};{\ar"12";"13"};{\ar"18";"19"};{\ar"24";"25"};
{\ar"11";"12"};{\ar"17";"18"};{\ar"23";"24"};{\ar"3";"4"};{\ar"9";"10"};
{\ar|-{1}"2";"3"};{\ar|-{1}"8";"9"};{\ar|-{1}"14";"15"};{\ar|-{1}"22";"23"};{\ar|-{12}"16";"17"};
\endxy
\\
Q^{c_1} & & Q^{c_2} & & Q^{c_3}\\
\\
\xy
(-20,-20)*{\circ}="1";(-10,-20)*{\cdot}="2";(0,-20)*{\cdot}="3";(10,-20)*{\cdot}="4";(20,-20)*{\circ}="5";
(-20,-10)*{\cdot}="6";(-10,-10)*{\cdot}="7";(0,-10)*{\cdot}="8";(10,-10)*{\cdot}="9";(20,-10)*{\cdot}="10";
(-20,0)*{\cdot}="11";(-10,0)*{\cdot}="12";(0,0)*{\cdot}="13";(10,0)*{\cdot}="14";(20,0)*{\cdot}="15";
(-20,10)*{\bullet}="16";(-10,10)*{\bullet}="17";(0,10)*{\bullet}="18";(10,10)*{\bullet}="19";(20,10)*{\bullet}="20";
(-20,20)*{\circ}="21";(-10,20)*{\cdot}="22";(0,20)*{\cdot}="23";(10,20)*{\cdot}="24";(20,20)*{\circ}="25";
{\ar|-{34}"1";"2"};{\ar"2";"3"};{\ar"3";"4"};{\ar"4";"5"};
{\ar"6";"7"};{\ar|-{34}"7";"8"};{\ar"8";"9"};{\ar"9";"10"};
{\ar"11";"12"};{\ar"12";"13"};{\ar|-{34}"13";"14"};{\ar"14";"15"};
{\ar|-{3}"16";"17"};{\ar"17";"18"};{\ar"18";"19"};{\ar|-{4}"19";"20"};
{\ar|-{34}"21";"22"};{\ar"22";"23"};{\ar"23";"24"};{\ar"24";"25"};
{\ar|-{1}"16";"21"};
{\ar|-{14}"17";"22"};
{\ar|-{14}"18";"23"};
{\ar|-{14}"19";"24"};
{\ar|-{1}"20";"25"};
{\ar"7";"1"};{\ar"8";"2"};{\ar"9";"3"};{\ar"10";"4"};
{\ar"12";"6"};{\ar"13";"7"};{\ar"14";"8"};{\ar"15";"9"};
{\ar|-{2}"17";"11"};{\ar|-{2}"18";"12"};{\ar|-{2}"19";"13"};{\ar|-{23}"20";"14"};
\endxy
& &
\xy
(-20,-20)*{\circ}="1";(-10,-20)*{\cdot}="2";(0,-20)*{\cdot}="3";(10,-20)*{\cdot}="4";(20,-20)*{\circ}="5";
(-20,-10)*{\cdot}="6";(-10,-10)*{\cdot}="7";(0,-10)*{\cdot}="8";(10,-10)*{\cdot}="9";(20,-10)*{\cdot}="10";
(-20,0)*{\bullet}="11";(-10,0)*{\bullet}="12";(0,0)*{\bullet}="13";(10,0)*{\bullet}="14";(20,0)*{\bullet}="15";
(-20,10)*{\cdot}="16";(-10,10)*{\cdot}="17";(0,10)*{\cdot}="18";(10,10)*{\cdot}="19";(20,10)*{\cdot}="20";
(-20,20)*{\circ}="21";(-10,20)*{\cdot}="22";(0,20)*{\cdot}="23";(10,20)*{\cdot}="24";(20,20)*{\circ}="25";
{\ar|-{34}"1";"2"};{\ar"2";"3"};{\ar"3";"4"};{\ar"4";"5"};
{\ar"6";"7"};{\ar|-{34}"7";"8"};{\ar"8";"9"};{\ar"9";"10"};
{\ar|-{3}"11";"12"};{\ar"12";"13"};{\ar|-{4}"13";"14"};{\ar"14";"15"};
{\ar|-{34}"16";"17"};{\ar"17";"18"};{\ar"18";"19"};{\ar"19";"20"};
{\ar|-{34}"21";"22"};{\ar"22";"23"};{\ar"23";"24"};{\ar"24";"25"};
{\ar|-{1}"11";"16"};{\ar"16";"21"};
{\ar|-{14}"12";"17"};{\ar"17";"22"};
{\ar|-{14}"13";"18"};{\ar"18";"23"};
{\ar|-{1}"14";"19"};{\ar"19";"24"};
{\ar|-{1}"15";"20"};{\ar"20";"25"};
{\ar"7";"1"};{\ar"8";"2"};{\ar"9";"3"};{\ar"10";"4"};
{\ar|-{2}"12";"6"};{\ar|-{2}"13";"7"};{\ar|-{23}"14";"8"};{\ar|-{23}"15";"9"};
\endxy
& & 
\xy
(-20,-20)*{\circ}="1";(-10,-20)*{\cdot}="2";(0,-20)*{\cdot}="3";(10,-20)*{\cdot}="4";(20,-20)*{\circ}="5";
(-20,-10)*{\bullet}="6";(-10,-10)*{\bullet}="7";(0,-10)*{\bullet}="8";(10,-10)*{\bullet}="9";(20,-10)*{\bullet}="10";
(-20,0)*{\cdot}="11";(-10,0)*{\cdot}="12";(0,0)*{\cdot}="13";(10,0)*{\cdot}="14";(20,0)*{\cdot}="15";
(-20,10)*{\cdot}="16";(-10,10)*{\cdot}="17";(0,10)*{\cdot}="18";(10,10)*{\cdot}="19";(20,10)*{\cdot}="20";
(-20,20)*{\circ}="21";(-10,20)*{\cdot}="22";(0,20)*{\cdot}="23";(10,20)*{\cdot}="24";(20,20)*{\circ}="25";
{\ar|-{34}"1";"2"};{\ar"2";"3"};{\ar"3";"4"};{\ar"4";"5"};
{\ar|-{3}"6";"7"};{\ar|-{4}"7";"8"};{\ar"8";"9"};{\ar"9";"10"};
{\ar|-{34}"11";"12"};{\ar"12";"13"};{\ar"13";"14"};{\ar"14";"15"};
{\ar|-{34}"16";"17"};{\ar"17";"18"};{\ar"18";"19"};{\ar"19";"20"};
{\ar|-{34}"21";"22"};{\ar"22";"23"};{\ar"23";"24"};{\ar"24";"25"};
{\ar|-{1}"6";"11"};{\ar"11";"16"};{\ar"16";"21"};
{\ar|-{4}"7";"12"};{\ar"12";"17"};{\ar"17";"22"};
{\ar|-{1}"8";"13"};{\ar"13";"18"};{\ar"18";"23"};
{\ar|-{1}"9";"14"};{\ar"14";"19"};{\ar"19";"24"};
{\ar|-{1}"10";"15"};{\ar"15";"20"};{\ar"20";"25"};
{\ar|-{2}"7";"1"};{\ar|-{23}"8";"2"};{\ar|-{23}"9";"3"};{\ar|-{23}"10";"4"};
\endxy
\\
Q^{d_1} & & Q^{d_2} & & Q^{d_3}
\end{array}$$
\caption{
For $g \in \left\{ b_j,c_j,d_j \right\}$, $Q^{g}$ supports the $\mathbb{C}^*$-family of $\mathbb{P}^1$-families of $\ell_{\mathcal{P}} =1$ almost large modules parameterized by the irreducible component of the exceptional locus corresponding to the vertex $g$ on the perimeter of the toric diagram (figure \ref{toric diagram}).
For each $1 \leq i \leq 4$, the subquiver obtained by removing all arrows from $Q^{g}$ labeled $i$ supports the $\mathbb{C}^*$- or $\mathbb{P}^1$-family of $\ell_{\mathcal{P}} = 2$ almost large modules corresponding to an edge emanating from $g$ in the toric diagram ($\mathbb{C}^*$ iff the edge is along the perimeter).}
\label{G2}
\end{figure}

\begin{figure}
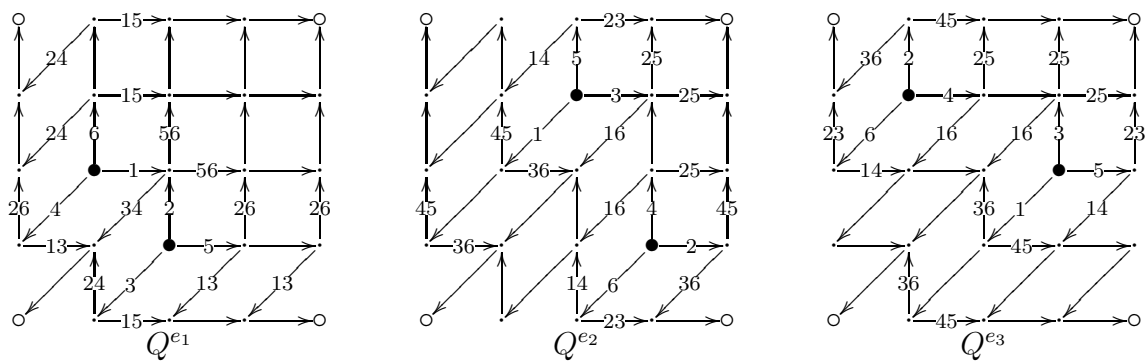

$$\begin{array}{ccccc}
\xy
(-20,-20)*{\circ}="1";(-10,-20)*{\cdot}="2";(0,-20)*{\cdot}="3";(10,-20)*{\cdot}="4";(20,-20)*{\circ}="5";
(-20,-10)*{\cdot}="6";(-10,-10)*{\cdot}="7";(0,-10)*{\bullet}="8";(10,-10)*{\cdot}="9";(20,-10)*{\cdot}="10";
(-20,0)*{\cdot}="11";(-10,0)*{\bullet}="12";(0,0)*{\cdot}="13";(10,0)*{\cdot}="14";(20,0)*{\cdot}="15";
(-20,10)*{\cdot}="16";(-10,10)*{\cdot}="17";(0,10)*{\cdot}="18";(10,10)*{\cdot}="19";(20,10)*{\cdot}="20";
(-20,20)*{\circ}="21";(-10,20)*{\cdot}="22";(0,20)*{\cdot}="23";(10,20)*{\cdot}="24";(20,20)*{\circ}="25";
{\ar|-{15}"2";"3"};{\ar"3";"4"};{\ar"4";"5"};
{\ar|-{13}"6";"7"};{\ar|-{5}"8";"9"};{\ar"9";"10"};
{\ar|-{1}"12";"13"};{\ar|-{56}"13";"14"};{\ar"14";"15"};{\ar|-{15}"17";"18"};{\ar"18";"19"};{\ar"19";"20"};
{\ar|-{15}"22";"23"};{\ar"23";"24"};{\ar"24";"25"};
{\ar|-{26}"6";"11"};{\ar"11";"16"};{\ar"16";"21"};
{\ar|-{24}"2";"7"};{\ar|-{6}"12";"17"};{\ar"17";"22"};
{\ar|-{2}"8";"13"};{\ar|-{56}"13";"18"};{\ar"18";"23"};
{\ar|-{26}"9";"14"};{\ar"14";"19"};{\ar"19";"24"};
{\ar|-{26}"10";"15"};{\ar"15";"20"};{\ar"20";"25"};
{\ar|-{24}"22";"16"};{\ar|-{24}"17";"11"};{\ar|-{4}"12";"6"};
{\ar|-{34}"13";"7"};{\ar"7";"1"};
{\ar|-{3}"8";"2"};{\ar|-{13}"9";"3"};{\ar|-{13}"10";"4"};
\endxy
&
\ \ \ 
&
\xy
(-20,-20)*{\circ}="1";(-10,-20)*{\cdot}="2";(0,-20)*{\cdot}="3";(10,-20)*{\cdot}="4";(20,-20)*{\circ}="5";
(-20,-10)*{\cdot}="6";(-10,-10)*{\cdot}="7";(0,-10)*{\cdot}="8";(10,-10)*{\bullet}="9";(20,-10)*{\cdot}="10";
(-20,0)*{\cdot}="11";(-10,0)*{\cdot}="12";(0,0)*{\cdot}="13";(10,0)*{\cdot}="14";(20,0)*{\cdot}="15";
(-20,10)*{\cdot}="16";(-10,10)*{\cdot}="17";(0,10)*{\bullet}="18";(10,10)*{\cdot}="19";(20,10)*{\cdot}="20";
(-20,20)*{\circ}="21";(-10,20)*{\cdot}="22";(0,20)*{\cdot}="23";(10,20)*{\cdot}="24";(20,20)*{\circ}="25";
{\ar"4";"5"};{\ar|-{36}"6";"7"};{\ar|-{2}"9";"10"};
{\ar|-{36}"12";"13"};{\ar|-{25}"14";"15"};{\ar|-{3}"18";"19"};{\ar|-{25}"19";"20"};
{\ar|-{23}"23";"24"};{\ar"24";"25"};
{\ar|-{45}"6";"11"};{\ar"11";"16"};{\ar"16";"21"};
{\ar"2";"7"};{\ar|-{45}"12";"17"};{\ar"17";"22"};
{\ar|-{14}"3";"8"};{\ar"8";"13"};{\ar|-{5}"18";"23"};
{\ar|-{4}"9";"14"};{\ar"14";"19"};{\ar|-{25}"19";"24"};
{\ar|-{45}"10";"15"};{\ar"15";"20"};{\ar"20";"25"};
{\ar"22";"16"};{\ar"17";"11"};{\ar"12";"6"};{\ar"7";"1"};
{\ar|-{14}"23";"17"};{\ar|-{1}"18";"12"};{\ar"13";"7"};{\ar"8";"2"};
{\ar|-{16}"19";"13"};{\ar|-{16}"14";"8"};{\ar|-{6}"9";"3"};
{\ar|-{36}"10";"4"};{\ar|-{23}"3";"4"};
\endxy
&
\ \ \ 
&
\xy
(-20,-20)*{\circ}="1";(-10,-20)*{\cdot}="2";(0,-20)*{\cdot}="3";(10,-20)*{\cdot}="4";(20,-20)*{\circ}="5";
(-20,-10)*{\cdot}="6";(-10,-10)*{\cdot}="7";(0,-10)*{\cdot}="8";(10,-10)*{\cdot}="9";(20,-10)*{\cdot}="10";
(-20,0)*{\cdot}="11";(-10,0)*{\cdot}="12";(0,0)*{\cdot}="13";(10,0)*{\bullet}="14";(20,0)*{\cdot}="15";
(-20,10)*{\cdot}="16";(-10,10)*{\bullet}="17";(0,10)*{\cdot}="18";(10,10)*{\cdot}="19";(20,10)*{\cdot}="20";
(-20,20)*{\circ}="21";(-10,20)*{\cdot}="22";(0,20)*{\cdot}="23";(10,20)*{\cdot}="24";(20,20)*{\circ}="25";
{\ar|-{45}"2";"3"};{\ar"3";"4"};{\ar"4";"5"};
{\ar"6";"7"};{\ar|-{45}"8";"9"};{\ar"9";"10"};
{\ar|-{14}"11";"12"};{\ar"12";"13"};{\ar|-{5}"14";"15"};
{\ar|-{4}"17";"18"};{\ar"18";"19"};{\ar|-{25}"19";"20"};
{\ar|-{45}"22";"23"};{\ar"23";"24"};{\ar"24";"25"};
{\ar|-{23}"11";"16"};{\ar"16";"21"};{\ar|-{36}"2";"7"};{\ar|-{2}"17";"22"};
{\ar|-{36}"8";"13"};{\ar|-{25}"18";"23"};{\ar|-{3}"14";"19"};{\ar|-{25}"19";"24"};
{\ar|-{23}"15";"20"};{\ar"20";"25"};
{\ar"7";"1"};{\ar"12";"6"};{\ar|-{6}"17";"11"};{\ar|-{36}"22";"16"};
{\ar|-{16}"18";"12"};{\ar"13";"7"};{\ar"8";"2"};{\ar|-{16}"19";"13"};{\ar|-{1}"14";"8"};{\ar"9";"3"};
{\ar|-{14}"15";"9"};{\ar"10";"4"};
\endxy\\
Q^{e_1} & & Q^{e_2} & & Q^{e_3}
\end{array}$$
\caption{
$Q^{e_j}$ supports the $\mathbb{P}^2$-family of $\ell_{\mathcal{P}} =1$ almost large modules parameterized by the irreducible component of the exceptional locus corresponding to the vertex $e_j$ in the toric diagram (figure \ref{toric diagram}).
For each $1 \leq i \leq 6$, the subquiver obtained by removing all arrows from $Q^{e_j}$ labeled $i$ supports the $\mathbb{P}^1$-family of $\ell_{\mathcal{P}} = 2$ almost large modules corresponding to an edge emanating from $e_j$ in the toric diagram.}
\label{G3}
\end{figure}

\begin{Remark} \rm{
This example shows that the irreducible components of the exceptional locus need not shrink to the annihilator of a vertex simple module: Each $\mathbb{P}^1$-family supported on a subquiver in figure \ref{G2} shrinks to the annihilator of a simple module supported on a subquiver with vertex set given by the bold vertices in the figure.  Such a point in $\operatorname{Max}A$, which we view as a point-like sphere, sits over a point of $\operatorname{Max}R$ with one non-vanishing coordinate ($x$, $y$, or $z$).  Furthermore, each $\mathbb{P}^2$-family supported on a subquiver in figure \ref{G3} collapses to two points in $\operatorname{Max}A$, namely the annihilators of the two vertex simples at the bold vertices in the figure.  Both of these points sit over the origin of $\operatorname{Max}R$.
} \end{Remark}

\begin{Remark} \rm{
This example and the conifold quiver algebra from section \ref{The conifold} are examples of square superpotential algebras (see \cite[Definition 1.1]{B}).  The supporting subquivers for the $\ell_{\mathcal{P}} = 1$ (resp.\ $\ell_{\mathcal{P}} = 2$; $\ell_{\mathcal{P}} = 3$) large and almost large modules coincide with the subquivers obtained by removing all the arrows from $Q$ that occur in a so called perfect matching (resp.\ the intersection of two perfect matchings; the intersection of three perfect matchings).  In this sense perfect matchings may be viewed as a special case of almost large modules over a particular class of quiver algebras whose centers are toric Gorenstein singularities, and whose relations are derived from a potential.  This observation will be addressed in a forthcoming paper, \cite{B2}.
} \end{Remark}

\newpage
\bibliographystyle{hep}
\def\cprime{$'$} \def\cprime{$'$}

\end{document}